\documentclass{amsart}

\usepackage[utf8]{inputenc}
\usepackage{amsmath,amscd,amssymb,amsthm,amsfonts}
\usepackage{enumerate}
\usepackage{array}
\usepackage[]{graphicx}
\usepackage[all]{xy}
\usepackage{tensor}
\usepackage{tikz-cd}
\usepackage{hyperref}

\usepackage{xcolor}
\definecolor{red}{rgb}{0.8,0,0.1}
\definecolor{blue}{rgb}{0.1,0,0.8}

\newtheorem{theorem}{Theorem}[section]
\newtheorem{proposition}[theorem]{Proposition}
\newtheorem{corollary}[theorem]{Corollary}
\newtheorem{lemma}[theorem]{Lemma}

\newtheorem*{theorem*}{Theorem}

\theoremstyle{definition}
\newtheorem{definition}[theorem]{Definition}
\newtheorem{example}[theorem]{Example}

\theoremstyle{remark}
\newtheorem{remark}[theorem]{Remark}

\numberwithin{equation}{section}

\def\C{{\mathcal C}}

\def\G{{\mathcal G}}

\def\id{\mathrm{id}}

\def\hcdot{\hat{\cdot}}

\newcommand{\cont}{\mathbin{\raisebox{\depth}{\scalebox{1}[-1]{$\lnot$}}}}

\def\to{\rightarrow}
\def\tto{\rightrightarrows}

\def\to{\rightarrow}
\def\tto{\rightrightarrows}

\newcommand{\hor}[0]{\mathrm{hor}}

\newcommand{\conn}[0]{\mathfrak{c}}
\newcommand{\Cbas}[0]{\mathrm{K}_\mathrm{bas}}
\newcommand{\Curv}[0]{\mathrm{K}}
\newcommand{\cbas}[0]{\mathrm{R}_\mathrm{bas}}
\newcommand{\curv}[0]{\mathrm{R}}
\newcommand{\mc}[0]{\mathfrak{mc}}

\newcommand{\pr}[0]{\mathrm{pr}}

\newcommand{\bis}[0]{\mathfrak{b}}

\newcommand{\an}[1]{\mathfrak{an}#1}

\usepackage[backend=bibtex, style=numeric-comp]{biblatex} 
\begin{document}

\title[Chern-Weil theory for groupoids with Cartan connections]{A Chern-Weil theory for Lie groupoids with Cartan connections}
\author[L. Accornero]{Luca Accornero$^1$}
\address{$^{1,3}$ Instituto de Matemática e Estatística - Universidade de São Paulo, Rua do Matão, 1010 - Butantã, São Paulo - SP, 05508-090, Brazil}
\email{$^{1}$luca.accornero@ime.usp.br, $^2$mateus.melo@ufes.br, 
$^3$ivanstru@ime.usp.br}
\author[M. de Melo]{Mateus M. de Melo$^2$}
\address{$^{2}$ Universidade Federal do Espírito Santo, Departamento de Matem\'atica, Avenida Fernando Ferrari 514 - Vitória, Espírito Santo - 29075-910, Brazil}
\author[I. Struchiner]{Ivan Struchiner$^3$}
\thanks{L.A. was supported by the São Paulo Research Foundation (FAPESP), Brasil, Process Number 2024/22841-1. M.dM. acknowledges FAPESP for supporting his research through the grant 2019/14777-3. I.S. acknowledges FAPESP for supporting his research through the grant 2022/16310-8.}
\setcounter{secnumdepth}{2}
\setcounter{tocdepth}{2}

\begin{abstract}
We extend Chern-Weil theory to Lie groupoids carrying a Cartan connection.
For a principal $(\Sigma,\omega)$-bundle
$(P,\theta)$ over $X$ we construct a map of differential graded algebras from an
$\omega$-basic subcomplex of the Weil algebra of the Lie algebroid of $\Sigma$ to
$\Omega^\bullet(X)$, whose map in cohomology does not depend on $\theta$. It
recovers the classical Chern-Weil map for Lie groups, the Cartan model for
action groupoids, and the characteristic classes of Cartan geometries for gauge groupoids. 
\end{abstract}

\maketitle
\tableofcontents

\section{Introduction}\label{sec:intro}

Classical Chern-Weil theory attaches characteristic classes to a principal
$G$-bundle by evaluating invariant polynomials on the curvature of a
connection. The present paper is the first step of a project to extend Chern-Weil theory to principal groupoid bundles.

Our motivation to attempt such 
an extension comes first of all from the theory of geometric structures on manifolds, in Haefliger's sense: a geometry on a manifold is a 
principal bundle whose structure groupoid is the germ groupoid of a Lie pseudogroup \cite{Haefliger}, and its
characteristic classes are the characteristic classes of
$\Gamma$-foliations \cite{BottHaefliger}. A Chern-Weil oriented description of such classes was given in \cite{AccCr2023}; on the other hand, Haefliger's 
principal bundles only cover the so-called \textit{integrable} geometries -- geometric structures possessing a local normal form -- and one has to look at more general 
principal objects to describe geometries with local invariants. 

A second motivation to extend Chern-Weil theory from Lie groups to Lie groupoids comes from its well known relation to the Cartan model for the equivariant cohomology
of the action of a Lie group $G$ on a manifold $M$. Such model can be understood as an infinitesimal description of the cohomology of the Bott-Shulman complex of the action groupoid $G\ltimes M$.
The search of an analogue of such a model for more general Lie groupoids has been going on for a while but, despite recent remarkable progress in the proper regular case \cite{KraaschTarnowsky2026BottSpectralSequence}, 
it has not yet reached a conclusion. In the context of this paper, the key observation is that the infinitesimal description of the cohomology of the Bott-Shulman complex of a Lie groupoid is expected to be the 
domain of definition of the Chern-Weil characteristic map, so progress with constructing the former is progress with constructing the latter and viceversa.

This already suggests that Chern-Weil theory for Lie groupoids is far from being an obvious construction. Indeed, when attempting to extend the Chern-Weil construction to principal \textit{groupoid} bundles one 
encounters several difficulties. The classical construction relies on the adjoint representation and on the equivariance of a principal connection
one-form with respect to such representation. When dealing with Lie groupoids, not only the adjoint representation is only a 
representation up to homotopy (hence, not associative), but, if $\Sigma\tto M$ is a Lie groupoid and
$P$ is a principal $\Sigma$-bundle, there is no natural action of $\Sigma$ on $TP$. In other words, the notion of principal connection is not clear.

A choice of \textbf{Cartan connection} on $\Sigma$ removes such difficulties. Such a connection is
a distribution $\ker(\omega)\subset T\Sigma$ complementary to the source
fibres and closed under multiplication, equivalently a one-form
$\omega\in\Omega^1(\Sigma,t^*A)$ restricting to the Maurer-Cartan form along
$\ker(ds)$ and satisfying a multiplicativity equation. Such a distribution
determines a representation of $\Sigma$ on the complex $A\overset{\rho}{\to}TM$,
an action on $TP$ for every $\Sigma$-space $P$, and hence a notion of
principal connection. The infinitesimal counterpart is a connection $D$ on the
vector bundle $A$ with vanishing basic curvature; pairs $(A,D)$ were studied e.g. by Blaom under the name of
\emph{Cartan algebroids}~\cite{Blaom2006}; see also~\cite{Blaom2016} for the groupoid picture.

Thus, our paper deals with pairs $(\Sigma,\omega)$ consisting of a Lie groupoid together with a Cartan connection, and develops 
Chern-Weil theory for principal $\Sigma$-bundles. Before presenting our work, let us add that the need for choosing a Cartan connection
has several drawbacks. First, Cartan connections do not always exist -- one immediately notices that a Cartan connection
provides a parallelization of the source-fibers of the groupoid (in fact, a Cartan connection in the sense \cite{Sharpe1997}). Second, even when
the groupoid supports a Cartan connection, the resulting Chern-Weil construction does depend on the Cartan connection chosen. Despite such drawbacks,
our construction applies nicely to a number of natural examples, some of which support a natural choice of Cartan connection.

The examples that we discuss are the following ones.
\begin{itemize}
    \item Lie groups with their Maurer-Cartan form. Our Chern-Weil construction coincides with the classical one.
    \item Bundles of Lie groups. Here, Cartan connections interact nicely with the structure of the bundles. Indeed, a locally trivial bundle of Lie groups carries a Cartan connection,
    and the converse holds true if the total space is assumed to be Hausdorff and the fibers are assumed to be connected. The resulting Chern-Weil theory is a refinement of the fiberwise version of the classical theory. 
    \item Action groupoids. Here, the Cartan connection is the pullback of the Maurer-Cartan form of a Lie group. As pointed out above, our Chern-Weil map is defined on the Cartan model of equivariant cohomology.
    \item Gauge groupoids of principal bundles. Here, a choice of Cartan connection is equivalent to a choice of Cartan connection on the principal bundle (with some subtleties regarding the underlying model geometry) \cite{Cattafi2021}; we discussed the resulting Chern-Weil theory separately in \cite{AccDeMelStr25}, though the map provided here is a refinement.
    \item Infinite jet groupoids of Lie pseudogroup. The relevant principal bundles here include Haefliger's $\Gamma$-foliations \cite{Haefliger}, \cite{BottHaefliger}, and the natural principal connections are flat, leading to the secondary invariant that recover the characteristic classes of $\Gamma$-foliations \cite{BottHaefliger} discussed in \cite{AccCr2023}.
\end{itemize}

\subsection*{The main results} Let $\Sigma\tto M$ be a Lie groupoid with Lie algebroid $A\to M$, and $\pi:P\to X$ be a principal $\Sigma$-bundle with moment map $\mu:P\to M$. 

Let $D$ be a connection on $A$. The \textbf{Weil algebra} of $(A,D)$, see \cite{AbCr2011RepUpHomAlg}, is the free graded $C^\infty(M)$-algebra
\[
W^\bullet(A,D) = \Gamma(\Lambda^\bullet T^*M \otimes S^\bullet A^*\otimes \Lambda^\bullet A^*)
\]
generated by $\Omega^1(M)$ and $\Omega^1(A)$ in degree $1$ and $\Gamma(S^1A^*)$ in degree two. It is equipped with a differential $d_W$ built from the connection $D$, making it into a differential graded algebra. A connection independent description of this object exists, see \cite{AbCr2009Weil}, but throughout the paper we need and use the description from \cite{AbCr2011RepUpHomAlg}. 

A connection one form on $P$ is a pointwise surjective one-form $\theta\in \Omega^1(P,\mu^*A)$ whose kernel defines an Ehresmann connection on $\pi:P\to X$ -- i.e. $TP=\ker(\theta)\oplus T^\pi P$. Let $\omega\in \Omega^1(\Sigma,t^*A)$ be a Cartan connection on $\Sigma$; $\omega$ induces an action of $\Sigma$ on $TP$ and a representation of $\Sigma$ on $A$ and $TM$ equivariant with respect to the anchor map. We define a  \textbf{principal $(\Sigma,\omega)$-connection one form} as a connection one-form $\theta\in \Omega^1(P,\mu^*A)$ that is equivariant with respect to the action of $\Sigma$ on $TP$ and $\mu^*A$ induced by $\omega$. Proposition \ref{prp:princ_conn_exist} shows that principal $(\Sigma,\omega)$-connections always exist.

The Cartan connection $\omega$ induces canonically a connection $D$ on $A$. We can then state our first main result.
\begin{theorem*}[Theorem~\ref{thm:conn_curv}]
    A connection $\theta$ on $P$ induces canonically a map of differential graded algebras
    \[
    \kappa^{\Theta}:\left(W^\bullet(A,D),d_W\right)\longrightarrow
    \left(\Omega^\bullet(P),d_\mathrm{dR}\right)
    \]    
    If $\omega$ is a Cartan connection on
    $\Sigma$, $D$ the corresponding connection on $A$, and $\theta$ is a principal
    $(\Sigma,\omega)$-connection, then $\ker(\kappa^{\Theta})$ contains the
    differential ideal generated by $\Gamma(S^1A^*)$ if and only if $\theta$ is
    flat.
\end{theorem*}
This statement is the analogue of the classical fact that a connection on a principal group bundle induces an algebra map $\Lambda^\bullet \mathfrak{g}^*\to \Omega^\bullet P$, that this map lifts to the Weil algebra as a map of DGAs, and that the lift commutes with the differential precisely when the connection is flat. 

Our second result involves an action of $\Sigma$ on $W^\bullet(A,D)$ determined by $\omega$. This provides $W^\bullet(A,D)$ with an $\omega$-basic subcomplex and the map from the previous theorem restricts to it.
\begin{theorem*}[Theorem~\ref{thm:map_on_basics}, Definition~\ref{def:main_CW_map} and
Proposition~\ref{prop:CW_indep_of_theta}]
    Let $(\Sigma,\omega)$ be a Lie groupoid with Cartan connection,
    and let $(P,\theta)$ be a principal
    $(\Sigma,\omega)$-bundle over $X$. Then $\kappa^{\Theta}$ restricts to a
    map of differential graded algebras
    \[
    \kappa^{\Theta}:\left(W^\bullet(A,D)_{\omega\mathrm{-bas}},d_W\right)
    \longrightarrow \left(\Omega^\bullet(X),d_\mathrm{dR}\right),
    \]
    and the induced map in cohomology
    \[
    \kappa^{\mathrm{CW}}: H^\bullet\!\left(W^\bullet(A,D)_{\omega\mathrm{-bas}}\right)
    \longrightarrow H^\bullet_\mathrm{dR}(X)
    \]
    does not depend on the choice of $\theta$.
\end{theorem*}
The characteristic map $\kappa^{\mathrm{CW}}$ is natural under pullback and invariant under isomorphisms of principal bundles (Lemmata \ref{lem:CW_natural} and \ref{lem:CW_iso_invariant}).

\subsection*{Related work}

Connections on Lie groupoids go back to~\cite{Blaom2016, Blaom2016Pseudoactions}. 

The correspondence between transitive
groupoids with Cartan connections and Cartan geometries
was worked out first~\cite{Cattafi2021}. A Chern-Weil theory for Cartan connections was developed first in~\cite{AleMichCartan}; we provided a generalization in \cite{AccDeMelStr25}. The construction presented here is different; we recover the one from \cite{AccDeMelStr25} in Subsection~\ref{subs:CartanConnections}.

The case of Lie groupoids equipped with \textit{flat} Cartan connections was discussed first in \cite{AccCr2023}.


Regarding the relation of our work with producing an infinitesimal model for the cohomology of the Bott-Shulman complex of a Lie groupoid, it is expected, at least under some compactness assumptions, for such a model to appear as the cohomology of a suitable subcomplex of the Weil algebra encoding basic cochains with respect to some kind of generalized action of $\Sigma$; see the introduction of \cite{AbCr2009Weil}, where it is also observed that this ``does not seem to work in an obvious way''. We stress that \emph{we make no claim} that our $H^\bullet\!\left(W^\bullet(A,D)_{\omega\mathrm{-bas}}\right)$ provides such a model. In fact, in general $H^\bullet\!\left(W^\bullet(A,D)_{\omega\mathrm{-bas}}\right)$ depends on $\omega$, and hence it cannot compute $H^\bullet(B\Sigma)$. For Lie groups or action groupoids, the Cartan connection is canonical, and $H^\bullet\!\left(W^\bullet(A,D)_{\omega\mathrm{-bas}}\right)$ reduces respectively to $S^\bullet_G(\mathfrak{g}^*)$ and the Cartan model of equivariant cohomology which, under compactness assumptions, do compute the cohomology of the classifying space of $G$ and $G\ltimes M$.

\subsection*{Structure of the paper}
We assume the reader to be familiar with the basics on Lie groupoids and Lie algebroids and their representations.

In Section~\ref{sec:groupoids_with_connections} we review the theory of connections on Lie groupoids, including multiplicativity, curvature and Maurer-Cartan expression, and the infinitesimal picture; for the sake of completeness, the discussion is extensive, but we refer to the literature for several proofs. Section~\ref{sec:principal_bundles}
discusses principal $\Sigma$-bundles and introduces principal $(\Sigma,\omega)$-connections. Section~\ref{sec:weil} introduces and discusses the almost Lie algebroid $TM\bowtie A$ and the Weil algebra for a
possibly non-Cartan connection $D$ on $A$. Section~\ref{sec:construction} contains the main results discussed above. Finally, Section~\ref{sec:examples} works out
some examples.

\section{Groupoids with Cartan connections}\label{sec:groupoids_with_connections}

We assume throughout a working knowledge of Lie groupoids and Lie algebroids,
including actions and representations.

\subsection{Connections on a Lie groupoid}\label{subs:connections_on_groupoids}
Let $\Sigma\rightrightarrows M$ be a Lie groupoid over $M$. We define its \textbf{first jet prolongation} as the subspace of the Grassmanian $\mathrm{Gr}(T\Sigma)$ of $\dim(M)$-dimensional vector subspaces in $T_g\Sigma$, $g\in \Sigma$, which are horizontal with respect to both $s$ and $t$:
\[
J^1\Sigma :=\{\hor_g\subset T_g\Sigma:\ ds \text{ and } dt \text{ restrict to } \hor_g \text{ as isomorphisms.}\}.
\]
Notice that the Lie groupoid structure of $\Sigma$ induces a canonical Lie groupoid structure on $J^1\Sigma$ that makes $J^1\Sigma \to \Sigma$ into a Lie groupoid map: for $\hor_g\in J^1\Sigma$ one sets
\[
s(\hor_g) = s(g),\quad t(\hor_g) = t(g)
\]
and, for $(g,h)\in \Sigma^{(2)}$,
\[
\begin{aligned}
\hor_g\cdot \hor_h &= dm(\hor_g,\hor_h)\\
&= \bigl\{\, dm_{g,h}(v,w)\ :\ v\in \hor_g,\ w\in \hor_h,\ ds(v)=dt(w) \,\bigr\}.
\end{aligned}
\]
Notice that the conditions $v\in \hor_g,\ ds(v)=dt(w)$ completely determine $v$ once $w$ is given, hence $\hor_g\cdot \hor_h \in J^1\Sigma$. The units in $J^1\Sigma$ are the tangent spaces to the image of the unit map $u:M\to \Sigma$.

A more common definition for the first jet prolongation is 
\[
{J^1\Sigma} :=\{j^1_x\bis:\ \bis\in \mathrm{Bis}_\mathrm{loc}(\Sigma) \}
\]
where we recall that a local bisection of $\Sigma$ is a section $\sigma:U\to \Sigma$ of the source map, defined on an open subset $U\subset M$, such that $t\circ \sigma:U\to t(\sigma(U))$ is a diffeomorphism onto an open subset of $M$. 

The two definitions are equivalent: any first jet of a bisection $j^1_x\bis$ induces the subspace $\mathrm{Im}(d_x\bis) \subset T\Sigma$, while any subspace $\hor_g\subset T_g\Sigma$ complementary to both $ds$ and $dt$ is tangent to a local bisection $\bis$. 

The description via bisection has the advantage to make it clear that $J^1\Sigma$ is canonically a Lie groupoid over $X$, and that the natural map $J^1\Sigma\to \Sigma$ is an fiber bundle. On the other hand, the description via subspaces has the advantage of being more symmetric and makes the definition below especially transparent. 

\begin{definition}\label{def:connection_gpd}
    Let $\Sigma\rightrightarrows M$ be a Lie groupoid over $M$. A \textbf{connection} on $\Sigma$ is a smooth section
    \[
    \conn:\Sigma \to J^1\Sigma
    \]
    of the canonical projection. We say that $\conn$ is \textbf{multiplicative} when it is also a groupoid map.
    Moreover, we call  $\conn$ a \textbf{flat connection} when there exists a foliation $\mathcal{F}$ on $\Sigma$ such that
    \[
    \conn(g) = T_g\mathcal{F} \in J^1\Sigma
    \]
\end{definition}
Observe that $\conn$ consists of a collection of subspaces $\hor_g\subset T_g\Sigma$, $\hor_g\in J^1\Sigma$, which is smooth in $g$. It follows that a connection on $\Sigma$ is both an Ehresmann connection for the source map $s:\Sigma\to M$ and for the target map $t:\Sigma\to M$. 

A connection $\conn$ on $\Sigma$ is flat if and only if it is flat as an Ehresmann connection (with respect to either projection). 

Notice also that, if $\conn$ is a Cartan connection, then the unit manifold $M$ is tangent to $\conn$ in the sense that $\conn(1_x) = du (T_{x} M)$, where $u:M\to \Sigma$ is the unit map.

While very transparent, the notion of connection on a groupoid that we just introduced is not very practical. Different but equivalent descriptions can be obtained straight away from the various descriptions of Ehresmann connections for submersions $p:Y\to X$, applied to the maps $s:\Sigma\to M$ and $t:\Sigma \to M$. 
\begin{lemma}[{see~\cite[\S 1.1]{Blaom2016}; see also~\cite[Definition 4.1 and Remark 4.2]{AccCr2023}}]\label{lem:various_conn_pov}
	Let $\Sigma\tto M$ be a Lie groupoid. A distribution $\C\subset T\Sigma$ complementary to the source can be equivalently encoded as a
	\begin{itemize}
		\item a splitting of the short exact sequence
		\[
		0\to T^s\Sigma\to T\Sigma\to s^*TM \to 0
		\]
		i.e. one of the following:
		\begin{itemize}
        \item a right splitting, i.e. a \textbf{horizontal lift} 
			\[
			\hor:s^*TM \to T\Sigma,\quad (g,v) \to \hor^\C_g(v);
            \]
			\item a left splitting, i.e. a \textbf{vertical projection} 
			\[
			\mathrm{pr}^s: T\Sigma\to T^s\Sigma,\quad V_g \to pr^s(V_g);
			\]
			
		\end{itemize}
		\item a \textbf{connection one-form} $\omega\in \Omega^1(\Sigma,t^*A)$ such that 
		\[
		\omega|_{\ker(ds)} = \omega_{\mc}
		\]
		where 
		\[
		\omega_{\mc}: T^s\Sigma \to t^*A,\quad v_g \to dR_{g^{-1}}(v_g)
		\]
		is the Maurer-Cartan form of $\Sigma\tto M$.
	\end{itemize}
A similar statement holds when replacing $s$ with $t$.
\end{lemma}

Let us point out very briefly some of the relationship between the various point of view. Given $\C$, one obtains the horizontal lift by sending $v\in T_xM$ to the only vector
\[
V_g \in \C_g:\ ds(V_g) = v_x.
\]
This operation is called the \textbf{horizontal lift} associated to $\C$ and denoted by $\hor^\C$.
On the other hand vertical projection is defined by sending $w_g\in T_g\Sigma$ to the difference
\[
w_g-\hor^\C_g(ds(w_g))\in \ker_g(ds).
\]
As for the form $\omega\in \Omega^1(\Sigma,t^*A)$. it is given by the vertical projection composed with the Maurer-Cartan form $\omega_{\mc}:\ker(ds)\cong t^*A$, defined by
\[
\omega_{MC}(V_g)=dR_{g^{-1}}\cdot V_g \in A_{t(g)}.
\]
Finally, given a form $\omega\in \Omega^1(\Sigma,t^*A)$ that acts as $\omega_{\mc}$ on $s$-vertical vectors, the distribution $\C:=\ker(\omega)$ is complementary to $\ker(ds)$.

In what follows, we will favor the point of view of forms, and adapt the notation accordingly. For instance, a groupoid with connection will be denoted by $(\Sigma,\omega)$, while the horizontal lift will be denoted by $\hor^\omega$. The reader should keep in mind that the section 
$\conn^\omega:\Sigma\to J^1\Sigma$ of $J^1\Sigma\to \Sigma$ associated to the form $\omega$ is given by 
\[
g\mapsto \conn^\omega(g) = \ker_g(\omega)\subset T_g\Sigma,
\]
with the understanding that $\ker(\omega)$ is complementary to both $T^s\Sigma$ and $T^t\Sigma$.

\subsection{Curvature and basic curvatures}\label{subs:curvature_gpd}
A connection on $\Sigma$ induces the two ''quasi representations'', discussed in detail in~\cite{AbCr2009RepUpHomGpd},
\begin{align*}
    \Sigma\times_M TM\to TM,\quad (g,v) &\to (dt\circ \hor^\omega_g)(v)\\
    \Sigma\times_M A\to A,\quad (g,\alpha) &\to dR_{g^{-1}}dm((\hor^\omega_g\circ \rho)(\alpha),\alpha)
\end{align*}
where the name is motivated by the fact that both maps may fail to satisfy associativity with respect to the multiplication in $\Sigma$, and become representations of $\Sigma$ on $TM$ and $A$ respectively precisely when they are associative.

Notice that $\rho(g\cdot \alpha) = g\cdot \rho(\alpha)$ holds. In other words, there is a quasi representation on the complex
\[
A\overset{\rho}{\to} TM.
\]
\begin{definition}\label{def:adjoint_rep_gpd}
The representation of $\Sigma$ on $\rho:A\to TM$ induced by a connection $\omega$ is called the \textbf{adjoint quasi representation} associated to $\omega$.
\end{definition}
It turns out that the associativity of the adjoint quasi representation is closely related to the multiplicativity of $\omega$ is multiplicative.
\begin{lemma}[{see~\cite[\S 5.2]{AccCr2023}}]\label{lem:quasi_rep_identity}
    Let $(\Sigma,\omega)$ be a Lie groupoid with connection over $M$. For all $v\in T_xM$ and $(g,h)\in \Sigma^{(2)}$ composable arrows such that $s(h) = x$, 
    \[
    \hor^\omega_{gh}(v)-dm_{g,h}(\hor^\omega_g(h\cdot v),\hor^\omega_h(v))\in \ker_{gh}(ds)
    \]
    and $\ker_{gh}(ds)\cong A_{t(g)}$ using the right multiplication by $(gh)^{-1}$.
\end{lemma}
\begin{definition}\label{def:basic_curvature}
Let $(\Sigma,\omega)$ be a groupoid with connection. The \textbf{basic curvature} of $\omega$ is the section
\[
\Cbas^\omega\in \Gamma(\Sigma^{(2)}, s^*T^*M\otimes t^*A)
\]
defined by
\begin{equation}\label{eq:basic_curvature}
\Cbas^\omega(g,h) = \hor^\omega_{gh}(\underline{\ })-dm_{g,h}(\hor^\omega_g(h\cdot \underline{\ })\hor^\omega_h(\underline{\ }))\in {T^*_{s(h)}M}\otimes A_{t(g)}
\end{equation}
\end{definition}
\begin{lemma}[{\cite[\S 1.1]{Blaom2016}, \cite[Definition 4.1]{AccCr2023}}]\label{lem:multipl_equiv_notion}
    Let $\omega\in \Omega^1(\Sigma,t^*A)$ be a connection on $\Sigma$ corresponding to the section $\conn:\Sigma\to J^1\Sigma$. The following are equivalent:
	\begin{itemize}
    \item $\omega$ is multiplicative in the sense of Definition \ref{def:connection_gpd}
    \item $\Cbas^\omega = 0$
		\item $\ker(\omega)\subset T\Sigma$ is a subgroupoid of the tangent groupoid, i.e.
                \[
                \ker(\omega)\cdot \ker(\omega) = \ker(\omega),\quad d \mathrm{i}(\ker(\omega))=\ker(\omega); 
                \]
	\item the \textbf{multiplicativity equation}
		\[
		(m^*\omega)_{(g,h)} = (pr_1^*\omega)_{(g,h)}+g\cdot (pr_2^*\omega)_{(g,h)}
		\]
		holds true at all compatible pairs $(g,h)\in \Sigma\times_M\Sigma$.
	\end{itemize}
\end{lemma}
\begin{proof}
    The first equivalence is almost tautological. Indeed, recall that the connection is completely encoded into the horizontal lift of $\omega$, which is equivalently described as the section
    \[
    \conn^\omega:\Sigma \to J^1\Sigma, g\to (s(g),g,\hor_g^\omega)
    \]
    of the canonical projection $J^1\Sigma\to \Sigma$. It is immediate to see that~\eqref{eq:basic_curvature} vanishes for all $(g,h)\in \Sigma^{(2)}$ if and only if $\conn^\omega$ is a groupoid morphism; since
    \[
    \ker_g(\omega) = \conn^\omega(g) \subset T_g\Sigma, \quad g \in \Sigma,
    \]
    the first point follows.

    We now focus on the equivalence between the vanishing of the basic curvature and the multiplicativity equation. Notice that, for all $(g,h)\in \Sigma^{(2)}$ and $(v_g,w_h)\in T\Sigma^{(2)}$, if we use the decompositions 
    \[v_g = v^\omega_g+v^s_g,\quad w_h = w^\omega_h+w^s_h\]
    where $v^\omega_g$ and $w^\omega_h$ are tangent to $\ker(\omega)$ while $v^s_g$ and $w_h^s$ are $s$-vertical, we have
    \[
    \begin{aligned}
    \omega\circ dm (v_g,w_h) ={}& \omega \circ dm (v^\omega_g-\hor^\omega_g(dt(w^s_h)),w^\omega_h)\\
    &+ \omega \circ dm (v^s_g,0) + \omega \circ dm (\hor^\omega_g(dt(w^s_h)),w^s_h).
    \end{aligned}
    \]
    By definition of basic curvature 
    \[
    \omega \circ dm (v^\omega_g-\hor^\omega_g(dt(w^s_h)),w^\omega_h) = -\Cbas^\omega(g,h)(ds(w^\omega_h));
    \]
    on the other hand
    \[
    \omega \circ dm (v^s_g,0) = \omega(dR_{h}(v^s_g)),
    \]
    while
    \[
    \omega \circ dm (\hor^\omega_g(dt(w^s_h)),w^s_h) = \omega(dR_g(g\cdot (dR_{h^{-1}} w^s_h))).
    \]
    Now recall that $\omega$ restricts to $\ker(ds)$ as the Maurer-Cartan form, see Lemma~\ref{lem:various_conn_pov}; in particular, it is right invariant and restricts to $\ker_M(ds)$ as the identity map. This implies
    \[
    \omega(dR_{h}(v^s_g))=\omega(v_g),\quad \omega(dR_g(g\cdot (dR_{h^{-1}} w^s_h))) = g\cdot\omega(w_h).
    \]
    Explicitly, for the first of these, $\omega$ at the point $gh$ is
    $dR_{(gh)^{-1}}$ on $\ker(ds)$, so
    \[
        \omega\bigl(dR_h(v^s_g)\bigr)=dR_{(gh)^{-1}}dR_h(v^s_g)
        =dR_{g^{-1}}(v^s_g)=\omega(v^s_g)=\omega(v_g),
    \]
    the last equality because $\omega$ vanishes on $\ker(\omega)$.
    Hence, we have
    \begin{equation}\label{eq:basic_curv_differential_omega}
        \delta(\omega) = -\Cbas^\omega(g,h)(ds(w^\omega_h))
    \end{equation}    
    and we see that the multiplicativity equation holds if and only if 
    \[
    \Cbas^\omega(g,h)(ds(w^\omega_h)) = 0.
    \]    

   Finally, we we show equivalence between the last two bullet points. Notice that for all $(g,h)\in \Sigma^{(2)}$ and $(v_g,w_h)\in T\Sigma^{(2)}$ and $v_g\in \ker_g(\omega)$, $w_h\in \ker_h(\omega)$ the multiplicativity equation reads as
    \[
    \omega (dm(v_g,w_h)) = 0.
    \]
    Consequently, if the multiplicativity equation is satisfied then $\ker(\omega)\cdot \ker(\omega) = \ker(\omega)$; similarly, one sees that the multiplicativity equation implies $di(\ker(\omega)) = \ker(\omega)$. On the other hand, $\ker(\omega)\cdot \ker(\omega) = \ker(\omega)$ implies immediately that the basic curvature of $\omega$ vanishes.
\end{proof}

Finally, as promised, we have the following.
\begin{lemma}[{see~\cite[Proposition 2.15]{AbCr2009RepUpHomGpd}}]\label{lem:multiplicativity_properties}
	Let $\omega\in\Omega^1(\Sigma,t^*A)$ satisfy
	$\omega|_{\ker(ds)}=\omega_\mathrm{MC}$ and $\ker(\omega)_{1_x}=T_xM$ for every
	$x\in M$. Then, for all $(g,h)\in \Sigma^{(2)}$, $\alpha\in A$, and $v\in TM$,
  \[
  g\cdot(h\cdot\alpha)-(gh)\cdot\alpha = -\Cbas^\omega(g,h)(\rho(\alpha))
  \]
  and
  \[
  g\cdot(h\cdot v)-(gh)\cdot v = -\rho(\Cbas^\omega(g,h)(v))
  \]
  hold true. In particular, if $\omega$ is a Cartan connection then
  $\Cbas^\omega=0$ and the quasi-representation induced by $\omega$ on
  $A\overset{\rho}{\to}TM$ is a genuine representation.
\end{lemma}

Apart from the basic curvature of a connection on a groupoid, one also has the usual curvature of an Ehresmann connection.

\begin{definition}\label{def:curvature}
	The curvature of connection on a groupoid $\Sigma$ is the bundle map
 \[
 \Curv: s^*\Lambda^2TM \to t^*A,\quad (g,(v,w))\to -\omega([\mathrm{hor}^\omega_g(v),\mathrm{hor}^\omega_g(w)]).
 \]
\end{definition}
In form terms, the curvature of $\omega$ correspond to the restriction of the "differential of $\omega$" to $\ker(\omega)$; while "the differential of $\omega$" depends on the choice of a connection on the coefficient space $t^*A$, its restriction to $\ker(\omega)$, hence the curvature, is well defined. 


In distribution terms, the curvature of an $s$-Ehresmann connection vanishes when $\ker(\omega)$ is tangent to a foliation.
\begin{proposition}[{\cite[Corollary 1.4 and Proposition 5.2]{Blaom2016}, \cite[Theorem 2.7]{Blaom2016Pseudoactions}}]\label{prop:flat_Cartan_action_global}
    Let $M$ be a simply connected manifold, and $\Sigma\tto M$ be a Hausdorff Lie groupoid whose source fibres are simply connected. Let $\omega$ be a flat Cartan connection on $\Sigma$ whose induced connection $D$ is \textbf{complete}, meaning that the $A$-paths $t\mapsto X_t$ with $D_{\rho(X_t)}X_t=0$ are defined for all times. Then there exist a Lie group $G$ acting on $M$ and an isomorphism
    \[
    (\Sigma,\omega)\cong (G\ltimes M,\mathrm{pr}^*_1(\omega_\mathrm{MC}))
    \]
    where $\omega_\mathrm{MC}\in \Omega^1(G,\mathfrak{g})$ denotes the Maurer-Cartan form of $G$.
\end{proposition}
\begin{proof}
    The proof follows from the infinitesimal version of the statement --  which is an application of parallel transport, see Proposition~\ref{prop:flat_Cartan_action} -- and an integration argument. 
\end{proof}
 \begin{definition}\label{def:flat_gpds}
     A groupoid with connection $(\Sigma,\omega)$ is called a \textbf{flat groupoid} 
     \[
     \Cbas^\omega =0,\quad \Curv^\omega=0.
     \]
 \end{definition}
 By the proposition above, action groupoids arising from group actions are flat and, under certain topological assumptions, flat groupoids arise from actions of groups. 

We conclude by spending some words on the interaction between curvature and basic curvature. Notice that the curvature of a connection $\omega$ on $\Sigma\rightrightarrows M$ induces the $1$-cochain $\Curv^\omega\in C^1(\Sigma, s^*\Lambda^2T^*M\otimes t^*A)$ valued in the ''quasi-representation'' $\Lambda^2T^*M\otimes A$ and given by
    \[
    g\to \Curv^\omega_g \in  \Lambda^2T_{s(g)}M \otimes A_{t(g)}
    \]
\begin{lemma}[{see~\cite[Lemmas 6.2.30 and 6.2.31]{Sal13}}]\label{lem:Curv_cocycle}
    Let $\omega\in \Omega^1(\Sigma,t^*A)$ be a connection on $\Sigma$. If $\omega$ is multiplicative, the one cochain associated to $\Curv^\omega$ is a groupoid cocycle -- i.e. $\delta\Curv^\omega=0$, where $\delta$ is the usual simplicial differential with values in a representation.
\end{lemma}
In other words we see that, if $\omega$ satisfies the multiplicativity equation from Lemma~\ref{lem:multipl_equiv_notion}, then $\Curv^\omega$ satisfies the same equation (an identity of forms on $\Sigma^{(2)}$):
        \[
		\Curv^\omega_{gh}(v,w) = \Curv^\omega_g(h\cdot v,h\cdot w)+g\cdot \Curv^\omega_h(v,w)
	\]
 for all compatible arrows $(g,h)\in \Sigma^{(2)}$ and all $v,w\in T_{s(h)}M$.

\subsection{The infinitesimal picture: connections on algebroids}\label{subs:conn_algebroids}
The infinitesimal counterpart of connections on groupoids is provided by linear connections $D$ on vector bundles $A\to M$ underlying their Lie algebroids, i.e. by operators
\[
D:\Gamma(A)\to \Omega^1(M,A)
\]
satisfying the Leibniz identity
\[
D(f\alpha) = fD(\alpha)+df\otimes \alpha,\quad f\in C^\infty(M),\ \alpha\in \Gamma(A).
\]

The statements and definition below are mostly from \cite{AbCr2011RepUpHomAlg}.

\begin{proposition}[{\cite[\S 2.3]{AbCr2011RepUpHomAlg}; see also~\cite[Corollary 5.5]{AccCr2023}}]\label{prop:adj_rep}
    Let $(A,D)$ be a Lie algebroid with connection over $M$. Then the following two operators
    \[
    \nabla^A: \Gamma(A)\times \Gamma(A)\to \Gamma(A),\quad \nabla^A_\alpha(\beta) = [\alpha,\beta]+D_{\rho(\beta)}(\alpha),
    \]
    \[
    \nabla^{TM}: \Gamma(A)\times \mathcal{X}(M)\to \mathcal{X}(M),\quad \nabla^{TM}_\alpha(X) = [\rho(\alpha),X]+\rho(D_{X}(\alpha)),
    \]
    define an $A$-connection $\nabla$ on the complex
    \[
    A\overset{\rho}{\to} TM,
    \]
    called the \textbf{adjoint connection}.
\end{proposition}
\begin{definition}\label{def:inf_basic_curvature}
    Let $(A,D)$ be a Lie algebroid with connection over $M$. The \textbf{basic curvature} of $(A,D)$ is the tensor
    \[
    \cbas^D \in \Omega^2(A,\mathrm{Hom}(TM,A))
    \]
    given by
    \[
    \cbas^D(\alpha,\beta)X = D_X[\alpha,\beta] - [D_X\alpha,\beta] - [\alpha,D_X\beta] + D_{\nabla_{\alpha} X} \beta - D_{\nabla_{\beta} X} \alpha
    \]
\end{definition}

\begin{proposition}[{\cite[Proposition 2.11]{AbCr2011RepUpHomAlg}; see also~\cite[Definition 5.2]{AccCr2023}}]\label{prop:cbas_identities}
The following facts are true.
\begin{itemize}
    \item The following identities hold:
    \[
    \cbas^D\circ \rho = -R_{\nabla^{A}},
    \]
    \[
    \rho\circ\cbas^D\ = -R_{\nabla^{TM}},
    \]
    where the right hand sides denote the curvatures of the $A$-connections $\nabla^{TM}$ and $\nabla^{A}$.
    In particular if $D$ is an infinitesimally Cartan connection on $A$, then the adjoint $A$-connection $\nabla$ is flat, i.e. it is a Lie algebroid representation.
    \item $\cbas^D$ is $d_\nabla$-closed: $d_\nabla(\cbas^D)=0$.
\end{itemize}
\end{proposition}
\begin{remark}\label{rmk:rep_up_to_homotopy}
    The framework that truly illuminates what is going on is that of representations up to homotopy from~\cite{AbCr2011RepUpHomAlg}. The point is that the Lie algebroid $A$ admits a representation up to homotopy on the graded vector bundle
    \[
    E^{-1}:=A,\quad E^0:=TM
    \]
    given by the triple $(\rho, \nabla, \cbas^D)$. We do not elaborate further on this point here.
\end{remark}

\begin{definition}\label{def:inf_mult}
	A connection $D$ on a Lie algebroid $A$ is called \textbf{infinitesimally multiplicative} if $\cbas^D = 0$.
\end{definition}

Cartan connections are compatible with the groupoid/algebroid structure, and one can discuss differentiation and integration for them. The following proposition appears, in greater generality, in~\cite{CrSalStr2012}.
\begin{proposition}[{\cite{CrSalStr2012}; see also~\cite[Theorem 5.9]{AccCr2023}}]\label{prop:diff_conn}
    Let $\omega$ be a Cartan connection on a Lie groupoid $\Sigma\tto M$. The following formula defines an infinitesimally Cartan connection on the associated Lie algebroid $A\to M$:
    \[
    D_X\alpha = \omega([{\mathrm{hor}^\omega(X)}, \alpha^\mathrm{R}]|_M),
    \]
    where ${\mathrm{hor}^\omega(X)}\in \Gamma(\ker(\omega))$ is the horizontal lift of $X$ and $\alpha^R\in \Gamma(\ker(ds))$ is the right invariant vector field on $\Sigma$ corresponding to $\alpha\in \Gamma(A)$.
    
    Conversely, let $(A,D)$ be a Lie algebroid equipped with an infinitesimally Cartan connection and let $\Sigma$ be an $s$-simply connected Lie groupoid integrating $A$. Then $D$ can be integrated to a Cartan connection $\omega$ on $\Sigma$. The two constructions are inverse of each other.
\end{proposition}
Under the Lie functor established in the previous proposition, the adjoint representation of $(\Sigma,\omega)$ with connection correspond to the adjoint representation of $(A, D^\omega)$.

    Together with the basic curvature, a Lie algebroid $A$ with connection $D$ carries the ordinary \textbf{curvature} of $D$, i.e. the tensor $\curv^D\in \Omega^2(TM,A^*\otimes A)$
     \[
     \curv^D \in \Omega^2(M,\mathrm{Hom}(A,A))
     \]
     given by
     \[
     \curv^D(X,Y)\alpha = D_XD_Y\alpha-D_YD_X\alpha-D_{[X,Y]}(\alpha).
     \]
     For infinitesimally Cartan connections, $\curv^D \in \Omega^2(TM,A^*\otimes A) \cong C^1(A,\Lambda^2T^*M\otimes A)$ becomes a Lie algebroid cocycle -- where $\Lambda^2T^*M\otimes A$ is made into a representation of $A$ by using the adjoint representation induced by $D$. Moreover
\begin{proposition}[{cf.~\cite[Lemma 6.2.32]{Sal13}}]\label{prop:curv_alg_is_lin_curv_gpd}
    Let $(\Sigma,\omega)$ be a Lie groupoid with Cartan connection, and let $D$ be the corresponding infinitesimally Cartan connection on its Lie algebroid. The Lie algebroid cocycle corresponding to the curvature $\Curv^\omega$ is precisely $\curv^D$.
\end{proposition}

\begin{definition}\label{def:flat_algebroid}
    A Lie algebroid $(A,D)$ is called a \textbf{flat algebroid} if 
    \[
    \cbas^D = 0,\quad \curv^D = 0.
    \]
\end{definition}
Flat Cartan connections strongly constrain the structure of an algebroid.
\begin{proposition}[{\cite[Theorem A]{Blaom2006}}]\label{prop:flat_Cartan_action}
    Let $(A,D)$ be a flat algebroid over a simply connected manifold. Then there exists a Lie algebra $\mathfrak{g}$ acting on $M$ and an isomorphism of algebroids
    \[
    A\cong \mathfrak{g}\ltimes M
    \]
    such that $D$ is the pullback connection of the standard flat connection on $\mathfrak{g}\ltimes M$.
\end{proposition}
\begin{proof}
    For our purposes, it is insightful to have descriptions of $\mathfrak{g}$. Hence, we sketch two arguments.

    For the first argument, we consider
    \begin{equation}\label{eq:torsion_bracket}
    \{\ ,\ \}: \Gamma(A)\times \Gamma(A) \to \Gamma(A),\quad (\alpha,\beta) \to - D_{\rho(\alpha)} \beta + D_{\rho(\beta)}\alpha + [\alpha,\beta];
    \end{equation}
    since the expression on the right is $C^\infty(M)$-linear in $\alpha$ and $\beta$, it induces a bracket $\{\ ,\ \}$ on $A$. In fact, such bracket is the opposite of the torsion of the connection $D$, and a direct computation shows that the Jacobi identity holds precisely due to the fact that $D$ is infinitesimally multiplicative and flat. We denote by $\mathfrak{a}_x$ the Lie algebra $(A_x,\{\ ,\ \}$. Another computation shows that $D$ acts on $(\Gamma(A),\{\ ,\ \})$ by derivations, ensuring that $\mathfrak{a}_x$ does not depend on $x$. It also follows that, if $M$ is simply connected, parallel transport with respect to $D$ gives an isomorphism between $A$ and $\mathfrak{a}_x\times M$, and it is straightforward to check that this is an isomorphism of algebroids.
    
    For the second argument, we use that $D$ is flat to identify $A_x$ with the set of germs of flat sections around $x\in M$ -- that is, germs around $x$ of sections $\alpha\in \Gamma(A)$ such that $D_X\alpha=0$ for all vector fields $X$. Since $M$ is simply connected, each germ of a flat section at $x$ is represented by a (necessarily unique) flat section defined on the whole $M$. The Lie bracket of $A$ restrict to a Lie bracket on the set $\Gamma(A,D)$ of flat sections because $D$ is infinitesimally multiplicative, and the resulting Lie algebra acts on $M$ via the anchor of $A$. Then, one has the algebroid isomorphism 
    \[
    A\to \Gamma(A,D)\ltimes M, \quad a\in A\mapsto (\alpha_a, \pi(a)),
    \]
    where $\pi:A\to M$ is the vector bundle projection and $\alpha_a\in \Gamma(A,D)$ is the flat section such that $\alpha_a(\pi(a))=a$.
    
    Finally, notice that the identification between $A_x$ and $\Gamma(A,D)$ induces a Lie algebra isomorphism between $\mathfrak{a}_x$ and $\Gamma(A,D)$.
\end{proof}

\subsection{The Maurer-Cartan expression}\label{subs:mc_expression}
Let $\omega$ be a Cartan connection on a Lie groupoid $\Sigma\tto M$, and let $D$ be the corresponding infinitesimally Cartan connection on the corresponding Lie algebroid. 
\begin{definition}\label{def:mc_expr_gpd}
	The \textbf{Maurer-Cartan expression} of $\omega$ is the $2$-form
	\[
	\mc_\omega := d_{t^*D}\omega + \frac{1}{2}\{\omega,\omega\}\in \Omega^2(\Sigma,t^*A),
	\]
	where
	\begin{itemize}
		\item $d_{t^*D}:\Omega^\bullet(\Sigma,t^*A) \to \Omega^{\bullet+1}(\Sigma,t^*A)$ is the operator induced by the pullback connection $t^*D$ of $D$ under the target map, hence
  \[
    d_{t^*D}\omega (X,Y) = D_{dt(X)}\omega(\hat{Y}) - D_{dt(Y)}\omega(\hat{X}) - \omega([\hat{X},\hat{Y}])
  \]
  for all $X,Y\in T\Sigma$, where $\hat{X},\hat{Y}$ are any $t$-projectable extensions;
		\item $\{\ ,\ \}$ is the opposite of the torsion of $D$, see Proposition~\ref{prop:flat_Cartan_action}:
			\[
				\{a,b\} = -D_{\rho(\alpha)}(\beta)+D_{\rho(\beta)}(\alpha)+[\alpha,\beta]
			\]
			for all $a,b\in A$, where $\alpha,\beta\in \Gamma(A)$ are extensions.
	\end{itemize}
\end{definition}

The following proposition shows that the Maurer-Cartan expression encodes the curvature.
\begin{proposition}[{see also~\cite[Proposition 5.18, its proof, and Remark 5.19]{AccCr2023}}]\label{prop:mc_expr_vanishing}
	Let $(\Sigma,\omega)$ be a groupoid with Cartan connection. Then
	\[
	\mc_\omega = 0
	\]
	if and only if $\omega$ is flat. More precisely, under the splitting
	\[
	\Omega^2(\Sigma,t^*A) = \Gamma(s^*\Lambda^2TM,t^*A)\oplus \Gamma(s^*TM\otimes t^*A,t^*A)\oplus \Gamma(\Lambda^2 t^*A,t^*A)
	\]
	the first component of $\mc_\omega$ is $\Curv^\omega$, and the other components vanish.
\end{proposition}
\begin{proof}
    The fact that $\mc_\omega$ vanishes on $s$-vertical vectors follows straight away from the fact that the anchor $\rho$ of $A$ is given by the restriction of $dt$ to $\ker_M(ds)$, and the equality
    \[
    \omega|_{\ker(ds)} = \omega_{MC}.
    \]
    
    When $X,Y\in T_{s(g)}M$, $g\in \Sigma$, we get
    \[
    (\mc_\omega)_g(X,Y) = -\omega([\hat{\mathrm{hor}^\omega_g(X)},\ \hat{\mathrm{hor}^\omega_g(Y)}])
    \]
    where in the right hand side we have extensions of the horizontal lifts 
    \[
    \mathrm{hor}^\omega_g(X),\ \mathrm{hor}^\omega_g(Y)\in \ker_g(\omega);
    \]
    this is precisely the curvature of $\omega$ at $X,Y$.

    We are left with the case $X\in T_{s(g)}M$, $a\in A_{t(g)}$, which gives, after taking an extension ${\mathrm{hor}^\omega(X)}\in \Gamma(\ker(\omega))$ of the horizontal lift of $X$ at $g$ and an extension $\alpha\in \Gamma(A)$ of $a$,
    \[
    (\mc_\omega)_g(X,a) = D^t_{\mathrm{hor}_g^\omega(X)}(\omega(\alpha^R)) -\omega([{\mathrm{hor}^\omega_g(X)},\ \alpha^R]).
    \]
    The vanishing of the above for all $X\in T_{s(g)}M$, $a\in A_{t(g)}$ follows from multiplicativity, as the next lemma implies. 
\end{proof}
\begin{lemma}[{see~\cite[Proposition 5.3.4]{Catt2020}, and the proof of~\cite[Proposition 5.18]{AccCr2023}}]\label{lem:two_connections_on_action_gpd}
    The following two connections on the action groupoid $\Sigma\ltimes \Sigma\tto \Sigma$ 
    \[
    \delta\omega\in \Omega^1(\Sigma\times_M\Sigma),\quad \delta\omega_{(g,h)} := (m^*\omega)_{(g,h)} - g\cdot (pr_2^*\omega)_{(g,h)}
    \]
    \[
    pr_1^*\omega\in \Omega^1(\Sigma\times_M\Sigma)
    \]
    are multiplicative and induce the following two connections on the Lie algebroid $t^*A = \mathrm{Lie}(\Sigma\ltimes \Sigma)$:
    \[
    D^{\delta\omega}: (Y,\beta) \mapsto (t^*\nabla^{A})_\beta (\omega(Y))-\omega[\beta,Y]
    \]
    \[
    D^{pr_1^*\omega}: (Y,\beta) \mapsto  D^t_Y\beta
    \]
    for all $Y\in \mathcal{X}(\Sigma)$, $\beta\in \Gamma(t^*A)$ (in the last term we use the identification $t^*A\cong \ker(ds)$).
\end{lemma}
\begin{proof}
    
    Propositions 3.4.25 and 3.4.26 in \cite{Catt2020}.
\end{proof}
    Since, on the one hand, multiplicativity implies
    \[
    D^{\delta\omega} = D^{pr_1^*\omega}
    \]
    and, on the other hand, ${\mathrm{hor}^\omega_g(X)}\in \Gamma(\ker(\omega))$ and $\alpha^R\in \Gamma(t^*A)$, by the Lemma above the difference $D^{pr_1^*\omega} - D^{\delta\omega}$ becomes
    \[
    D^t_{\mathrm{hor}_g^\omega(X)}(\omega(\alpha^R)) -\omega([{\mathrm{hor}^\omega_g(X)},\ \alpha^R]).
    \]
    Hence, if $\omega$ is multiplicative $(\mc_\omega)_g(X,a)=0$ for all $X\in \mathcal{X}(M)$, $\alpha\in \Gamma(A)$, concluding the proof of Proposition \ref{prop:mc_expr_vanishing}.

We defer examples to Section \ref{sec:examples}.

\section{Principal bundles with a structure groupoid with connection}\label{sec:principal_bundles}
\subsection{Groupoid actions and principal bundles}\label{subs:principal_bundles}
Recall that, given a Lie groupoid $\Sigma\rightrightarrows M$, a (left) \textbf{$\Sigma$}-space $P$ consists of a manifold $P$ together with a \textbf{momentum map} $\mu:P\to M$ and a smooth map
\[
m_P:\Sigma \tensor[_s]{\times}{_\mu} P \to P
\]
satisfying the associativity property
\[
m_P(g,m_P(h,p))=m_P(gh,p),\quad (g,h,p) \in \Sigma\tensor[_s]{\times}{_t}\Sigma \tensor[_s]{\times}{_\mu} P;
\]
when there is no danger of confusion, we will write $m_P(g,p)=g\cdot p$. We call the map $m_P$ a (left) \textbf{action} of $\Sigma$ on $P$ along $\mu$. 

Given $p\in P$, the orbit of $p$ is defined as the set of $p'$ such that $p'=g\cdot p$, for some $g\in s^{-1}(\mu(p))$. Orbits define a partition of $P$; the corresponding quotient is denoted by $P/\Sigma$ and called the \textbf{orbit space} of the action.

An action of a Lie groupoid $\Sigma$ along a map $\mu:P\to M$ can be neatly packed into the \textbf{action groupoid} $\Sigma\ltimes P$, i.e. the Lie groupoid over $P$ whose arrow space is $\Sigma\tensor[_s]{\times}{_\mu}P$ and whose structure maps are
\[
s(g,p) = p,\ t(g,p) = g\cdot p,\ (h,g\cdot p)\cdot(g,p) = (hg,p).
\]
Notice that the orbits of $\Sigma\ltimes P \tto P$ are precisely the orbits of the $\Sigma$-action.

Actions of Lie groupoids have an infinitesimal counterpart, just as actions of Lie groups. Namely, one considers the bundle map
\[
a_P:\mu^*A\to TP,\quad \alpha_{\mu(p)}\to dm_P(\alpha_{\mu(p)},0_p)
\]
where $A$ denotes the Lie algebroid of $\Sigma$. It follows from the associativity property of $m_P$ that $\rho = d\mu \circ a_P$, where $\rho$ is the anchor of $A$, and that the map of sections (still denoted by $a_P$)
\[
a_P:\Gamma(A) \to \mathfrak{X}(P),\quad \alpha \to \left[p\to a_P(\alpha_{\mu(p)})\right]
\]
is a map of Lie algebras. 

More generally, an \textbf{$A$-space} is a smooth manifold $P$ together with a momentum map $\mu:P\to M$ and a vector bundle map $a_P:\mu^*A\to TP$ over $P$ satisfying $d\mu\circ a_P=\mu^*\rho$ and inducing a Lie algebra homomorphism $\Gamma(A)\to \mathfrak{X}(P)$; we also say that $A$ \textbf{acts} on $P$ along $\mu$. The image of $a_P$ in $TP$ defines an integrable (singular) distribution; the orbits of the action are the leaves of such distribution. As we just recalled, an action of a Lie groupoid induces an action of the associated Lie algebroid along the same map. Moreover, an action of an $s$-connected groupoid is totally determined by
its associated infinitesimal action. Conversely, an infinitesimal action of a Lie
algebroid integrates to an action of its $s$-simply connected integration
whenever it is \textbf{complete}, i.e.\ when $a_P(\alpha)$ is a complete vector
field for every compactly supported $\alpha\in \Gamma(A)$. See~\cite[Theorem
5.3]{MoerMrc2002} together with the remark following its proof.

Similar to groupoid actions, an action of $A$ on $P$ along the moment map $\mu:P\to M$ can be described in terms of the \textbf{action algebroid} $A\ltimes P$, whose underlying vector bundle is $\mu^*A$, while the anchor is the map $a_P$ and the Lie bracket on $\Gamma(\mu^*A)$ is the unique one satisfying
\[
[\mu^*\alpha,\mu^*\beta] = \mu^*[\alpha,\beta].
\]
Of course, when the infinitesimal action is induced by a Lie groupoid action of $\Sigma$ on $P$, the action algebroid $A\ltimes P$ associated to it is precisely the Lie algebroid associated to $\Sigma\ltimes P$. 

A \textbf{principal $\Sigma$-bundle} $P$ over a manifold $X$ is a (left) $\Sigma$-space $\mu:P\to M$ together with a $\Sigma$-invariant surjective submersion $\pi:P\to X$ such that the action of $\Sigma$ is free and proper on $P$ and transitive along the fibers of $\pi$. Here, freeness and properness are defined as usual: an action $m_P$ of $\Sigma$ along $\mu:P\to M$ is free when 
\[
m_P(g,p)=p \implies g=1_{\mu}(p)
\]
and proper when 
\[
\Sigma \tensor[_s]{\times}{_\mu} P \to P\times P,\quad (g,p) \to (g\cdot p, p).
\]
is a proper map. Similarly, transitivity along fibers means that for all $x\in X$, given $p,p'\in \pi^{-1}(x)$ there is $g\in s^{-1}(\mu(p))$ such that $g\cdot p = p'$. Notice that the orbits of a principal $\Sigma$-bundle are precisely the fibers of $\pi$, hence we have $P/\Sigma \cong X$; under this identification $\pi$ is the quotient projection. A \textbf{morphism} of principal $\Sigma$-bundles $P_1$, $P_2$ over the same base $X$ is a $\Sigma$-equivariant smooth map $\varphi:P_1\to P_2$ which covers the identity of $X$ and intertwines the momentum maps, $\mu_2\circ\varphi=\mu_1$. As in the case of principal group bundles such a map is necessarily an isomorphism.

\begin{remark}
    A principal bundle is often defined by requiring the action map
    \[
    \Sigma \tensor[_s]{\times}{_\mu} P\ \longrightarrow\ P\times_XP,
    \qquad (g,p)\mapsto(g\cdot p,p),
    \]
    to be a diffeomorphism. Bijectivity is equivalent to freeness and transitivity along the fibers of the action. The smoothness of its inverse -- the division map -- encodes properness of the action.
\end{remark}

Notice that the infinitesimal action $a_P:\mu^*A\to TP$ corresponding to a principal $\Sigma$-bundle is injective, with image $T^\pi P:=\ker(d\pi)$, the bundle of $\pi$-vertical vectors.

Let us provide two basic examples.
\begin{example}\label{ex:trivial_bundle_over_f}
    Given a map $f:X\to M$, the \textbf{trivial principal $\Sigma$-bundle over $f:X\to M$} is the principal bundle with total space $P_f: \Sigma\tensor[_s]{\times}{_f} X$, momentum map $t\circ \mathrm{pr}_1$ and projection $\mathrm{pr}_2$, while the action is given by the left multiplication on $\Sigma$. 
\end{example}
\begin{example}\label{ex:trivial_Sigma_bundle}
    Taking $f=\mathrm{id}_M$ gives $P_f=\Sigma$, acting on itself by left multiplication, with momentum map $t$ and projection $s$. We call this one the \textbf{trivial principal $\Sigma$-bundle}.
\end{example}

Following the examples above, we say that a principal $\Sigma$-bundle over $X$ is \textbf{trivializable} when it is isomorphic to a trivial bundle $P_f$ over $X$ for some $f:X\to M$.

The following statements are the groupoid counterparts of standard facts about principal group bundles. 
\begin{lemma}\label{lem:trivializations_are_sections}
    There is a one to one correspondence between trivialization of a principal $\Sigma$-bundle $P\to X$ and sections $\sigma:X\to P$ of the projection map. More precisely, the following facts hold.
    \begin{itemize}
        \item Any section $\sigma:X\to P$ induces the function $f_\sigma:=\mu\circ \sigma$ and the trivialization
        \[
        \Phi_\sigma: P \to \Sigma\times_M X,\quad p\to ([p:\sigma(\pi(p))],\pi(p)),
        \]
        where $[p:\sigma(\pi(p))]$ denotes the unique $g\in \Sigma$ such that $g\cdot \sigma(\pi(p)) = p$.
        \item A trivialization $\Phi: P\to \Sigma\times_M X$, where the fibered product is with respect to the source map of $\Sigma$ and a smooth map (not necessarily surjective) $f:X\to M$, induces a global section
        \[
        \sigma_\Phi:X\to P,\quad x \to \Phi^{-1}(1_{f(x)},x).
        \]
        \item The two constructions above are inverse of each other.
    \end{itemize}
\end{lemma}
\begin{lemma}\label{lem:change_of_trivialization}
    Let $P$ be a principal $\Sigma$-bundle over $X$ and $\Phi:P\to \Sigma\times X$ be a trivialization with respect to $f:X\to M$. There is a one to one correspondence between trivializations $\Psi$ of $P$ and maps $F:X\to \Sigma$ such that $s\circ F = f$. More precisely, the following facts hold.
    \begin{itemize}
        \item Any map $F:X\to \Sigma$ induces the following trivialization over $g:= t\circ F$:
        \[
        \Psi_F:P\to \Sigma\times_M X,\quad p\to \bigl(\Phi_1(p)\,F(\pi(p))^{-1},\ \pi(p)\bigr),
        \]
        \item Any trivialization $\Psi \to \Sigma\times_M X$ over $g:X\to M$ induces the following map:
        \[
        F_\Psi:X\to \Sigma,\quad x\to [\Psi^{-1}(1_{g(x)},x):\Phi^{-1}(1_{f(x)},x)],
        \]
        which satisfies $s\circ F_\psi = f$, $t\circ F_\psi = g$.
        \item The two constructions above are inverse of each other.
    \end{itemize}
\end{lemma}
\begin{lemma}\label{lem:locally_trivial}
    Any principal $\Sigma$-bundle is locally trivializable. That is, for all $x\in X$ there exists an open neighbourhood $U_x$, a map
    \[
    f:U_x \to M,
    \]
    and a trivialization
    \[
    \Phi: P|_{U_x} \to \Sigma \times_M U_x.
    \]
\end{lemma}
Still proceeding by analogy with principal group bundles, one can characterize principal bundles in terms of groupoid cocycles. In fact, by Lemma~\ref{lem:locally_trivial} and Lemma~\ref{lem:change_of_trivialization}, any principal $\Sigma$-bundle $P$ comes with 
\begin{itemize}
    \item an open cover $\{U_i\}_{i\in I}$ of $X$,
    \item maps $f_i:U_i\to M$,
    \item maps $\varphi_{ij}:U_i\cap U_j\to \Sigma$,
\end{itemize}
such that
\begin{itemize}
    \item $s(\varphi_{ij}(x)) = f_i(x)$ and $t(\varphi_{ij}(x)) = f_j(x)$ for all $x\in U_i\cap U_j$;
    \item $ \varphi_{jk}(x)\cdot \varphi_{ij}(x) = \varphi_{ik}(x)$.
\end{itemize}
Any triple $(U_i,f_i,\varphi_{ij})_{i,j\in I}$ as above is called a \textbf{$\Sigma$-valued cocycle}. Two $\Sigma$-valued cocycles are \textbf{equivalent} if there is a $\Sigma$-valued cocycle containing both, and different $\Sigma$-valued cocycles arising from a principal bundle $P$ are equivalent. If a $\Sigma$-valued cocycle $(U_i,f_i,\varphi_{ij})_{i,j\in I}$ on $X$ is given, we can set
\[
P := \coprod\limits_{} \Sigma\times_{M} U_i / \sim,
\]
where $\sim$ is the equivalence relation 
\[
(g,x)\in \Sigma\times_{M} U_i\sim (h,y)\in \Sigma\times_{M} U_j\ \text{if and only if } x_i=y_j,\ g=h\cdot \varphi_{ij}(x);
\]
the resulting $P$ is a principal $\Sigma$-bundle over $X$, which is trivializable over each $f:U_i\to M$. It is also possible to check that equivalent cocycles result in isomorphic principal bundles. Putting it all together and filling in the details, one gets the following.
\begin{proposition}\label{prop:principal_bundles_are_cocycles}
    There is a one-to-one correspondence (in fact, an equivalence of categories) between principal $\Sigma$-bundles over a manifold $X$ up to isomorphisms and $\Sigma$-valued cocycles on $X$ up to equivalence.
\end{proposition}

\subsection{Principal $(\Sigma,\omega)$-connections}\label{subs:principal_connections}
\begin{definition}\label{def:connection_pb}
    Let $\Sigma$ be a Lie groupoid and let $\pi: P\to X$ be a principal $\Sigma$-bundle with momentum map $\mu$. A \textbf{connection on $P$} is an Ehresmann connection $\mathcal{C}\subset TP$ with respect to $\pi$, i.e. a regular distribution such that
    \[
    TP = \mathcal{C} \oplus \ker(d\pi).
    \]
\end{definition}

Of course, a connection on a principal bundle can be equivalently encoded as a splitting of the short exact sequence
\[
T^\pi P \to TP \to \pi^*TX;
\]
 this point of view makes it clear that connections do always exist. Such a splitting has two equivalent incarnations:
\begin{itemize}
    \item a \textbf{horizontal lift}
    \[
    v_x\in T_xX \mapsto \hor^{\mathcal{C}}_p(v_x) \in \mathcal{C}_p
    \]
    such that $d\pi \circ \hor^{\mathcal{C}} = \id$;
    \item a \textbf{vertical projection}
    \[
    v_p\in T_pP \to \pr^\pi(v_p) \in T^\pi P
    \]
    such that $\pr^\pi|_{T^\pi P} = \id$.
\end{itemize}

Since the infinitesimal action
\[
a_P: \mu^*A\to TP
\]
defines a $T^\pi P$-valued linear isomorphism, the vertical projection associated to a connection on $P$ can be alternatively described as a one form $\theta\in \Omega^1(P,\mu^*A)$ satisfying the normalisation
\[
\theta\circ a_P = \id .
\]
From now, when talking about connections we refer to such forms $\theta$.

We are interested in connections that are invariant under the groupoid action in some reasonable sense. Here, the theory starts differing sensibly from its group counterpart; for instance, there is no natural action of $\Sigma$ on $TP$ -- in fact, an arrow $g\in \Sigma$ acts via $dm_P$ as a linear isomorphism
\[
g:T \mu^{-1}(s(g)) \to T \mu^{-1}(t(g))
\]
and there is no natural way to extend this to an isomorphism $T_{s(g)}P \to T_{t(g)}P$. Moreover, there is no natural representation of $\Sigma$ on its Lie algebroid $A$ (which provides the coefficients of a connection $1$-form $\theta$). 

On the other hand, as we saw in Subsection~\ref{subs:curvature_gpd} (see Definition~\ref{def:adjoint_rep_gpd} and Lemma~\ref{lem:multiplicativity_properties}), a representation of $\Sigma$ on $A$ can be constructed if $\Sigma$ is equipped with a Cartan connection and, in that case, one can also construct an action of $\Sigma$ on the tangent space $TP$ to any $\Sigma$-space. In fact, if $\omega\in \Omega^1(\Sigma, t^*A)$ is a Cartan connection on $\Sigma$, then $\pr_1^*\omega$ is a Cartan connection on the action groupoid $\Sigma\ltimes P$ and, as such, it induces a representation of $\Sigma\ltimes P$ on $TP$ -- hence an action of $\Sigma$ on $TP$. Explicitly, for $g\in\Sigma$ and $v\in T_pP$ with $s(g)=\mu(p)$,
\[
g\cdot v\ =\ dm_P\bigl(\hor^\omega_g(d\mu(v)),\,v\bigr)\ \in\ T_{g\cdot p}P .
\]

\begin{definition}\label{def:princ_conn}
    Let $\theta\in \Omega^1(P,\mu^*A)$ be a connection on a principal $\Sigma$-bundle $P$, and let $\omega$ be a Cartan connection. $\theta$ is called a \textbf{principal $(\Sigma,\omega)$-connection}, and $(P,\theta)$ is called a \textbf{principal $(\Sigma,\omega)$-bundle} if
    \[
        \theta (g\cdot v) = g\cdot \theta(v)\qquad\text{for all } v\in TP .
    \]

    A morphism of \textbf{principal $(\Sigma,\omega)$-bundles} $(P_1,\theta_1)$ and $(P_2,\theta_2)$ over the same base $X$ is a morphism of the underlying principal $\Sigma$-bundles that pulls back $\theta_2$ to $\theta_1$.
\end{definition}

One can characterize principal connections as follows.
\begin{definition}\label{def:mult_conn_pb}
	Let $(\Sigma,\omega)$ be a groupoid with a Cartan connection, and 
	let $\pi: P\to X$ be a principal $\Sigma$-bundle with momentum map $\mu$. A connection $\theta\in \Omega^1(P,\mu^*A)$ is called a \textbf{multiplicative with respect to $(\Sigma,\omega)$}, if the following identity holds:
\begin{equation}\label{eq:mult_up_to_omega}
\delta\theta_{(g,p)} =  \mathrm{pr}_1^*\omega
\end{equation}
where $\delta\theta$ denotes the simplicial differential of $\theta\in \Gamma(\mathrm{Hom}(TP,\mu^*A))$ on $\Sigma\ltimes P$, i.e.
    \[
    \delta\theta_{(g,p)} = m_P^*\theta_{(g,p)} - g\cdot(\mathrm{pr}_2^*\theta)_{(g,p)}, \quad (g,p)\in \Sigma\ltimes P;
    \]
\end{definition}
 Multiplicativity of $\theta$ is equivalent to $\theta$ being a principal $(\Sigma,\omega)$-connection.

The formulas defining $g\cdot v$ and $g\cdot \alpha$ make sense for any
$\omega\in \Omega^1(\Sigma,t^*A)$ with $\omega|_{\ker(ds)} = \omega_\mathrm{MC}$, multiplicative
or not, and so does~\eqref{eq:mult_up_to_omega}; we state the equivalence in that generality.
\begin{lemma}\label{lem:principal_iff_multiplicative}
    Let $\omega\in \Omega^1(\Sigma,t^*A)$ satisfy $\omega|_{\ker(ds)} = \omega_\mathrm{MC}$, and
    let $P$ be a principal $\Sigma$-bundle. A connection $\theta\in \Omega^1(P,\mu^*A)$ satisfies
    \[
        \theta(g\cdot v) = g\cdot \theta(v)\qquad \text{for all } v\in TP
    \]
    if and only if it satisfies~\eqref{eq:mult_up_to_omega}.
\end{lemma}
\begin{proof}
    Let $(\xi,v)\in T_{(g,p)}(\Sigma\ltimes P)$, so that $ds(\xi) = d\mu(v)$, and write
    $\xi = \hor^\omega_g(d\mu\, v) + \xi^v$ with $\xi^v\in \ker(ds)$. Then $\xi^v = dR_g\alpha$
    with $\alpha:=\omega(\xi^v) = \omega(\xi)\in A_{t(g)}$. Let $h(t)$ be a curve in the source
    fibre over $t(g)$ with $h(0) = 1_{t(g)}$ and $\dot h(0) = \alpha$. Since
    $m_P(h(t)g,p) = h(t)\cdot (g\cdot p)$, one has $dm_P(\xi^v,0_p) = a_P(\alpha)_{g\cdot p}$, and
    therefore
    \[
        dm_P(\xi,v)\ =\ g\cdot v + a_P(\alpha).
    \]
    If $\theta(g\cdot v) = g\cdot\theta(v)$ for all $v$, then
    $\theta(dm_P(\xi,v)) = g\cdot\theta(v) + \alpha = \omega(\xi) + g\cdot\theta(v)$, which
    is~\eqref{eq:mult_up_to_omega}. Conversely, \eqref{eq:mult_up_to_omega} applied with
    $\xi = \hor^\omega_g(d\mu\, v)$, where $\omega(\xi) = 0$, gives
    $\theta(g\cdot v) = g\cdot\theta(v)$.
\end{proof}

\begin{remark}\label{rmk:quasi_actions}
    The reader might wonder about  Definitions~\ref{def:princ_conn} and~\ref{def:mult_conn_pb} when $\omega\in \Omega^1(\Sigma,t^*A)$ is a possibly non multiplicative form satisfying
    \[
    \omega|_{\ker(ds)} = \omega_\mathrm{MC}\in \Omega(\ker(ds),t^*A)
    \]
    together with $\ker(\omega)_{1_x}=T_xM$ for every $x\in M$.
    In that setting, the actions of $\Sigma$ on $TP$ and $\mu^*A$ become quasi actions -- i.e. they are not necessarily associative. However, it is possible to see that the conditions in Definition \ref{def:princ_conn} have strong consequences on $\omega$.

    Let $\xi_g\in\ker(\omega)_g$ and $\xi_h\in\ker(\omega)_h$ be
    composable and let $u\in T_pP$ satisfy $ds(\xi_h)=d\mu(u)$, so that
    $(\xi_h,u)$ is tangent to $\Sigma\ltimes P$. Applying
    \eqref{eq:mult_up_to_omega} twice,
    \[
        \theta\bigl(dm_P(\xi_h,u)\bigr)=h\cdot\theta(u),
        \qquad
        \theta\bigl(dm_P(\xi_g,dm_P(\xi_h,u))\bigr)=g\cdot\bigl(h\cdot\theta(u)\bigr),
    \]
    the terms $\omega(\xi_g)$ and $\omega(\xi_h)$ vanishing, while applying it at
    $(gh,p)$ to the pair $(dm(\xi_g,\xi_h),u)$ gives
    \[
        \theta\bigl(dm_P(dm(\xi_g,\xi_h),u)\bigr)
        =(gh)\cdot\theta(u)+\omega\bigl(dm(\xi_g,\xi_h)\bigr).
    \]
    Associativity of the action makes the two left hand sides equal, so by
    Lemma~\ref{lem:multiplicativity_properties}
    \[
        \omega\bigl(dm(\xi_g,\xi_h)\bigr)
        = g\cdot\bigl(h\cdot\theta(u)\bigr)-(gh)\cdot\theta(u)
        = -\Cbas^\omega(g,h)\bigl(\rho(\theta(u))\bigr).
    \]
    On the other hand 
    $\omega(dm(\xi_g,\xi_h))=-\Cbas^\omega(g,h)(ds(\xi_h))$ follows from the definition of basic curvature. Subtracting, and
    using $ds(\xi_h)=d\mu(u)$ together with
    $\rho(\theta(u))=d\mu(a_P\theta(u))$,
    \[
        \Cbas^\omega(g,h)\bigl(d\mu(u-a_P\theta(u))\bigr)=0 .
    \]
    Since $u-a_P\theta(u)$ is the $\theta$-horizontal part of $u$, this says that
    $\Cbas^\omega(g,h)$ vanishes on $d\mu(\ker(\theta)_p)$, for every
    $(g,h)\in\Sigma^{(2)}$ and every $p$ with $\mu(p)=s(h)$. In particular, if the subspaces $d\mu(\ker(\theta)_p)$, $p\in\mu^{-1}(s(h))$,
    together span $T_{s(h)}M$, if follows that $\omega$ is multiplicative.
\end{remark}

Given a non necessarily principal connection $\theta$ on a principal $\Sigma$-bundle $P$, one has the usual notion of curvature -- the curvature of the distribution $\ker(\theta)\subset TP$:
\[
\Curv^\theta: \Lambda^2\ker(\theta) \to \mu^*A, \quad (v,w) \mapsto -\theta([V,W])
\]
where $V, W$ are extensions of $v,w$ to vector fields tangent to $\ker(\theta)$. The connection $\theta$ is called flat when $\Curv^\theta = 0$. We can regard $\Curv^\theta$ as a section
    \[
        \Curv^\theta\in\Gamma\bigl(P,\pi^*\Lambda^2T^*X\otimes\mu^*A\bigr),
    \]
    by composing with horizontal lifts
    \[
        \Curv^\theta(p)(v,w):=-\theta\bigl([\,\hor^\theta v,\hor^\theta w\,]\bigr)_p,
        \qquad v,w\in T_{\pi(p)}X.
    \]

The following lemma states that for principal connections the curvature $\Curv^\theta$ interacts with the curvature of the connection $\omega$ on $\Sigma$ in a way similar to~\eqref{eq:mult_up_to_omega}.
\begin{lemma}\label{lem:curv_is_mult_pb}
    Let $P$ be a principal $\Sigma$-bundle, $\omega$ a Cartan connection,
    and $(P,\theta)$ a principal $(\Sigma,\omega)$-bundle. Then, for all $(g,p)\in\Sigma\ltimes P$ and all $v,w\in\ker(\theta)_p$,
    \[
    \Curv^\theta_{gp}(g\cdot v,g\cdot w) - g\cdot \Curv^\theta_p(v,w) = \Curv^\omega_g(d\mu(v),d\mu(w)).
    \]
\end{lemma}
\begin{proof}
    We can take the curvatures of the two sides in the equality
    \[
    (m_P^*\theta)_{(g,p)} = (\pr_1^*\omega)_{(g,p)}+ g\cdot (\pr_2^*\theta)_{(g,p)}
    \] 
    and restrict to
    \[
    \{(\mathrm{hor}_g(d\mu(v_p)),v_p):\ v_p\in (\ker(\theta))_p\}
    \]
    which is contained in both
    \[
    \ker(m_P^*\theta)_{(g,p)} = \{(\mathrm{hor}_g(d\mu(v_p)),v_p):\ v_p\in (\ker(\theta))_p\}\oplus \ker(dm_P)_{(g,p)},
    \]
    as well as $\ker(\pr_1^*\omega)$ and $\ker(\pr_2^*\theta)$.
    This gives the identity
    \[
    \Curv^\theta_{gp}(g\cdot v,g\cdot w) = \Curv^\omega_g(d\mu(v),d\mu(w))+g\cdot \Curv^\theta_p(v,w).
    \]
\end{proof}

The lemma has two consequences, recorded in the two remarks below. 
\begin{remark}\label{rmk:flat_bun_flat_grpd}
    A consequence of the multiplicativity of the curvature of a principal connection is that, given a Cartan connection $\omega$ on $\Sigma$ and a principal $(\Sigma,\omega)$-bundle $(P,\theta)$, if
    \[
    d\mu:\ker(\theta) \to TM
    \]
    is fiberwise surjective, then flatness of $\theta$ implies flatness of $\omega$ is flat over $\mu(P)$; that is, $\Curv^\omega_g=0$ for every $g$ with $s(g)\in\mu(P)$.
\end{remark}
\begin{remark}\label{rmk:curv_is_basic}
    If $(\Sigma,\omega)$ is a flat groupoid -- see Definition \ref{def:flat_gpds} -- it follows from Lemma \ref{lem:curv_is_mult_pb} that the curvature of a principal $(\Sigma,\omega)$-connection $\theta$ on $P$ is a basic form, and there exists a form $\mathrm{C}\in \Omega^2(X,P[A])$ such that 
    \[
    \Curv^\theta = \pi^*C.
    \]
    Here $P[A]\to X$ is the bundle obtained from $A$ via fiber attachment along $P$, i.e. $P[A]:=P\times_M A/\Sigma$ where the action of $\Sigma$ is the diagonal action.
\end{remark}

Similarly to Definition~\ref{def:mc_expr_gpd}  and Proposition~\ref{prop:mc_expr_vanishing} for groupoid connections, the curvature can be encoded into a $2$-form on $P$.
\begin{definition}\label{def:mc_expr_pb}
    Let $\theta\in \Omega^1(P,\mu^*A)$ be a (non necessarily principal) connection one form on a principal $\Sigma$-bundle $P$, and $\omega$ be a Cartan connection on $\Sigma$. The \textbf{Maurer-Cartan expression} of $\theta$ is the two form
    \[
    \mc_\theta := d_{\mu^*D}\theta + \frac{1}{2}\{\theta,\theta\} \in \Omega^2(P,\mu^*A)
    \]
    where
    \begin{itemize}
        \item $d_{\mu^*D}$ is the operator induced by the pullback via $\mu$ of the connection $D$ on $A$ associated to $\omega$;
        \item $\{\ ,\ \}$ is the bracket~\eqref{eq:torsion_bracket}, pointwise given by the opposite of the torsion of $D$,
            \[
                \{a,b\}\ =\ -D_{\rho(a)}b+D_{\rho(b)}a+[a,b] ,
            \]
            as in Definition~\ref{def:mc_expr_gpd}. 
    \end{itemize}
\end{definition}

The following proposition is essentially Proposition 5.18, its proof, and Remark 5.19 in~\cite{AccCr2023}.
\begin{proposition}\label{prop:mc_expr_vanishing_pb}
	Let $(\Sigma,\omega)$ be a groupoid with connection, $P$ be a principal $\Sigma$-bundle, and $\theta$ be a connection one-form on $P$. Then, if $\theta$ is multiplicative and flat, then
	\[
	\mc_\theta = 0;
	\]
	moreover, the converse holds true if $\Sigma$ is $s$-connected (and hence the fibers of $\pi:P\to X$ are connected). More precisely, under the splitting
	\[
	\Omega^2(P,\mu^*A) = \Gamma(\pi^*\Lambda^2TX,\mu^*A)\oplus \Gamma(\pi^*TX\otimes \mu^*A,\mu^*A)\oplus \Gamma(\Lambda^2 \mu^*A,\mu^*A)
	\]
	we have that
	\begin{itemize}
		\item the first component of $\mc_\theta$ is $\Curv^\theta$;
		\item the second component of $\mc_\theta$ vanishes if $\theta$ is multiplicative, and the converse holds under $s$-connectedness assumptions; 
            \item the third component of $\mc_\theta$ vanishes identically.
	\end{itemize}
\end{proposition}
\begin{proof}
    The first and third point are easily checked. The second point follows from a rather general identity concerning multiplicativity of forms on principal groupoid bundles, proven in Proposition 5.3.4 in~\cite{Catt2020}; see also~\cite[Proposition 5.18, its proof, and Remark 5.19]{AccCr2023}.
\end{proof}
For convenience, we state the following corollary.
\begin{corollary}\label{cor:mc_is_curvature}
    Let $(\Sigma,\omega)$ be a groupoid with Cartan connection, $P$ be a principal $\Sigma$-bundle, and $\theta$ be a principal $(\Sigma,\omega)$-connection. Then, for all $w_1,w_2 \in TP$
    \[
    \mc_\theta(w_1,w_2) = \Curv^\theta(w_1^\theta,w_2^\theta)
    \]
    where $w_1^\theta$ and $w_2^\theta$ are the components of $w_1$ and $w_2$ tangent to $\ker(\theta)$.
\end{corollary}

The existence of principal $(\Sigma,\omega)$ connections follows from a standard argument.

\begin{proposition}\label{prp:princ_conn_exist}
    Let $P$ be a principal $\Sigma$-bundle and let $\omega$ be a Cartan connection on $\Sigma$. Then, $P$ admits a principal $(\Sigma,\omega)$-connection $\theta$.
\end{proposition}
\begin{proof}

    As we already pointed out, the Cartan connection $\omega$ induces a representation of $\Sigma\ltimes P$ on $TP$ and, in fact, on the short exact sequence of vector bundle
    \[
    T^\pi P\to TP \to \pi^*TX.
    \]
    Taking the quotient with respect to the action results in another sequence of vector bundles, denoted by
    \[
    T^\pi P/\Sigma \to TP/\Sigma \to TX.
    \]
    Since splittings of the latter sequence correspond precisely to principal $(\Sigma,\omega)$ connections, the existence follows.
\end{proof}
    

Finally, as for ordinary principal connection, we see that principal $(\Sigma,\omega)$ connections are closed under convex combinations.
\begin{lemma}\label{lem:pconn_are_convex}
    Let $P$ be a principal $\Sigma$-bundle and let $\omega$ be a Cartan connection on $\Sigma$. If $\theta_0$ and $\theta_1$ are principal $(\Sigma,\omega)$-connections on $P$, then
    \[
    \theta_t = t\theta_1 + (1-t)\theta_0 \in \Omega^1(P,\mu^*A),\quad t\in\mathbb{R}
    \]
    is a principal $(\Sigma,\omega)$-connection on $P$.
\end{lemma}
Both defining conditions are affine in $\theta$, so all that is used of the coefficients is that they sum to $1$; in particular $t$ may be replaced by any function on the base, which is $\Sigma$-invariant because $\pi(g\cdot p)=\pi(p)$.

\subsection{Trivializations and principal connections}\label{subs:Trivializations_and_connections}
We start by investigating principal connections on trivial principal bundles. 

Hence, throughout this subsection we have $P_f:=\Sigma\times_M X$, where the fibered product is with respect to some map $f:X\to M$; $P$ is a principal $\Sigma$-bundle over $X$ with quotient map $\pi$ given by the second projection $\mathrm{pr}_2$ and momentum map $\mu$ given by the target map of $\Sigma$ composed with the first projection $t\circ \mathrm{pr}_1$. We denote by $\sigma_f$ the section
\[
x\mapsto (1_{f(x)},x).
\]

Let $\theta$ be a principal connection on $P_f$. First of all, the condition  $\theta\circ a_P=\id$ can be written as
\[
\theta|_{\ker d\mathrm{pr}_2} = \mathrm{pr}_1^*\omega_\mathrm{MC};
\]
alternatively, $\theta|_{\ker d\mathrm{pr}_2} = \mathrm{pr}_1^*\omega|_{\ker(ds)}$. In particular, the form
\[
\hat{\eta} := \theta - \mathrm{pr}_1^*\omega\in \Omega^1(P,\mu^*A)
\]
turns out to be horizontal -- that is, it vanishes along vertical vectors; hence, it descends to  
\[
\overline{\eta} \in \Gamma(\mathrm{Hom}(\pi^*TX,\mu^*A)).
\]
Since $\theta$ is also equivariant -- that is, it is a Cartan connection -- then $\overline{\eta}$ is equivariant as well, and descends to a one form on $X$ valued in $\mu^*A/\Sigma$, which is isomorphic to $f^*A$. Hence, we have a one form
\[
\eta_f\in \Omega^1(X,f^*A).
\]
In fact $\eta_f=\sigma_f^*\theta$, obtained by pulling back $\theta$ along the canonical section $\sigma_f$.
In other words
\begin{proposition}\label{prop:connections_on_trivial_bundle}
    Let $(\Sigma,\omega)$ be a Lie groupoid with (multiplicative) connection, $P_f = \Sigma\times_M X$ be the trivial principal bundle with respect to $f:X\to M$, and $\theta\in \Omega^1(P_f,\mu^*A)$ be a Cartan connection on $P_f$. Then
    \[
    \theta_{(\gamma,x)} = (\mathrm{pr}_1^*\omega)_{(\gamma,x)}+\gamma\cdot(\pi^*\eta_f)_{(\gamma,x)}
    \]
    with $\eta_f:=\sigma_f^*(\theta)$ or, equivalently, $\eta_f$ is the form on $X$ corresponding to the basic form $\hat{\eta} := \theta - \mathrm{pr}_1^*\omega$ on $P$. Conversely, for every $\eta\in \Omega^1(X,f^*A)$, the formula above defines a Cartan connection on $P_f$.
\end{proposition}
We now discuss what happens after a change of trivialization.
\begin{lemma}\label{lem:change_of_section}
    Let $\sigma$ be a section of $P_f$ and let 
    \[
    G:= \mathrm{pr}_1 \circ \sigma,\quad g:=\mu\circ \sigma = t\circ \mathrm{pr}_1 \circ \sigma.
    \]
    Then
    \[
    \eta_g = G^*\omega + G\cdot\eta_f.
    \]
\end{lemma}
\begin{proof}
    Notice that 
    \[
    \sigma = m_P\circ (G,\sigma_f) = G\cdot \sigma_f;
    \]
    hence, by Lemma~\ref{lem:change_of_trivialization} $(G,g)$ is the change of trivialization from $\sigma_f$ to $\sigma$.

    Moreover, we get 
    \[
    \sigma^*\theta = (G,\sigma_f)^* m_P^*\theta;
    \]
    using the multiplicativity equation for $\theta$ and the identity
    \[
    \theta_{(\gamma,x)} = (\mathrm{pr}_1^*\omega)_{(\gamma,x)}+\gamma\cdot(\pi^*\eta_f)_{(\gamma,x)}
    \]
    the claim follows:
    \[
    \sigma^*\theta = \eta_g = G^*\omega + G\cdot \eta_f.
    \]
\end{proof}
Combining the Lemma above together with the correspondence between isomorphism classes principal bundles and $\Sigma$-valued cocycles up to equivalence from Proposition~\ref{prop:principal_bundles_are_cocycles}, we get the following result.
\begin{proposition}\label{prop:cocycles_with_connection}
    Let $(\Sigma,\omega)$ be a Lie groupoid equipped with a (multiplicative) connection. There is a one to one correspondence between isomorphism classes of principal $\Sigma$-bundles with Cartan connection $\theta$ and maximal $\Sigma$-valued cocycles $(U_i,f_i,\varphi_{ij})_{i,j\in I}$ equipped with
    \[
    \eta_i\in \Omega^1(U_i,f_i^*A)
    \]
    satisfying
    \[
    \eta_j|_{U_i\cap U_j} =  \varphi_{ij}^*\omega + \varphi_{ij}\cdot\eta_i.
    \]
\end{proposition}
\section{The Weil algebra of a Lie algebroid with connection}\label{sec:weil}

\subsection{The bowtie construction}\label{subs:bowtie}
The presence of a connection $D$ on a Lie algebroid $A$ allows to introduce the adjoint $A$-connection $\nabla$ from Proposition~\ref{prop:adj_rep} on the complex $\rho:A\to TM$. The operators $D$ and $\nabla$ can be combined to obtain a bracket on $TM\oplus A$.

\begin{proposition}\label{prop:adj_alg}
    Let $(A,D)$ be a Lie algebroid with connection. By setting, for all $(X,\alpha)$, $(Y,\beta)$ in $\Gamma(TM\oplus A)$
    \begin{equation}\label{eq:anchor_bowtie}
    \rho_\oplus(X,\alpha) = X+\rho(\alpha)
    \end{equation}
    and 
    \begin{equation}\label{eq:bracket_bowtie}
    [(X,\alpha),(Y,\beta)]_\oplus = ([X,Y]+\nabla_\alpha Y-\nabla_\beta X,[\alpha,\beta]+D_X \beta - D_Y \alpha)
    \end{equation}
    the triple $TM\bowtie A = (TM\oplus A, \rho_\oplus,[\ ,\ ]_\oplus)$ becomes an almost Lie algebroid such that both the inclusions are almost Lie algebroid maps.  
\end{proposition}
\begin{proof}
    Both the Leibniz identity and the fact that $\rho_\oplus$ is compatible with the bracket follow by direct inspection.
\end{proof}

\begin{proposition}\label{prop:adj_alg_Jacobi}
    Let $u_k=(X_k,\alpha_k)\in\Gamma(TM\oplus A)$ for $k=1,2,3$. The
    Jacobiator of $TM\bowtie A$ is
    \begin{equation}\label{eq:bowtie_jacobiator1}
        \mathrm{Jac}(u_1,u_2,u_3)\ =\ \bigl(-\rho(\mathcal{J}(u_1,u_2,u_3)),\ \mathcal{J}(u_1,u_2,u_3)\bigr),
    \end{equation}
    where
    \begin{equation}\label{eq:bowtie_jacobiator2}
        \mathcal{J}\ =\ -\sum_{\mathrm{cyc}}
        \Bigl(\curv^D(X_1,X_2)\alpha_3+\cbas^D(\alpha_1,\alpha_2)X_3\Bigr),
    \end{equation}
    and the sum runs over the cyclic permutations of the indices. In particular the
    bracket satisfies the Jacobi identity if and only if $(A,D)$ is a flat
    algebroid.
\end{proposition}
\begin{proof}
    Because of the Leibniz identity, $\mathcal{J}$ is completely determined by its values on sections of $A$ and $TM$; moreover, both sides
    of~\eqref{eq:bowtie_jacobiator1} are multilinear. On three sections of $TM$ the Jacobiator $\mathrm{Jac}$ coincides with the Jacobiator of vector fields and it identically vanishes; the same holds true for three sections of $A$. 

    Before handling mixed terms, notice that by Proposition~\ref{prop:adj_alg} $\rho_\oplus$ is
    compatible with the bracket and hence it carries $\mathrm{Jac}$
    to the Jacobiator of vector fields; it follows that 
    \[
    \rho \circ \mathrm{pr}_2 \circ \mathrm{Jac} = - \mathrm{pr}_1\circ \mathrm{Jac}.
    \]

    Since
    $[(X,0),(Y,0)]_\oplus=([X,Y],0)$ and
    $[(Z,0),(0,\gamma)]_\oplus=(-\nabla_\gamma Z,\,D_Z\gamma)$, 
    \[
    \mathcal{J}\bigl((X,0),(Y,0),(0,\gamma)\bigr) = -\curv^D(X,Y)\gamma.
    \]

    Moreover, since $[(0,\beta),(0,\gamma)]_\oplus=\bigl(0,[\beta,\gamma]\bigr)$ and $[(W,\delta),(0,\gamma)]_\oplus=\bigl(-\nabla_\gamma W,\ [\delta,\gamma]+D_W\gamma\bigr)$, we have 
    \[
    \mathcal{J}\bigl((X,0),(0,\beta),(0,\gamma)\bigr) = -\cbas^D(\beta,\gamma)X ,
    \]
    by Definition~\ref{def:inf_basic_curvature}.
\end{proof}
The statement above can be understood within the framework of matched pairs, introduced in~\cite{Lu97} and further studied in~\cite{Mok97} and~\cite{Mackenzie2005}. A matched pair of Lie agebroids $E$ and $F$ is the datum of an $A$ connection $\nabla^F$ on $F$ and an $F$-connection $\nabla^E$ on $E$ satisfying certain compatibility conditions, for which we refer to~\cite{Mok97}; when $E=A$, $F=TM$, and we are given an $A$-connection $\nabla$ on $TM$ and a connection $D$ on $A$, the conditions boil down to $D$ being infinitesimally multiplicative, $\nabla$ being the adjoint connection induced by $D$, and $D$ being flat. On the other hand, in~\cite{Mok97} it is also shown that $\nabla^F$ and $\nabla^E$ form a matched pair if and only if the bracket on $E\oplus F$ defined by
    \begin{equation*}
    [(e_1,f_1),(e_2,f_2)] = ([e_1,e_2]^E+\nabla^E_{f_1} e_2-\nabla^E_{f_2} e_1,[e_1,e_2]^F+\nabla^F_{e_1} f_2-\nabla^F_{e_2} f_1)
    \end{equation*}
satisfies the Jacobi identity.

Below, an element of $\Gamma(\Lambda^pT^*M\otimes \Lambda^qA^*)$ is identified with the element of $\Lambda^{p+q}(TM\oplus A)^*$ evaluated on base directions first -- i.e. its value on
    \[
        (X_1,\dots,X_p,\alpha_1,\dots,\alpha_q),
        \qquad X_i\in TM,\quad \alpha_j\in A,
    \]
    is the corresponding component. 

For later use, we need to describe the Chevalley-Eilenberg operator
\[
\Omega^\bullet(TM\bowtie A) \to \Omega^{\bullet+1}(TM\bowtie A)
\]
on $TM\bowtie A$. By~\cite[Prop.~4.6]{LaurStienXu2010}, if $TM\bowtie A$ is a Lie algebroid the Chevalley-Eilenberg differential can be presented as the differential of a double complex that projects to both $(\Omega^\bullet(M),d_{\mathrm{dR}})$ and $(\Omega^\bullet(A),d_{\mathrm{CE}})$. Here, $d_\mathrm{dR}$ denotes the de Rham differential on $\Omega^\bullet(M)$ and $d_\mathrm{CE}$ denotes the Chevalley-Eilenberg differential on $\Omega^\bullet(A)$. 
   
More precisely: 
\begin{itemize}
    \item the connection $D$ on $A$ induces a dual connection $D^*$ on $A^*$ and on all the exterior powers $\Lambda^pA^*$, hence for all $p\in \mathbb{N}$ one has operators 
    \[
    d_{D^*}: \Omega^\bullet(M,\Lambda^p A^*) \to \Omega^{\bullet+1}(M,\Lambda^p A^*), 
    \]
    where, for $p=0$, we set $d_{D^*} = d_\mathrm{dR}$, the de Rham differential on $M$;
    \item similarly the $A$-connection $\nabla$ on $TM$ induces a dual $A$-connection $\nabla^*$ on $T^*M$ and on all the exterior powers $\Lambda^qT^*M$, hence for all $q\in \mathbb{N}$ one has operators 
    \[
    d_{\nabla^*}: \Omega^\bullet(A,\Lambda^q T^*M)\to \Omega^{\bullet+1}(A,\Lambda^q T^*M),
    \]
    where, for $q=0$, we set $d_{\nabla^*} = d_\mathrm{CE}$, the Chevalley-Eilenberg differential of $A$;
    \item the operator ${d}_\oplus := d_{D^*} + (-1)^p d_{\nabla^*}$,
    \[
    {d}_\oplus:
    \Gamma(\Lambda^pT^*M\otimes\Lambda^qA^*) \to \Gamma(\Lambda^{p+1}T^*M\otimes\Lambda^qA^*)\oplus \Gamma(\Lambda^{p}T^*M\otimes\Lambda^{q+1}A^*),
    \]
    which induces an operator
    \[
    d_\oplus: \Omega^\bullet(TM\bowtie A) \to \Omega^{\bullet+1}(TM\bowtie A),
    \]
    which is precisely the Chevalley-Eilenberg operator of $TM\bowtie A$.
\end{itemize}
Putting all together, we obtain a bi-graded algebra structure on $\Omega^\bullet(TM\bowtie A)$ 
and the operator
\begin{equation}\label{eq:diff_adjoint_algebroid}
{d}_\oplus := d_{D^*} + (-1)^p d_{\nabla^*}: \Omega^\bullet(TM\bowtie A)\to \Omega^{\bullet+1}(TM\oplus A),
\end{equation}
which satisfies the Leibniz identity and 
comes with the maps of complexes
\[
\Omega^\bullet(M)\overset{p_L}{\longleftarrow} \Omega^\bullet(TM\bowtie A) \overset{p_R}{\longrightarrow} \Omega^\bullet(A)
\]
dual to the inclusions
\[
TM\overset{i_L}{\longrightarrow} TM\oplus A \overset{i_R}{\longleftarrow} A
\]

\begin{remark}\label{rmk:jacobi_via_d_oplus}
Notice that, since an almost Lie algebroid satisfies the Jacobi identity if and only if its Chevalley-Eilenberg operator squares to zero, the remarks above can be used to give a proof of Proposition~\ref{prop:adj_alg_Jacobi}, by using Lemma~\ref{lem:dplus_square_curv} below.
\end{remark}

\begin{remark}\label{rmk:isotropy_adjoint}[The isotropy of $TM\bowtie A$]
The almost Lie algebroid $TM\bowtie A$ is transitive. Its isotropy vector bundle is isomorphic to the vector bundle $A$ equipped with a structure of bundle of almost Lie algebras $(A,\{\ ,\ \})$. Explicitly, the isomorphism is given by
\[
\alpha\in A \mapsto (-\rho(\alpha),\alpha)\in TM\oplus A
\]
and, given a point $x\in M$, the bracket on $A_x$ is
 \[
    \{\ ,\ \}_x: A_x\times A_x \to A_x,\quad (\alpha_x,\beta_x) \to - D_{\rho(\alpha)} \beta + D_{\rho(\beta)}\alpha + [\alpha,\beta],
\]
where $\alpha,\beta\in \Gamma(A)$ are any extensions of $\alpha_x,\beta_x$. This is precisely the bracket introduced in the proof of Proposition~\ref{prop:flat_Cartan_action}. 


\end{remark}

\subsection{The Weil algebra}\label{subs:weil_algebra}

Here, we discuss the Weil algebra of a Lie algebroid $A$ over $M$. It was introduced in~\cite{AbCr2009Weil}; the description in terms of a connection $D$ on $A$, which is the one we present, is that of~\cite{AbCr2011RepUpHomAlg}. The main difference is presentation is a different choice of bidegree (the total degree stays the same): for instance, $\Gamma(S^1A^*)$ has bidegree $(0,2)$ below and $(1,1)$ in~\cite[Definition 3.8]{AbCr2009Weil}.

The Weil algebra $W^\bullet(A,D)$ is the bi-graded algebra
\begin{equation}\label{eq:Weil_alg_KT}
W^{r,s}(A,D) = \bigoplus\limits_{\substack{u+w = r,\ 2v = s}} \Gamma(\Lambda^uT^*M \otimes S^v(A^*)\otimes \Lambda^w(A^*)) 
\end{equation}
which is isomorphic to 
\[
\bigoplus\limits_{\substack{2v = s}} \Gamma(\Lambda^r(TM\oplus A)^*\otimes S^v(A^*))
\]
and admits as sub-algebra the algebra $\Omega^\bullet(TM\bowtie A)$ of linear forms on the almost Lie algebroid $TM\bowtie A$ -- Proposition \ref{prop:adj_alg}. Additionally, $W^\bullet(A,D)$ can be given the structure of differential graded algebra with grading given by the sum $r+s$. 

Before describing the differential we observe that there is a tautological isomorphism $\iota:\Lambda^1A^*\to S^1A^*$ -- the identity on $A^*$. In what follows, we will use the same notation for its inverse. Moreover, we will extend $\iota$ to a map defined on $\mathrm{Hom}(\Lambda^k(TM\oplus A),\Lambda^1A^*)$. That is, we consider
\begin{equation*}
\iota: \mathrm{Hom}(\Lambda^k(TM\oplus A), \Lambda^1A^*)\to \mathrm{Hom}(\Lambda^k(TM\oplus A), S^1A^*)
\end{equation*}
 characterized by the formula
\begin{equation}\label{eq:tautological_shift}
\iota(\eta\otimes \alpha) = (-1)^{k}\eta\otimes \iota(\alpha),\quad \eta\in \Lambda^k(TM\oplus A)^*,\ \alpha\in \Lambda^1A^*
\end{equation}
together with $\mathbb{R}$-linearity; concisely,
\[
\iota(\xi) = (-1)^k\cdot \iota\circ \xi, \quad \xi\in \mathrm{Hom}(\Lambda^k(TM\oplus A), \Lambda^1A^*).
\]
\begin{remark}
    Notice that $\iota$ acts on the output of an element of $\mathrm{Hom}(\Lambda^k(TM\oplus A),\Lambda^1A^*)$. When presenting $\mathrm{Hom}(\Lambda^k(TM\oplus A),\Lambda^1A^*)$ as $\Lambda^k(TM\oplus A)^*\otimes A^*$, by convention we understand $\iota$ as acting on the last copy of $A^*$.
\end{remark}

With this in mind, the differential $d_W$ on the Weil algebra is given by
    \begin{equation}\label{eq:diff_Weil}
        d_W: W^\bullet(A,D) \to W^{\bullet+1}(A,D)
    \end{equation}
    acting as 
    \[
    d_W f = d_\mathrm{dR} f + d_\mathrm{dR} f\circ \rho
    \]
    on functions $f\in C^\infty(M)$ and as the following sum of three derivations given on generators  $\eta\in \Omega^1(M)$, $\alpha\in \Omega^1(A)$, and $a\in \Gamma(S^1A^*)$.
    \begin{itemize}
        \item $d_{\iota}:W^{r,s}(A,D)\to W^{r-1,s+2}$ such that
        \begin{equation}\label{eq:diff_shift}
        d_{\iota}(\eta) = -\iota(\eta\circ\rho),\ d_{\iota}(\alpha) = \iota(\alpha),\ d_{\iota}(a) = 0.
        \end{equation}
        \item $d_\mathrm{cad}: W^{r,s}(A,D) \to W^{r+1,s}(A,D)$ such that
        \begin{equation}\label{eq:diff_ad}
            d_\mathrm{cad}(\eta) = d_\oplus(\eta),\ d_\mathrm{cad}(\alpha) = d_\oplus(\alpha),\ d_\mathrm{cad}(a) = -\iota\left((d_{\nabla^{A^*}}+d_{D^*})(\iota(a))\right).
        \end{equation}
        \item $d_\mathrm{curv}:W^{r,s}(A,D)\to W^{r+3,s-2}$ such that
        \begin{equation}\label{eq:diff_curv}
        d_\mathrm{curv}(\eta) = 0,\ d_\mathrm{curv}(\alpha) = 0,\ d_\mathrm{curv}(a) = \langle a, \curv^D \rangle + \langle a, \cbas^D \rangle.
        \end{equation}
    \end{itemize}
    Here $d_{\nabla^{A^*}}$ is the Chevalley-Eilenberg differential of $A$ with coefficients in $A^*$, taken with respect to the dual of the adjoint $A$-connection $\nabla^A$ of Proposition~\ref{prop:adj_rep},
    \[
    (\nabla^{A^*}_\alpha \xi)(\beta) = \rho(\alpha)\left(\xi(\beta)\right) - \xi\left(\nabla^A_\alpha\beta\right),\quad \alpha,\beta\in \Gamma(A),\ \xi\in\Gamma(A^*).
    \]
 \begin{remark}
     Locally, on open subsets $U\subset M$ supporting coordinates charts, the Weil algebra is freely generated as a dga by $C^\infty(U)$ and $\Omega^1(A|_U)$; indeed, one can express elements of $\Omega^1(U)$ and $\Gamma(S^1A^*|_U)$ in terms of the former and their differentials; one checks, by a direct computation on generators, that the formulas given above do define a dga. 

     Since below we often need formulas involving elements in $\Gamma(S^1A^*)$, we mostly work with the set of generators consisting of $C^\infty(M)$, $\Omega^1(A)$, $\Omega^1(M)$, and $\Gamma(S^1A^*)$.
 \end{remark}

    While $d_\iota$ squares to $0$, $d_\mathrm{cad}$ does so if and only if $d_\mathrm{curv}$ vanishes identically -- in fact, we can describe the relation between $d_\mathrm{cad}$ and $d_\mathrm{curv}$ rather precisely. First of all

\begin{lemma}\label{lem:relation_differentials}
    For $\eta\in \Omega^1(M)$ and $\alpha\in \Omega^1(A)$
    \begin{align}
        d_W(\eta) &= d_\oplus(\eta) -\iota(\eta\circ \rho),\label{eq:dW_on_eta}\\
        d_W(\alpha) &= d_\oplus(\alpha)+\iota(\alpha).\label{eq:dW_on_alpha}
    \end{align}
\end{lemma}
\begin{proof}
    By~\eqref{eq:diff_curv} the term $d_\mathrm{curv}$ vanishes in both cases, by~\eqref{eq:diff_ad} the term $d_\mathrm{cad}$ is $d_\oplus$ in both cases, and the remaining terms are those in~\eqref{eq:diff_shift}.
\end{proof}

\begin{lemma}\label{lem:dplus_square_curv}
    The following identity holds:
    \[
    d_\oplus^2(\alpha) = -d_\mathrm{curv}d_\iota(\alpha),\quad \alpha\in \Omega^1(A).
    \]
\end{lemma}
\begin{proof}
    We have that
    \[
    d_\oplus^2(\alpha) = d_{D^*}^2(\alpha) + d_{D^*}d_\mathrm{CE}(\alpha)-d_{\nabla^*}d_{D^*}(\alpha).
    \]

    The statement follows from the two identities
    \begin{equation}\label{eq:flat_from_diff}
    d_{D^*}^2(\alpha) = - \langle \alpha, \curv^D \rangle.
    \end{equation}
    and
    \begin{equation}\label{eq:mult_from_diff}
    d_{D^*}d_\mathrm{CE}(\alpha)-d_{\nabla^*}d_{D^*}(\alpha) = -\langle \alpha, \cbas^D \rangle
    \end{equation}

    We begin with proving~\eqref{eq:flat_from_diff}, which is standard. The $3$-form $d_{D^*}(d_{D^*}(\alpha))$ is the differential of $d_{D^*}(\alpha)$ with respect to the connection $D^*$ on $A^*$ (the dual connection of $D$); hence
    \[
    d_{D^*}^2(\alpha)(X,Y) = D_X^*(d_{D^*}\alpha(Y))-D_Y^*(d_{D^*}\alpha(X)) - d_{D^*}\alpha([X,Y])
    \]
    for all vector fields $X,Y$ on $M$. When pairing the above expression with some $c\in \Gamma(A)$, we obtain
    \begin{align*}
        L_X(\langle c, d_{D^*}\alpha(Y) \rangle) - \langle D_X c, d_{D^*}\alpha(Y) \rangle\\
        L_Y(\langle c, d_{D^*}\alpha(X) \rangle) - \langle D_Y c, d_{D^*}\alpha(X) \rangle\\
        -L_{[X,Y]} \langle c, \alpha \rangle + \langle D_{[X,Y]} c, \alpha\rangle
    \end{align*}
    simply by using the Leibniz rule for $L$, $D$, and the dual connection $D^*$ on the evaluation $\langle\ ,\ \rangle$:
    \[
    L_Z\langle e, f \rangle = \langle D_Z e, f \rangle +\langle e, D^*_Zf \rangle,\quad Z\in \mathcal{X}(M),\ e,f\in \Gamma(A).
    \]
    Using this same formula again yields
    \[
    \langle c, d_{D^*}^2(\alpha)(X,Y) \rangle = \langle -D_XD_Yc + D_YD_Xc + D_{[X,Y]}c, \alpha \rangle
    \]
    which, by definition, is
    \[
    \langle  -\curv^D(X,Y)(c), \alpha \rangle
    \]
    proving~\eqref{eq:flat_from_diff}.

    To prove~\eqref{eq:mult_from_diff}, we start by computing $d_{D^*}d_\mathrm{CE}(\alpha)$. Let $X\in \mathcal{X}(M)$, $b,c\in \Gamma(A)$. Then
    \[
    (d_{D^*}d_\mathrm{CE}(\alpha))(X,b,c) = (D^*_Xd_\mathrm{CE}\alpha)(b,c),
    \]
    which, using again the Leibniz rule for $L$, $D$, and $D^*$ on the evaluation, is equal to
    \[
    L_X[d_\mathrm{CE}\alpha(b,c)] -d_\mathrm{CE}\alpha(D_Xb,c)-d_\mathrm{CE}\alpha(b,D_Xc).
    \]
    Expanding $d_\mathrm{CE}\alpha$ and using the definition of $\nabla$ and $\cbas^D$, we find the expression above to be equal to 
    \begin{align*}
    -\langle \alpha, (\cbas^D(b,c)(X))-D^*_X\alpha([b,c]) -D^*_{\nabla_bX}\alpha(c)\\ +D^*_{\nabla_cX}\alpha(b)+L_{\rho(b)}(D^*_X\alpha(c))-L_{\rho(c)}(D^*_X\alpha(b))
    \end{align*}
    We also have that $D^*_Y\alpha(d) = (d_{D^*}\alpha(d))(Y)$ for all $Y\in \mathcal{M}$ and $d\in \Gamma(A)$, hence we can manipulate the expression 
    \[
    D^*_X\alpha([b,c]) -D^*_{\nabla_bX}\alpha(c)+D^*_{\nabla_cX}\alpha(b)+L_{\rho(b)}(D^*_X\alpha(c))-L_{\rho(c)}(D^*_X\alpha(b))
    \]
    and obtain
    \[
    \begin{gathered}
    (d_{D^*}\alpha([b,c]))(X)-(d_{D^*}\alpha(c))(\nabla_bX)+(d_{D^*}\alpha(b))(\nabla_cX)\\
    +L_{\rho(b)}(d_{D^*}\alpha(c))-L_{\rho(c)}(d_{D^*}\alpha(b))
    \end{gathered}
    \]
    Using the Leibniz rule for $\nabla$, $\nabla^*$, and the evaluation, the last expression is equal to 
    \[
    (\nabla^*_b(d_{D^*}\alpha(c)))(X) -(\nabla^*_c(d_{D^*}\alpha(b)))(X) -(d_{D^*}\alpha([b,c]))(X)
    \]
    which is nothing but
    \[
    (d_{\nabla^*}d_{D^*}(\alpha))(b,c,X).
    \]
    Putting all together, we have 
    \[
    d_{D^*}d_\mathrm{CE}\alpha(X,b,c) = d_{\nabla^*}d_{D^*}\alpha(b,c,X) -\langle \alpha, \cbas^D(b,c,X) \rangle 
    \]
    concluding the proof.
\end{proof}

\section{The connection-curvature construction}\label{sec:construction}

In this section we explain how to construct characteristic classes living in the base of a principal $(\Sigma,\omega)$-bundle. Our approach is closely inspired to Kamber and Tondeur's treatment of classical Chern-Weil theory from~\cite{KamTond1975}. The classical theory is recovered from ours by observing that, if $G$ is a Lie group, a principal connection $\theta$ on a principal $G$-bundle $P$ makes $(P,\theta)$ into a principal $(G,\omega_\mathrm{MC})$-bundle, where $\omega_\mathrm{MC}$ is the Maurer-Cartan form.

\subsection{The map $\Theta^*:\Omega^\bullet(TM\bowtie A)\to \Omega^\bullet(P)$}\label{subs:Theta_star}
We begin by constructing the groupoid version of the algebra map
\[
\theta^*:C^\bullet(\mathfrak{g})\to \Omega^\bullet(P)
\]
associated to a connection $\theta$ on a principal $G$-bundle for some Lie group $G$ with Lie algebra $\mathfrak{g}$. Here, we do not need the connection $\omega$ to be multiplicative. In fact, it is sufficient to work at the infinitesimal level: thus, we fix a Lie algebroid $A$ and a connection $D$ on the vector bundle $A$. We will use $\nabla$ to denote the adjoint $A$-connection induced by $D$ on the complex $A\overset{\rho}{\to} TM$.

Rather than defining a map on $\Omega^\bullet(A)$, we will work on the almost algebroid $TM\bowtie A$ from Proposition~\ref{prop:adj_alg}. Recall that the underlying vector bundle of $TM\bowtie A$ is $TM\oplus A$, the anchor is $\mathrm{id}+\rho$, and the bracket on $\Gamma(TM\bowtie A)$ is obtained from the connection $D$ and $\nabla$ in such a way that both $TM$ and $A$ are subalgebroids. $TM\bowtie A$ is itself a Lie algebroid if and only if $D$ is infinitesimally multiplicative flat.

Let $(P,\theta)$ be a pair consisting of a principal $\Sigma$-bundle and a connection one-form $\theta\in \Omega^1(P,\mu^*A)$. Recall that, given a surjective submersion $f:P\to B$, and an anchored vector bundle $E\to B$, $\rho_E:E\to TB$, a one form $\eta\in \Omega^1(P,f^*E)$ is said to be \textbf{anchored} if
\[
\rho_E\circ \eta = df.
\]
Equivalently, $\eta$ defines a map of vector bundles anchored over $P$
\[
\begin{tikzcd}
TP \arrow{d}\arrow["\eta"]{r}& E \arrow{d}\\
P  \arrow["f"]{r}& B
\end{tikzcd}.
\]
and in turn such a map induces (and is equivalently described in terms of) a map of vector bundles
\[
\begin{tikzcd}
TP \arrow{d}\arrow["\hat{\eta}"]{r}& f^*E \arrow{d}\\
P  \arrow["f"]{r}& P
\end{tikzcd}.
\]
In general, a connection $\theta\in \Omega^1(P,\mu^*A)$ on a principal $\Sigma$-bundle $P$ is not anchored: $\rho\circ \theta\neq d\mu$. However, it induces a canonical anchored one form with values in $TM\oplus A$.
\begin{definition}\label{def:anch_form_conn}
Given a connection $\theta$ on $P$, \textbf{the anchored $1$-form $\Theta\in \Omega^1(P,TM\oplus A)$ associated to $\theta$} is defined by setting
\[
\Theta := (\an_\theta,\theta).
\]
here, 
$\an_\theta:=d\mu-\rho\circ \theta$.
\end{definition}

For later use, we immediately prove the following
\begin{lemma}\label{lem:anchored_form_is_equiv}
    Let $(\Sigma,\omega)$ be a groupoid with Cartan connection and $\theta$ be a principal $(\Sigma,\omega)$-connection on $P$. The associated anchored one-form $\Theta$ is $\Sigma$-equivariant, i.e.
    \[
    g\cdot \Theta(w) = \Theta(g\cdot w),\quad w\in TP
    \]
    where
    \begin{itemize}
        \item the representation on the left hand side is the adjoint representation of $\Sigma$ on the complex $A\overset{\rho}{\longrightarrow} TM$ induced by the Cartan connection $\omega$;
        \item the representation on the right hand side is the adjoint representation of $\Sigma\ltimes P$ on the complex $\mu^*A \overset{a_P}{\longrightarrow} TP$ induced by the Cartan connection $\pr^*_1\omega$.
    \end{itemize}
\end{lemma}
\begin{proof}
    The second component of $\Theta$ is $\theta$, which is $\Sigma$-equivariant by definition. We need to show that $\an_\theta = d\mu-\rho\circ \theta$ is $\Sigma$-equivariant as well.
    
    First, notice for all $(g,p)\in \Sigma\ltimes P$ and $w\in T_pP$, we have
    \[
    g\cdot w := (g,p)\cdot w = dm_P (\hor_g^\omega (d\mu(w)),w).
    \]
    On the other hand, for all
    $x\in M$, $g\in \Sigma$ with $s(g) = x$ and $v\in T_xM$, one has
    \[
    g\cdot v = dt\left(\hor_g^\omega (v)\right).
    \]
    We deduce that 
    \[
    d\mu (g\cdot w) = g\cdot d\mu(w)
    \]
    for all $(g,p)\in \Sigma\ltimes P$ and $w\in T_pP$ -- i.e. $\mu$ is $\Sigma$-equivariant. Since both $\theta$ and $\rho$ are $\Sigma$-equivariant as well,
    \[
    \an_\theta(g\cdot w) = d\mu(g\cdot w)-\rho\bigl(\theta(g\cdot w)\bigr)
    = g\cdot d\mu(w)-g\cdot\rho\bigl(\theta(w)\bigr) = g\cdot \an_\theta(w) .
    \]
\end{proof}

Given the anchored form $\Theta$, we have the corresponding map
\[
\begin{tikzcd}
TP \arrow{d}\arrow["\Theta"]{r}& TM\oplus A \arrow{d}\\
P  \arrow["\mu"]{r}& M
\end{tikzcd}
\]
and the lift 
\[
\begin{tikzcd}
TP \arrow{d}\arrow["\hat{\Theta}"]{r}& TP\oplus \mu^*A \arrow{d}\\
P  \arrow["\id"]{r}& P
\end{tikzcd}
\]
defined by
\[
\hat{\Theta}: v\in T_pP \mapsto \left(v-a_P\circ \theta(v), \theta(v)\right)
\]
and satisfying
\[
(d\mu,\id) \circ \hat{\Theta} = \Theta.
\]
\begin{remark}\label{rmk:image_hat_Theta}
	The image of $\hat{\Theta}$ is the vector sub-bundle
	\[
	\ker(\theta)\oplus\mu^*A \subset TP\oplus \mu^*A.
	\]
\end{remark}
Notice that $TP\oplus \mu^*A$ is an anchored vector bundle over $TP$, with anchor given by
\[
TP\to TP\oplus \mu^*A,\quad  v \mapsto (v,a_P(v)).
\]
In fact, we can endow $TP\oplus \mu^*A$ with a structure of almost algebroid as in Proposition~\ref{prop:adj_alg} by using the pull-back connection 
\[
D^\mu:=\mu^*D: \mathcal{X}(P)\times \Gamma(\mu^*A) \to \Gamma(\mu^*A).
\]
\begin{proposition}\label{prop:anchored_form_connection}
	 For all $X,Y\in \mathcal{X}(P)$ it holds
		\[
		[\hat{\Theta}(X),\hat{\Theta}(Y)] -\hat{\Theta}\left([X,Y]\right)  = \left(-a_P \circ \mc_\theta(X,Y), \mc_\theta(X,Y)\right),
		\]
        hence
        \[
         [{\Theta}(X),{\Theta}(Y)] - {\Theta}\left([X,Y]\right) = \left( -\rho\circ \mc_\theta(X,Y), \mc_\theta(X,Y)\right).
        \]
		In particular, the obstructions for $\Theta$ to be compatible with the Lie brackets of $TP$ and $TM\bowtie A$ takes values in the isotropy bundle of $TM\bowtie A$. 
\end{proposition}

\begin{proof}
    It is enough to prove the formula involving $\hat{\Theta}$, since the other one follows from
    \[
    (d\mu,\hat{\mu})\circ \hat{\Theta} = \Theta.
    \]
	
	We first compute $[\hat{\Theta}(X),\hat{\Theta}(Y)]$, and get
	\begin{align*}
	\left([X-a_P\circ\theta(X),Y-a_P\circ \theta(Y)] -\nabla^\mu_{\theta(Y)}(X-a_P\circ \theta(X)) +\nabla^\mu_{\theta(X)}(Y-a_P\circ \theta(Y)) \right.,\\
	\left. [\theta(X),\theta(Y)]+D^\mu_{X-a_P\circ\theta(X)}\theta(Y)-D^\mu_{Y-a_P\circ\theta(Y)}\theta(X)\right).
	\end{align*}
	Then, we compute $\hat{\Theta}\left([X,Y]\right)$, obtaining
	\[
	 \left([X,Y]-a_P\circ \theta([X,Y]),\theta([X,Y])\right)
	\] 
	
	We claim that, subtracting these two terms, we get
	\[
	\left(-a_P \circ \mc_\theta, \mc_\theta\right)
	\]
	By direct calculation, for $ [\hat{\Theta}(X),\hat{\Theta}(Y)]-\hat{\Theta}\left([X,Y]\right) $ we have that the $TP$-component is given by
		\begin{align*}
		[a_P\circ\theta(X),a_P\circ\theta(Y)]-[a_P\circ\theta(X),Y]  -[X,a_P\circ\theta(Y)] \\
		 -\nabla^\mu_{\theta(Y)}(X-a_P\circ \theta(X)) +\nabla^\mu_{\theta(X)}(Y-a_P\circ \theta(Y)) +a_P\circ \theta([X,Y]).
		\end{align*}
	By using the identity
	\[
	\nabla^\mu_{\theta(V)}(W) = [a\circ \theta(V),W]+a_P\circ D^\mu_W \theta(V),\quad V,W\in \mathcal{X}(P)
	\]
	on the above expressions, we get
	\begin{align*}
	+[a_P\circ\theta(X),a_P\circ\theta(Y)] +a_P\circ D^\mu_Y\theta(X) -a_P \circ D^\mu_X\theta(Y) \\
	+a_P\circ D^\mu_{a_P\circ \theta(X)}\theta(Y) -a_P\circ D^\mu_{a_P\circ \theta(Y)}\theta(X)-2[a_P\circ \theta(X), a_P\circ \theta(Y)] +a_P\circ \theta([X,Y]).
	\end{align*}
	which, using that $a_P$ is compatible with the bracket, gives
	\begin{align*}
	a_P\circ \left( D_Y\theta(X)-D^\mu_X\theta(Y) + D^\mu_{a_P\circ \theta(X)}\theta(Y) - D^\mu_{a_P\circ \theta(Y)}\theta(X)+\theta[X,Y] - [\theta(X),\theta(Y)]\right).
	\end{align*}
	We now observe that the expression within brackets is precisely minus the $\mu^*A$-component of $[\hat{\Theta}(X),\hat{\Theta}(Y)] - \hat{\Theta}\left([X,Y]\right)$, consistently with the $TP$-component being $-a_P\circ\mc_\theta(X,Y)$.
    The fact that such expression is precisely $-\mc_\theta$ follows since, by definition,
	\[
	\theta[X,Y]-D^\mu_X\theta(Y)+D^\mu_Y\theta(X) = - d_{\mu^*D} \theta
	\]
	and
	\[
	D^\mu_{a_P\circ \theta(X)}\theta(Y) - D^\mu_{a_P\circ \theta(Y)}\theta(X) - [\theta(X),\theta(Y)] = -\frac{1}{2}\{\theta,\theta\}(X,Y).
	\]
\end{proof}

\begin{corollary}\label{cor:flatness_anchored}
	Let $\omega$ be a Cartan connection on $\Sigma$, $D$ the induced infinitesimally Cartan connection on the associated Lie algebroid $A$, and $\theta$ a principal connection on $P$. Then the maps
	\[
	\hat{\Theta}: TP \to TP\bowtie \mu^*A, \quad \Theta: TP \to TM\bowtie A
	\]
    from Proposition~\ref{prop:anchored_form_connection} are compatible with the brackets if and only if $\theta$ is flat.
\end{corollary}
\begin{proof}
    By Proposition~\ref{prop:anchored_form_connection}, each of $\hat\Theta$ and $\Theta$ is compatible with the brackets if and only if $\mc_\theta=0$. Since $\theta$ is principal, it is multiplicative, and Proposition~\ref{prop:mc_expr_vanishing_pb} then gives $\mc_\theta=0$ if and only if $\theta$ is flat.
\end{proof}

Recall now that, if $P$ is a principal group bundle, then connections correspond to algebra maps 
\begin{equation}\label{eq:conn_are_alg_maps_grp}
\theta^*:C^\bullet(\mathfrak{g}^*)\to \Omega^\bullet(P),
\end{equation}
which are dga maps when $\theta$ is a flat principal connection. Corollary \ref{cor:flatness_anchored} generalizes this -- see corollary \ref{cor:flat_conn_dga_map} below. 

Thus, let $\Sigma$ be a Lie groupoid with Lie algebroid $A$, and let $D$ be a connection on $A$ and $P$ be a principal $\Sigma$-bundle. Any connection $\theta$ on $P$ induces the map of graded algebras 
\begin{equation}\label{eq:flat_connection_map}
    \Theta^*:\Omega^\bullet(TM\bowtie A)\to \Omega^\bullet(P),\quad (\eta,\alpha) \to \eta \circ \an?{\theta} + \alpha\circ \theta
\end{equation}
where now $\Omega^\bullet(TM\bowtie A)$ comes equipped with the degree one operator 
\[
d_\oplus: \Omega^\bullet(TM\bowtie A) \to C^{\bullet+1}(TM\bowtie A)
\]
given by~\eqref{eq:diff_adjoint_algebroid} and corresponding to the Chevalley-Eilenberg operator of $TM\bowtie A$. 

Recall that $(\Omega^\bullet(TM\bowtie A),d_\oplus)$ is a dga if and only if $TM\bowtie A$ is a Lie algebroid, i.e. when $D$ is multiplicative and flat -- see Proposition \ref{prop:adj_alg_Jacobi}. Then, corollary~\ref{cor:flatness_anchored} implies
\begin{corollary}\label{cor:flat_conn_dga_map}
    Let $(\Sigma,\omega)$ be a Lie groupoid with Cartan connection, $P$ be a principal $\Sigma$ bundle and $\theta$ be a principal $(\Sigma,\omega)$-connection on $P$. The algebra map \eqref{eq:flat_connection_map} satisfies
    \[
    d_\mathrm{dR}\Theta^* = \Theta^*d_\oplus
    \]
    if and only if $\theta$ is flat.
\end{corollary}  
By dualizing Proposition~\ref{prop:anchored_form_connection} and Corollary~\ref{cor:flatness_anchored}, and introducing the notation
\begin{equation}\label{eq:obstruction}
R_\Theta: \Omega^\bullet(TM\bowtie A) \to \Omega^{\bullet+1}(P),\quad   R_\Theta := d_\mathrm{dR}\circ \Theta^* - \Theta^*\circ d_\oplus,
\end{equation}
where $d_{\mathrm{dR}}:\Omega^\bullet(P)\to \Omega^{\bullet+1}(P)$ denotes the de Rham differential on $P$, we have also the following
\begin{corollary}\label{cor:def_obstruction}
The principal connection $\theta$ is flat if and only if 
\[
R_\Theta: \Omega^\bullet(TM\bowtie A) \to \Omega^{\bullet+1}(P)
\]
vanishes identically.
\end{corollary}
    Notice that the value taken by $R_\Theta$ can be expressed in terms of the curvature of $\theta$ via the calculations of Proposition~\ref{prop:anchored_form_connection}; for later use, we spell out some observations in the corollary below.
\begin{corollary}\label{cor:calc_obstruction}
    The following identities hold.
    \begin{itemize}
        \item $R_\Theta(\eta) = -\eta(\rho\circ \mc_\theta)$, $\eta\in \Omega^1(M)\hookrightarrow C^1(TM\oplus A)$;
        \item $R_\Theta(\alpha) = \alpha(\mc_\theta)$, $\alpha\in \Omega^1(A)\hookrightarrow C^1(TM\oplus A)$;
        \item for a degree $k$-cochain of the form
    \[
        \sigma = \sigma^{1}\wedge \dots \wedge \sigma^{k},\quad \sigma^i \in C^1(TM\oplus A),\ 1\leq i\leq k,
    \]
    we have
    \[
    R_\Theta (\sigma) = \sum\limits^{k}_{i=1} (-1)^{i-1}\,\Theta^*(\sigma^1) \wedge \dots \wedge R_\Theta(\sigma^i)\wedge \dots \wedge \Theta^*(\sigma^k).
    \]
    \end{itemize}
\end{corollary}
On cochains of more general form, the value of $R_\Theta$ is determined by $C^\infty(M)$-linearity.

\subsection{The lift $\kappa^{\Theta}:W^\bullet(A,D)\to \Omega^\bullet(P)$}\label{subs:kappa_Theta}
In this subsection we lift the map $\Theta^*$ defined in \eqref{eq:flat_connection_map} to the Weil algebra of the Lie algebroid with connection $(A,D)$. The advantage in doing so is that the resulting map of graded algebras
\[
\kappa^{\Theta}: W^\bullet(A,D) \to \Omega^\bullet(P),
\]
is a map of differential graded algebras.

More precisely, we devote this subsection to prove the following theorem, where $P$ is a principal $\Sigma$-bundle, $D$ is a connection on its Lie algebroid $A$.
\begin{theorem}\label{thm:conn_curv}
    A connection $\theta$ on $P$ induces a map of differential graded algebras
    \[
    \kappa^{\Theta}: (W^\bullet(A,D),d_W)\to (\Omega^\bullet(P), d_\mathrm{dR});
    \]
    such that $\kappa^{\Theta}$ fits into a commutative diagram of graded algebras
        \[
        \begin{tikzcd}
        W^\bullet(A,D)\arrow[r,"\kappa^{\Theta}"] & \Omega^\bullet(P)\\
        \Omega^\bullet(TM\bowtie A)\arrow[u,hook]\arrow[ur,"\Theta^*",swap] & 
        \end{tikzcd}
        \]
        
        Additionally, if $\omega$ is multiplicative, $D$ is the induced infinitesimal connection, and $\theta$ is principal,
        then $\ker(\kappa^{\Theta})$ contains the differential ideal generated by $\Gamma(S^1A^*)$ if and only if $\theta$ is flat.
\end{theorem}
\begin{remark}\label{rmk:conn_curv_restated}
    When $\Sigma$ is a Lie group $G$, a more enlightening version of the theorem above states that a morphism of graded algebras
    \[
    \Lambda^\bullet(\mathfrak{g}) \to \Omega^\bullet(P)
    \]
    corresponds in a canonical way to morphism of dgas
    \[
    W^\bullet(\mathfrak{g}) \to \Omega^\bullet(P).
    \]

    Indeed, it is possible to restate Theorem~\ref{thm:conn_curv} in a similar fashion: commutative diagrams of graded algebras
    \[
    \begin{tikzcd}
        \Omega^\bullet(TM\bowtie A)\arrow[rr] && \Omega^\bullet (P)\\
        &\Omega^\bullet(M)\arrow[ul,"\id+\rho^*"]\arrow[ur,"\mu^*",swap]
    \end{tikzcd}
    \]
    correspond canonically to commutative diagrams of dgas
    \[
    \begin{tikzcd}
        W^\bullet(A,D)\arrow[rr]&& \Omega^\bullet (P)\\
        &\Omega^\bullet(M)\arrow[ul,"\id+\rho^*"]\arrow[ur,"\mu^*",swap].
    \end{tikzcd}
    \]
\end{remark}

We devote the rest of this subsection to proving Theorem~\ref{thm:conn_curv}. First of all, we define $\kappa^{\Theta}$ as a map of graded algebras
\[
\kappa^{\Theta}: W^\bullet(A,D) \to \Omega^\bullet(P).
\]
The idea is to have $\kappa^{\Theta}$ be equal to $\Theta^*$ on the image of $\Omega^1(TM\oplus A)$ and to the "obstruction map" $R_\theta$ from Corollary~\ref{cor:def_obstruction} on $\Gamma(S^1A^*)$. More precisely, $\kappa^{\Theta}$ is defined by the formulas below, where the tautological map $\iota:\Lambda^1A^*\to S^1A^*$ and its extension are discussed in Subsection~\ref{subs:weil_algebra}; see \eqref{eq:tautological_shift} and the discussion around it.
\begin{itemize}
\item On linear one-forms on $TM \oplus A$, $\kappa^{\Theta}$ acts as $\Theta^*$, i.e.:
\begin{itemize}
    \item on $1$-forms $\eta\in \Omega^1(M)$, 
    \[
    \kappa^{\Theta}(\eta) := \Theta^*(\eta) = \eta\circ (d\mu-\rho\circ \theta);
    \]
    \item on linear $1$-forms $\alpha\in \Omega^1(A)$, 
    \[
    \kappa^{\Theta}(\alpha) := \Theta^*(\alpha) = \alpha\circ \theta;
    \]
\end{itemize}
    \item On sections $a$ of $S^1A^*$, we 
    set $\kappa^{\Theta} = R_\Theta\circ \iota$:
    \[
    \kappa^{\Theta}(a) := d_\mathrm{dR}\circ \Theta^*\circ \iota (a)-\Theta^*\circ d_\oplus\circ \iota (a),
    \]
    which, by corollary~\ref{cor:calc_obstruction}, is given by
    \[
    \kappa^{\Theta}(a) = + \langle \iota(a), \mc_\theta\rangle
    \]
    \item On elements of higher degree in $W^\bullet(A,D)$, $\kappa^{\Theta}$ is defined by imposing that it is a map of graded algebras; in particular 
    \[
    \kappa^{\Theta}(\eta\otimes\alpha\otimes a) = \Theta^*(\eta)\wedge \Theta^*(\alpha)\wedge R_\Theta\circ \iota(a)
    \]
    for $\eta\otimes\alpha\otimes a\in \Gamma(\Lambda^uT^*M\otimes \Lambda^wA^*\otimes S^vA^*)$.
\end{itemize}

By definition, $\kappa^{\Theta}$ fits into a commutative diagram of graded algebras
        \[
        \begin{tikzcd}
        W^\bullet(A,D)\arrow[r,"\kappa^{\Theta}"] & \Omega^\bullet(P)\\
        \Omega^\bullet(TM\bowtie A)\arrow[u,hook]\arrow[ur,"\Theta^*",swap] & 
        \end{tikzcd}.
        \]
The second assertion of Theorem~\ref{thm:conn_curv} follows readily from this description. What one really needs to show in order to prove Theorem~\ref{thm:conn_curv} is that $\kappa^{\Theta}$ is a map of differential graded algebras
 \[
(W^\bullet(A,D), d_W) \to (\Omega^\bullet(P),d_{\mathrm{dR}})
\]
where $d_W$ is the differential of the Weil algebra. Since  $W^\bullet(A,D)$ is the free graded commutative $C^\infty(M)$-algebra generated by $\Omega^1(M)\oplus \Omega^1(A)$ in degree $1$ and $\Gamma(S^1A^*)$ in degree $2$; hence, if
\[
    \Delta:=\kappa^{\Theta}\circ d_W - d_\mathrm{dR}\circ \kappa^{\Theta}
\]
vanishes on generators and on $C^\infty(M)$, it vanishes identically.

If $f\in C^\infty(M)$, the only surviving summand of $d_W f = d_\mathrm{cad}f$, and $d_\oplus f = \left(d_\mathrm{dR}f,\ d_\mathrm{dR}f\circ \rho\right)$, so
\[
\kappa^{\Theta}(d_Wf) = d_\mathrm{dR}f\circ\left(d\mu-\rho\circ \theta\right)+d_\mathrm{dR}f\circ \rho\circ \theta = d_\mathrm{dR}f\circ d\mu = d_\mathrm{dR}\left(\mu^*f\right) = d_\mathrm{dR}\kappa^{\Theta}(f).
\]

We then look at elements of $\Omega^1(TM\bowtie A)$, where the check is simple.
\begin{lemma}\label{lem:kappa_on_one_forms}
    The following identities hold:
    \begin{itemize}
        \item $\kappa^{\Theta}(d_W\eta)=d_\mathrm{dR}\kappa^{\Theta}(\eta)$, for all $\eta\in \Omega^1(M)$;
        \item $\kappa^{\Theta}(d_W\alpha)=d_\mathrm{dR}\kappa^{\Theta}(\alpha)$ for all $\alpha\in \Omega^1(A)$.
    \end{itemize}
\end{lemma}
\begin{proof}
    Both identities follow from the definition of $\Theta$ and $\kappa^{\Theta}$ together with Lemma~\ref{lem:relation_differentials}.

    Indeed, for any $\eta\in \Omega^1(M)$, we have
    \[
    \kappa^{\Theta} d_W(\eta) = \kappa^{\Theta} (d_\oplus\eta - \iota(\eta\circ \rho))
    \]
    which, by definition of $\kappa^{\Theta}$, is equal to
    \[
    \Theta^*(d_\oplus \eta) - R_\Theta(\eta\circ \rho).
    \]
    Moreover, 
    \[
    \Theta^*(d_\oplus \eta) = d_\mathrm{dR} \Theta^* \eta - R_\Theta(\eta).
    \]
    Finally, we observe that
    \[
    R_\Theta(\eta\circ \rho) = - R_\Theta(\eta)
    \]
    follows from Corollary~\ref{cor:calc_obstruction}. All together
    \[
    \kappa^{\Theta}d_W(\eta) = \left(d_\mathrm{dR}\Theta^*\eta - R_\Theta(\eta)\right)+R_\Theta(\eta) = d_\mathrm{dR}\kappa^{\Theta}(\eta).
    \]

    For $\alpha\in \Omega^1(A)$, by \eqref{eq:dW_on_alpha}, $d_W\alpha = d_\oplus\alpha+\iota(\alpha)$, and $\kappa^{\Theta}\circ \iota = R_\Theta$ on $\Omega^1(A)$ by the definition of $\kappa^{\Theta}$ on $\Gamma(S^1A^*)$. Hence
    \[
    \kappa^{\Theta}(d_W\alpha) = \Theta^*(d_\oplus\alpha)+R_\Theta(\alpha)
    = \left(d_\mathrm{dR}\Theta^*\alpha - R_\Theta(\alpha)\right)+R_\Theta(\alpha)
    = d_\mathrm{dR}\kappa^{\Theta}(\alpha).
    \]
\end{proof}

We are left with showing that 
\[
    \kappa^{\Theta}(d_W a)=d_\mathrm{dR}\kappa^{\Theta}(a), \quad a\in \Gamma(S^1A^*).
\]
First we observe that, by writing $a=\iota(\alpha)$, $\alpha\in \Omega^1(A)$, and recalling the definitions of $\kappa^{\Theta}$ and $d_W$, we have
\[d_\mathrm{dR}\kappa^{\Theta}(a) = -d_\mathrm{dR}\circ \Theta^*\circ d_\oplus(\alpha) = -\Theta^*\circ d_\oplus^2(\alpha) -R_\Theta\circ d_\oplus (\alpha)\]
and
\[\kappa^{\Theta}(d_W(a)) = \kappa^{\Theta}(d_\mathrm{cad}(a))+\Theta^*(\langle a, \curv^D \rangle + \langle a, \cbas^D \rangle)\]

We know that
\[
 d_\oplus^2(\alpha) = - \langle \alpha, \curv^D \rangle - \langle \alpha, \cbas^D \rangle
 \]
from Lemma~\ref{lem:dplus_square_curv}. Hence, it remains to prove the following Lemma.

\begin{lemma}\label{lem:kappa_carrying_step}
   Let $\alpha\in \Omega^1(A)$ and $a=\iota(\alpha)$. Then
    \[
    \kappa^{\Theta}\left(d_\mathrm{cad}(a)\right) = -R_\Theta\left(d_\oplus \alpha\right).
    \]
\end{lemma}

\begin{proof}
Write
\[
\Xi := \left(d_{\nabla^{A^*}}+d_{D^*}\right)(\alpha)\ \in\ \Gamma\big((TM\oplus A)^*\otimes A^*\big).
\]
By~\eqref{eq:diff_ad}, $d_\mathrm{cad}(a) = -\iota(\Xi)$. Applying $\kappa^{\Theta}$ gives the three-form
\begin{equation}\label{eq:proof1}
\begin{aligned}
\kappa^{\Theta}\left(d_\mathrm{cad}(a)\right)(X,Y,Z) = {}& \Xi\big(\Theta(X),\mc_\theta(Y,Z)\big) - \Xi\big(\Theta(Y),\mc_\theta(X,Z)\big)\\
&+ \Xi\big(\Theta(Z),\mc_\theta(X,Y)\big).
\end{aligned}
\end{equation}

To manipulate the other side, set $V:=\left(-\rho\circ \mc_\theta,\ \mc_\theta\right)\in \Omega^2(P,TM\oplus A)$, so that $R_\Theta(\xi) = \xi\circ V$ for any linear one-form $\xi$ -- see Corollary~\ref{cor:calc_obstruction}. Since $R_\Theta$ is a degree $+1$ derivation along $\Theta^*$, on a two-form $\zeta$ written in a local basis $\sigma^I$ of $(TM\oplus A)^*$ as $\zeta = \tfrac12\zeta_{IJ}\,\sigma^I\wedge \sigma^J$ one gets
\begin{equation}\label{eq:proof2}
R_\Theta(\zeta) = \zeta_{IJ}\ (\sigma^I\circ V)\wedge (\sigma^J\circ \Theta),
\end{equation}
Evaluating on $(X,Y,Z)$ and using $\zeta(V,\Theta)=-\zeta(\Theta,V)$, we get
\[
-R_\Theta(\zeta)(X,Y,Z) = \zeta\big(\Theta(X),V(Y,Z)\big) - \zeta\big(\Theta(Y),V(X,Z)\big) + \zeta\big(\Theta(Z),V(X,Y)\big).
\]

Since $V(Y,Z)=\big(-\rho(\mc_\theta(Y,Z)),\mc_\theta(Y,Z)\big)$, by comparing the equation above with \eqref{eq:proof1} we see that it is sufficient to show that 
\begin{equation*}
\Xi(\Theta(X),\mc_\theta(Y,Z)) = d_\oplus\alpha\big(\Theta(X),(-\rho(\mc_\theta(Y,Z)),\mc_\theta(Y,Z))\big);
\end{equation*}
in fact, we will show that 
\begin{equation}\label{eq:proof_pointwise}
\Xi(u,\sigma) = d_\oplus\alpha\big(u,(-\rho(\sigma),\sigma)\big),\qquad u=(u_M,u_A)\in TM\oplus A,\ \sigma\in A.
\end{equation}

Both sides of~\eqref{eq:proof_pointwise} are computed from the definitions. 

For $\gamma,\beta\in\Gamma(A)$, the definition of $\nabla^{A^*}$ together with $\nabla^A_\gamma\beta=[\gamma,\beta]+D_{\rho(\beta)}\gamma$ gives
\[
\begin{aligned}
(d_{\nabla^{A^*}}\alpha)(\gamma)(\beta) &= \rho(\gamma)\big(\alpha(\beta)\big) - \alpha([\gamma,\beta]) - \alpha\big(D_{\rho(\beta)}\gamma\big)\\
&= d_\mathrm{CE}\alpha(\gamma,\beta) + \big(D^*_{\rho(\beta)}\alpha\big)(\gamma).
\end{aligned}
\]
Hence
\[
\Xi(u,\sigma) = \big(D^*_{u_M}\alpha\big)(\sigma) + d_\mathrm{CE}\alpha(u_A,\sigma) + \big(D^*_{\rho(\sigma)}\alpha\big)(u_A).
\]

On the other hand, $d_\oplus\alpha = d_{D^*}\alpha+d_\mathrm{CE}\alpha$ and 
\[
d_{D^*}\alpha(u,\sigma) = (D^*_{u_M}\alpha)(\sigma) + (D^*_{\rho(\sigma)}\alpha)(u_A),
\]
\[
d_\mathrm{CE}\alpha(u,\sigma) = d_\mathrm{CE}\alpha(u_A,\sigma).
\]
We deduce \eqref{eq:proof_pointwise} and hence
\begin{equation}\label{eq:proof3}
\kappa^{\Theta}\left(d_\mathrm{cad}(a)\right) = -R_\Theta\left(d_\oplus \alpha\right).
\end{equation}
\end{proof}

\subsection{The action induced by $\omega$ and the basic subcomplex}\label{subs:basic_subcomplex}
Both sides of the connection-curvature dga map
\[
\kappa^{\Theta}: W^\bullet(A,D)\to \Omega^\bullet(P)
\]
from Theorem~\ref{thm:conn_curv} are equipped with a natural action of $A$, as we discuss in this subsection.

First of all, we clarify some terminology. By \textbf{action of $A$ on a graded commutative dga $(\mathcal{E}^\bullet, d_\mathcal{E})$} we mean a map
\[
\Gamma(A)\to \mathrm{Der}^{-1}(E^\bullet),\quad  \epsilon \mapsto \epsilon\cont
\]
from $\Gamma(A)$ to derivations of $(\mathcal{E}^\bullet, d_\mathcal{E})$ such that, for all $\epsilon\in \Gamma(A)$, $\epsilon\cont$ is a degree $-1$ derivation and, for all $\epsilon,\epsilon'\in\Gamma(A)$,+
\[
[\epsilon \cont, \epsilon' \cont] = 0,\quad \epsilon,\epsilon'\in \Gamma(A).
\]
Given such a action, we can equip $\mathcal{E}^\bullet$ with a Cartan-like calculus by setting
\[
L_\epsilon := [d_\mathcal{E},\epsilon\cont], \epsilon \in A
\]
where on the right hand side we have the graded commutator of derivations between the differential $d_\mathcal{E}$ of $\mathcal{E}^\bullet$ and $\epsilon\cont$. For all $\epsilon$, $L_\epsilon$ is a degree zero derivation; thus, we get a map
\[
\Gamma(A)\to \mathrm{Der}^0(\mathcal{E}^\bullet), \epsilon \mapsto L_\epsilon.
\]
    If  $L$ defines a map of Lie algebras, i.e.
    \[
    L_{[\epsilon,\epsilon']} = [L_{\epsilon},L_{\epsilon'}],\quad \epsilon,\epsilon'\in \Gamma(A),
    \]
    we call (with a slight conflict with usual terminology) $(\mathcal{E},d_\mathcal{E})$ a \textbf{representation} of $A$.

Given an $A$-action on $(\mathcal{E},d_\mathcal{E})$ one has three natural sub-algebras:
\begin{itemize}
    \item the sub-algebra of \textbf{horizontal elements}
    \[
    \mathcal{E}^\bullet_\mathrm{hor} = \{e\in \mathcal{E}^\bullet:\ \epsilon\cont e = 0, \quad \text{for all } \epsilon \in \Gamma(A)\};
    \]
    \item the sub-algebra of \textbf{invariant elements}
    \[
    \mathcal{E}^\bullet_\mathrm{inv} = \{e\in \mathcal{E}^\bullet:\ L_\epsilon e = 0, \quad \text{for all } \epsilon \in \Gamma(A)\};
    \]
    \item the sub-algebra of \textbf{basic elements}
    \[
    \mathcal{E}^\bullet_\mathrm{bas} = \mathcal{E}^\bullet_\mathrm{hor}\cap \mathcal{E}^\bullet_\mathrm{inv}.
    \]
\end{itemize}
Notice that the sub-algebras of invariant elements and basic elements are in fact sub-dgas.

\begin{example} \label{ex:TM_dga}
The prototype of action of an algebroid on a dga is given by the $TM$-dga $(\Omega^\bullet(M),d_\mathrm{dR})$, where $M$ is a manifold and $\cont$, $L$ are the usual contraction and Lie derivative operators. 
\end{example}

\begin{example}[Principal bundles]~\label{ex:Cart_calc_pb}
Slightly more relevant for our purposes, the infinitesimal action of the Lie algebroid $A$ of a groupoid $\Sigma$ on a principal $\Sigma$-bundle $\pi:P\to X$ induces an action on the dga $(\Omega^\bullet(P),d_\mathrm{dR})$. The infinitesimal action $a_P: \mu^*A \to TP$ induces a map
\[
a_P: \Gamma(\mu^*A) \to \mathfrak{X}(P)
\]
which can be dualized, providing, for all $\epsilon\in \Gamma(A)$, the map
\[
\epsilon\cont: \Omega^\bullet(P) \to \Omega^{\bullet-1}(P)
\]
given by contraction of $\eta \in \Omega^\bullet(P)$ with the $\pi$-vertical vector field 
\[
\hat{\epsilon} := a_P(\mu^*\epsilon)\in \mathfrak{X}(P);
\]
accordingly, by Cartan's magic formula, the degree zero derivation $L_\epsilon$ is simply the Lie derivative along $\hat{\epsilon}$. It follows that the action on $(\Omega^\bullet(P),d_\mathrm{dR})$ is in fact a representation.

Here, the algebra of horizontal elements corresponds to horizontal forms, while basic elements are precisely basic forms on $P$ -- hence, pullbacks of forms on its base $X$ if $\Sigma$ is $s$-connected.
\end{example}

\begin{definition}\label{def:A_dga}
    We say that a graded commutative dga $(\mathcal{E}^\bullet,d_\mathcal{E})$ is an \textbf{$A$-dga} if 
    \begin{itemize}
        \item $\mathcal{E}^\bullet$ is the algebra of sections of a graded vector bundle on $M$
        \[
        \mathcal{E}^\bullet = \Gamma(E^\bullet)
        \]
        and $\mathcal{E}^0=C^\infty(M)$, $E^0=M\times \mathbb{R}$;
        \item the map $\cont$ arises as a vector bundle map, denoted in the same way,
        \[
        A\otimes E^\bullet \to E^{\bullet-1},\quad \epsilon \otimes \alpha \to \epsilon\cont \alpha;
        \]
        \item the operator $L$ defines an $A$-connection on horizontal elements.
    \end{itemize}
    
\end{definition}

Let us clarify the third point. The first two bullet points ensure that the derivations
        \[
        L_\epsilon\ =\ [d_\mathcal{E},\epsilon\cont]\ \in\ \mathrm{Der}^0(\mathcal{E}^\bullet),
        \qquad \epsilon\in\Gamma(A),
        \]
        satisfy 
        \begin{equation}\label{eq:L_is_not_linear}
            L_{f\epsilon}\ =\ f\,L_\epsilon\ +\ d_\mathcal{E}f\wedge(\epsilon\cont\,\cdot\,),
            \qquad f\in C^\infty(M).
        \end{equation}
In particular, the map $\epsilon\to L_\epsilon$ becomes $C^\infty$-linear when its image is restricted to horizontal sections. The sub-algebra of such sections can in turn be realized as the sub-algebra of sections taking values in 
    \[
    E^\bullet_\mathrm{hor}=\{e\in E^\bullet: \epsilon \cont e = 0,\quad \text{for all } \epsilon \in \Gamma(A)\}.
    \]
Then, the request that this map defines an $A$-connection boils down to the three conditions
\begin{itemize}
    \item for all $\epsilon\in \Gamma(A)$, the operator $L_\epsilon$ sends $E^\bullet_\mathrm{hor}$-valued sections to $E^\bullet_\mathrm{hor}$-valued sections;
    \item for all $f\in C^\infty(M)$ and $\epsilon \in \Gamma(A)$, $\epsilon\cont d_{\mathcal{E}}(f) = L_{\rho(\epsilon)}(f)$
    \item $E^\bullet_\mathrm{hor}$ has constant rank.
\end{itemize}
    An $A$-dga is called \textbf{flat} if the $A$-connection induced on horizontal elements is flat -- i.e., if $L$ makes $E^\bullet_\mathrm{hor}$ into a representation of $A$ in the usual sense. Notice that an $A$-dga satisfying the Cartan-like identity
    \[
    [L_\epsilon, \epsilon'\cont] = [\epsilon,\epsilon']\cont
    \]
    is a flat $A$-dga.

The example below discusses the flat $A$-dga which is central for our purposes: the Weil algebra of a Lie groupoid with Cartan connection $(\Sigma,\omega)$. 

\begin{example}[Weil algebra and basic subcomplex induced by $\omega$]\label{ex:Cart_calc_Weil_alg}
Let $(\mathcal{E}^\bullet, d_\mathcal{E}) = (W^\bullet(A,D), d_W)$, where $D$ is the infinitesimally Cartan connection induced by the Cartan connection $\omega$ on the Lie groupoid $\Sigma$. For all $\epsilon\in \Gamma(A)$ and 
\[
\eta \otimes a \otimes \alpha \in \Gamma(\Lambda^uT^*M\otimes S^v(A^*)\otimes \Lambda^w(A^*)) = W^{u+w+2v}(A,D)
\]
we set, following the Koszul rule for a derivation of degree $-1$ and recalling that the symmetric leg carries degree $2v$, so that passing $\epsilon\cont$ across $\eta\otimes a$ produces $(-1)^{u+2v}=(-1)^u$,
\[
\epsilon\cont(\eta \otimes a\otimes \alpha) := (-1)^u\,\eta\otimes a \otimes \epsilon\cont\alpha
\]
and obtain a map
\[
\epsilon\cont: W^\bullet(A,D)\to W^{\bullet-1}(A,D)
\]
by imposing $C^\infty(M)$-linearity. As for $L$, it acts on generators as follows. For all $\epsilon \in \Gamma(A)$, by definition
\[
L_\epsilon \eta = \epsilon \cont d_W \eta,\quad L_\epsilon a = \epsilon \cont d_W a
\]
for all $\eta \in \Omega^1(M)$, $a\in \Gamma(S^1A^*)$; expanding $d_W$
\[
\epsilon \cont d_W \eta = \nabla^*_\epsilon \eta
\]
\[
\epsilon \cont d_W a =  \iota(\nabla^{A^*}_\epsilon\iota(a)) + \epsilon \cont \langle a, \curv^D\rangle .
\]

On the other hand 
\[
L_\epsilon \alpha =  d_\mathrm{dR}(\alpha(\epsilon))+\rho^*d_\mathrm{dR}(\alpha(\epsilon)) + \epsilon \cont d_\mathrm{CE}\alpha -\langle d_{D^*}\alpha,\epsilon \rangle
\]

Notice that the sub-algebra of horizontal elements is the sub-algebra generated by $\Omega^1(M)$ and $\Gamma(S^1A^*)$ -- which consists of sections of the bundle $\Lambda^\bullet T^*M\otimes S^\bullet A^*$. 

On this sub-algebra there is a global counterpart of the $A$-connection $L$, namely the representation $\hcdot$ of $\Sigma$ defined on generators by
\begin{equation}\label{eq:omega_action_on_forms}
(g,\eta) \mapsto g\hcdot \eta = g\cdot \eta,\quad (g,\eta) \in \Sigma\times_M T^*M
\end{equation}
and
\begin{equation}\label{eq:omega_action_on_polys}
(g,a)\mapsto g\hcdot a = \iota(g\cdot \iota(a)) + \big\langle \iota(a),\Curv^\omega_{g^{-1}}\big\rangle,\quad (g, a)\in \Sigma\times_M S^1 A^* 
\end{equation}
where $\cdot$ denotes the co-adjoint representation of $\Sigma$ on $T^*M\to A^*$ induced by $\omega$ and $\Curv^\omega$ is the curvature of $\omega$. 

Notice that formula \eqref{eq:omega_action_on_polys} defines an associative action. The first summand is simply the co-adjoint representation; the second summand is associative because the curvature of $\omega$ is a cocycle. Indeed, the cocycle identity
\[
\Curv^\omega_{gh}(v,w) = \Curv^\omega_g(h\cdot v,h\cdot w) + g\cdot \Curv^\omega_h(v,w),\qquad v,w\in T_{s(h)}M
\]
applied to $(gh)^{-1}$ gives
\[
\big\langle \iota(a),\Curv^\omega_{h^{-1}g^{-1}}\big\rangle = \big\langle h\cdot \iota(a), \Curv^\omega_{g^{-1}}\big\rangle + g\cdot \big\langle \iota(a),\Curv^\omega_{h^{-1}}\big\rangle.
\]

Moreover, this action differentiates to the $A$-connection $L$. For $\epsilon\in\Gamma(A)$ and $x\in M$, let $\lambda\mapsto g_\lambda$ be the integral curve through $1_x$ of the right invariant vector field $\epsilon^R$, so that $s(g_\lambda)=x$. Then, for every section $\chi$ of the horizontal sub-algebra,
\[
\frac{d}{d\lambda}\Big|_{\lambda=0}\ g_\lambda^{-1}\cdot\chi\big(t(g_\lambda)\big) = (L_\epsilon\chi)(x),
\]
since the co-adjoint representation of $\Sigma$ differentiates to the $A$-connections $\nabla^*$ and $\nabla^{A^*}$ in the formulas for $L_\epsilon$ above, $(\Curv^\omega)_{1_x}=0$, and the linearisation of $\Curv^\omega$ is $\curv^D$ by Proposition~\ref{prop:curv_alg_is_lin_curv_gpd}. Consequently, when $\Sigma$ is $s$-connected, an element of the horizontal sub-algebra is $\Sigma$-invariant if and only if it is $L$-invariant.

In what follows, we will use the notation
\[
W^\bullet(A,D)_{\omega\mathrm{-bas}}
\]
for the sub-algebra of horizontal elements which are also invariant under the action of $\Sigma$ defined by \eqref{eq:omega_action_on_forms} and \eqref{eq:omega_action_on_polys}. Notice that $W^\bullet(A,D)_{\omega\mathrm{-bas}}$ is contained in the sub-dga of basic elements of $W^\bullet(A,D)$ with respect to the $A$-dga structure that we just introduced, and it coincides with it when $\Sigma$ is $s$-connected. 
\end{example}

\begin{proposition}\label{prop:omega_bas_subdga}
    The sub-algebra $W^\bullet(A,D)_{\omega\mathrm{-bas}}\subset W^\bullet(A,D)$ is a sub-dga.
\end{proposition}
\begin{proof}
    When $\Sigma$ is $s$-connected, $W^\bullet(A,D)_{\omega\mathrm{-bas}}$ is the algebra of basic elements with respect to the $A$-dga structure of Example \ref{ex:Cart_calc_Weil_alg}, and there is nothing to prove. For the general case, we first need to shed some more light on \eqref{eq:omega_action_on_forms} and \eqref{eq:omega_action_on_polys}. Namely, we will construct, for every bisection $\mathfrak{b}$, a dga map
    \[
    \Phi_\mathfrak{b}:  W(A,D)|_{t\circ \mathfrak{b}(U)} \to W(A,D)|_U
    \]
    with the property that, for every $\chi\in \Lambda^\bullet T^*M\otimes S^\bullet A^*$
    \[
    \Phi_\mathfrak{b}(\chi)_x = \mathfrak{b}(x)^{-1}\hat{\cdot} \chi_{\varphi_\mathfrak{b}(x)};
    \]
    since every arrow $g$ lies on a local bisection, and since the differential of $\chi \in W^\bullet(A,D)_{\omega\mathrm{-bas}}\subset \Lambda^\bullet T^*M\otimes S^\bullet A^*$ satisfies
    \[
    \epsilon \cont d_W \chi = L_\epsilon \chi  = 0,
    \]
    this will imply the claim.

    Let $\mathfrak{b}:U\to \Sigma$ be a bisection of $\Sigma$; we set 
    \[
    \varphi_\mathfrak{b}:=t\circ \mathfrak{b}:U\to t\circ \mathfrak{b}(U),\quad V:= {\varphi_\mathfrak{b}}(U)
    \]
    and
    \[
    \beta_\mathfrak{b}:=b^*\omega \in \Omega^1(U, \varphi_\mathfrak{b}^*A).
    \]
    For each $x\in U$, $\mathfrak{b}$ induces maps
    \[
    \nu_x:=\mathfrak{b}(x)\cdot: A_x \to A_{\varphi_\mathfrak{b}(x)},\quad  \lambda_x:= \mathfrak{b}(x)\cdot: T_xM \to T_{\varphi_\mathfrak{b}(x)}M
    \]
    satisfying $\rho\circ \nu_x = \lambda_x \circ \rho$. Additionally, we have 
    \[
    d\varphi_\mathfrak{b} = \lambda + \rho\beta_\mathfrak{b}.
    \]
    We can assume that $U$ supports a coordinate chart, and treat $W(A,D)|_{V}$ and $W(A,D)|_{U}$ as freely generated by functions and linear forms on $A$.

    The claimed dga map 
    \[
    \Phi_\mathfrak{b}:  W(A,D)|_{V} \to W(A,D)|_U
    \]
    induced by the bisection $\mathfrak{b}$ is then the unique one such that 
    \[
    \Phi_\mathfrak{b}(f) = f\circ \varphi_\mathfrak{b},\quad \Phi_\mathfrak{b}(\alpha) = \nu^* \varphi_\mathfrak{b}^*\alpha + \langle \varphi_\mathfrak{b}^*\alpha,\beta_\mathfrak{b}+\beta_\mathfrak{b}\circ \rho \rangle 
    \]
    for all $f\in C^\infty(V)$ and $\alpha\in \Gamma(\Lambda^1A^*|_V)$. That such a map exists --- that is, that these assignments are compatible with $d_W$ --- is a direct computation on generators. This implies
    \[
    \Phi_\mathfrak{b}(\eta) = \lambda^*\varphi_\mathfrak{b}^*\eta,\quad \Phi_\mathfrak{b}(a) = \nu^* \varphi_\mathfrak{b}^*a + \langle \varphi_\mathfrak{b}^*\iota(a), \Curv^\omega_{\mathfrak{b}}\rangle
    \]
    for all $\eta\in \Omega^1(V)$, $a\in \Gamma(S^1A^*|_V)$. 
    
    The former identity can be checked via a direct calculation. We have 
    \[
    d_W\Phi_\mathfrak{b}(f) = d(f\circ \varphi_\mathfrak{b})+d(f\circ \varphi_\mathfrak{b})\circ \rho 
    \]
    and 
    \[
    \Phi_\mathfrak{b}(df)+\nu^*\varphi_\mathfrak{b}^*(df\circ \rho) + \langle \varphi_\mathfrak{b}^*(df\circ \rho),\beta_\mathfrak{b}+\beta_\mathfrak{b}\circ \rho \rangle,
    \]
    obtained using the definition of $\Phi_\mathfrak{b}$ and that on smooth functions $g$ $d_Wg = dg+dg\circ \rho$.
    In the first equality we substitute
    \[
    d(f\circ \varphi_\mathfrak{b}) = df \circ (\lambda + \rho\beta_\mathfrak{b});
    \]
    imposing $\Phi_\mathfrak{b}d_Wf=d_W\Phi_\mathfrak{b}(f)$ and simplifying by using $\rho\circ \nu = \lambda \circ \rho$ yields
    \[
    \Phi_\mathfrak{b}(df) = \lambda^*df
    \]
    as desired. 

    We are now going to prove the identity
    \[
    \Phi_\mathfrak{b}(a) - \nu^* \varphi_\mathfrak{b}^*a = \langle \iota(a), \Curv^\omega_{\mathfrak{b}}\rangle.
    \]
    Recall that
    \[
    d_W \alpha = \iota(\alpha)+d_{D^*}\alpha+d_\mathrm{CE}\alpha.
    \]
    Thus 
    \[
    \Phi_\mathfrak{b}(\iota(\alpha)) = d_W(\Phi_\mathfrak{b}(\alpha))-\Phi_\mathfrak{b}(d_{D^*}\alpha)-\Phi_\mathfrak{b}(d_\mathrm{CE}\alpha).
    \]
    It follows from the formulas defining $\Phi_\mathfrak{b}(\alpha)$ and $d_W$ that the only term of $\Phi_\mathfrak{b}(\iota(\alpha))$ belonging in $\Gamma(S^1A^*|_U)$ is $
    \nu^* \varphi_\mathfrak{b}^*\iota(\alpha)$
    as predicted by the formula that we want to prove. It remains to show that the other terms of $\Phi_\mathfrak{b}(\alpha)$ reduce to $\langle \iota(a), \Curv^\omega_{\mathfrak{b}}\rangle$.
    
    First we isolate such terms by using Theorem \ref{thm:conn_curv} on the principal $(\Sigma,\omega)$-bundle consisting of $(\Sigma,\omega)$ itself, which yields a dga map
    \[
    \kappa^\Omega: (W^\bullet(A,D),d_W) \to (\Omega^\bullet(\Sigma),d_\mathrm{dR})
    \]
    which pairs $\Gamma(S^1A^*)$ with $\mc_\omega$, and $\Omega^1(M)$ and $\Omega^1(A)$ with $dt-\rho\circ \omega$ and $\omega$ respectively.
    Since $\mc_\omega$ vanishes identically at points of $M$, and since at such points $dt-\rho\circ \omega = ds$, we have
    \[
    \kappa^\Omega\circ \Phi_\mathfrak{b}(\iota(\alpha))|_U=\langle\Phi_\mathfrak{b}(\iota(\alpha)),(ds,\omega)\rangle;
    \]
    but $\ker(\omega)$ at $M$ is exactly $TM$, hence the right hand side computes precisely the terms of $\Phi_\mathfrak{b}(\iota(\alpha))$ which do not lie in $\Gamma(S^1A^*|_U)$. 

    The bisection $\mathfrak{b}$ induces the diffeomorphism
    \[
    \Psi_\mathfrak{b}:t^{-1}(U) \to t^{-1}(V), \quad g\mapsto \mathfrak{b}(t(g))\cdot g
    \]
    which, restricted at units, is just $\mathfrak{b}$. We claim that
    \[
    \Psi_\mathfrak{b}^*\circ \kappa^\Omega = \kappa^\Omega\circ \Phi_\mathfrak{b}
    \]
    which will conclude because
    \[
    \Psi_\mathfrak{b}^*\circ \kappa^\Omega (\iota(\alpha))|_{U} = \langle \iota(\alpha), \Curv^\omega_\mathfrak{b} \rangle.
    \]

    Since both $\Psi_\mathfrak{b}^*\circ \kappa^\Omega$ and $\kappa^\Omega\circ \Phi_\mathfrak{b}$ are dga maps, it is enough to check that they coincide on functions and elements of $\Omega^1(A)$. On functions $\kappa^\Omega (f) = t^*f$ hence, by using $t\circ \Psi_\mathfrak{b} = \varphi_\mathfrak{b}\circ t$ 
    \[
    \Psi_\mathfrak{b}^*\circ \kappa^\Omega(f) = (\varphi_\mathfrak{b}\circ t)^*f = \kappa^\Omega(\Phi_\mathfrak{b}(f)).
    \]
    On linear forms $\alpha\in \Omega^1(A)$,
    \[
    \kappa^\Omega(\Phi_\mathfrak{b}(\alpha)) = \Phi_\mathfrak{b}(\alpha),(dt-\rho\circ \omega,\omega)
    \]
    which gives
    \[
    \langle \nu^*(\varphi_\mathfrak{b}^*\alpha),\omega\rangle + \langle \varphi_\mathfrak{b}^*\alpha, \langle \beta_\mathfrak{b} , dt\rangle \rangle.
    \]
    The multiplicativity of $\omega$ yields
    \[
    \Psi^*_\mathfrak{b}\omega = \beta_\mathfrak{b}\circ dt + \nu\cdot \omega,
    \]
    which gives
    \[
    \kappa^\Omega(\Phi_\mathfrak{b}(\alpha)) = \langle \varphi_\mathfrak{b}^*\alpha,\Psi^*\omega \rangle = \Psi^*_\mathfrak{b}\langle \alpha,\omega\rangle
    \]
\end{proof}

\subsection{The characteristic map}\label{subs:characteristic_map}
Throughout this subsection $(\Sigma,\omega)$ is a Lie groupoid with Cartan connection, $D$ is the connection it induces on $A$, and $(P,\theta)$ is a principal $(\Sigma,\omega)$-bundle over $X$ with momentum map $\mu$. 
\begin{theorem}\label{thm:map_on_basics}
    The map $\kappa^\Theta$ sends $W^\bullet(A,D)_{\omega\mathrm{-bas}}$ into the basic forms on $P$, and therefore restricts to a dga map
    \[
    \kappa^{\Theta}: (W^\bullet(A,D)_{\omega\mathrm{-bas}},d_W) \longrightarrow (\Omega^\bullet (X),d_\mathrm{dR}).
    \]
\end{theorem}
The proof relies on the two lemmata below, which prove that $\kappa^\Theta$ is compatible with both the $A$-dga structure on $(W^\bullet(A,D),d_W)$ and the $\Sigma$-action on horizontal elements \eqref{eq:omega_action_on_forms}-\eqref{eq:omega_action_on_polys}.

\begin{lemma}\label{lem:comp_with_action}
    For every $\epsilon\in\Gamma(A)$,
    \[
    \kappa^{\Theta} \circ (\epsilon \cont) = \epsilon \cont \kappa^{\Theta}
    \]
    and 
    \[
    \kappa^{\Theta} \circ (L_\epsilon) = L_\epsilon \kappa(\Theta).
    \]
\end{lemma}
\begin{proof}
    The second identity follows from the first and the fact that $\kappa^{\Theta}$ is a dga map. The first identity can be checked on generators.
    By 
    \[
    \theta\circ a_P = \id,
    \]
    one has that $\epsilon \cont \kappa^{\Theta} (\alpha) = \alpha(\epsilon)$ for all $\alpha \in \Omega^1(A)$ and $\epsilon \in \Gamma(A)$. Taking also into account that
    \[
    d\mu \circ a_P = \rho,
    \]
    it follows that, for all $\eta\in \Omega^1(M)$, $\kappa^{\Theta} (\eta)$ vanishes on the image of the infinitesimal action, as desired. Finally, since $\mc_\theta$ is horizontal, we see that for all $\phi \in \Gamma(S^1A^*)$, $\kappa^{\Theta} (\phi)$ vanishes on the image of the infinitesimal action as well, which concludes the proof.
\end{proof}
We then need to show that $\kappa^{\Theta}$ is equivariant with respect to the action of $\Sigma$ on horizontal elements of the Weil algebra -- see Example \ref{ex:Cart_calc_Weil_alg}.
\begin{lemma}\label{lem:kappa_equivariant}
   For every $(g,p)\in\Sigma\ltimes P$ and every horizontal $\gamma$ at $\mu(p)$,
   \[
   \kappa^{\Theta}_{g\cdot p}(g\hcdot \gamma) = g\cdot \kappa^{\Theta}_p(\gamma),
   \]
   where $g$ acts on the left by~\eqref{eq:omega_action_on_forms}--\eqref{eq:omega_action_on_polys}, and on the right by the co-adjoint representation of $(\Sigma\ltimes P,\pr_1^*\omega)$, that is, $(g\cdot\beta)(w_1,\dots,w_k)=\beta(g^{-1}\cdot w_1,\dots,g^{-1}\cdot w_k)$ for $w_1,\dots,w_k\in T_{g\cdot p}P$.
\end{lemma}
\begin{proof}
    Both actions are by algebra automorphisms and $\kappa^\Theta_p$, $\kappa^\Theta_{g\cdot p}$ are morphisms of algebras, so it suffices to take either $\gamma=\eta$ in $T^*_xM$ or in $\gamma = a$ in $S^1A^*_x$, where $x:=\mu(p)=s(g)$.

    Let $\eta\in T^*_xM$. Then $\kappa^\Theta_p(\eta)=\eta\circ\an_\theta$, and $\an_\theta$ is $\Sigma$-equivariant by Lemma~\ref{lem:anchored_form_is_equiv}. Hence, for $w\in T_{g\cdot p}P$,
    \[
    \bigl(g\cdot\kappa^\Theta_p(\eta)\bigr)(w)=\eta\bigl(\an_\theta(g^{-1}\cdot w)\bigr)=\eta\bigl(g^{-1}\cdot\an_\theta(w)\bigr)=(g\cdot\eta)\bigl(\an_\theta(w)\bigr)=\kappa^\Theta_{g\cdot p}(g\cdot\eta)(w).
    \]

    Let $a\in S^1A^*_x$. The $2$-form $\kappa^\Theta_{g\cdot p}(g\hcdot a)$ vanishes on $\pi$-vertical vectors by Lemma~\ref{lem:comp_with_action}, and so does $g\cdot\kappa^\Theta_p(a)$; it is therefore enough to compare the two forms on $w_1,w_2\in\ker(\theta)_{g\cdot p}$. 
    
    Let $v_i:=d\pi(w_i)\in T_{\pi(p)}X$ and $u_i:=d\mu(w_i)\in T_{t(g)}M$; then,
    \[
    \an_\theta(w_i)=d\mu(w_i)=u_i .
    \]
    By Corollary~\ref{cor:mc_is_curvature} and Lemma~\ref{lem:curv_is_mult_pb},
    \[
    \begin{aligned}
    \bigl(g\cdot\kappa^\Theta_p(a)\bigr)(w_1,w_2)
    &=\bigl\langle\iota(a),\Curv^\theta_p(g^{-1}\cdot w_1,g^{-1}\cdot w_2)\bigr\rangle\\
    &=\bigl\langle g\cdot\iota(a),\Curv^\theta_{g\cdot p}(w_1,w_2)\bigr\rangle
      -\bigl\langle g\cdot\iota(a),\Curv^\omega_g(g^{-1}\cdot u_1,g^{-1}\cdot u_2)\bigr\rangle .
    \end{aligned}
    \]
    On the other hand
    \[
    \kappa^\Theta_{g\cdot p}(g\hat{\cdot} a)(w_1,w_2)=\bigl\langle g\cdot\iota(a),\Curv^\theta_{g\cdot p}(w_1,w_2)\bigr\rangle+
    \big\langle \iota(a),\Curv^\omega_{g^{-1}}\left(u_1, u_2\right)\big\rangle.
    \]
    Using the cocycle identity for $\Curv^\omega$ and that $\Curv^\omega$ vanishes at units, we deduce that 
    \[
    \Curv^\omega_{g}\left(g^{-1}\cdot u_1, g^{-1}\cdot u_2\right)+g\cdot \Curv^\omega_{g^{-1}}\left(u_1, u_2\right) = 0,
    \]
    which concludes the proof.
\end{proof}

The two lemmata above show that, for every $\chi\in W^\bullet(A,D)_{\omega\mathrm{-bas}}$ the form $\kappa^\Theta(\chi)$ on $P$ is basic; hence, the restriction of the map $\kappa^\Theta$ to $W^\bullet(A,D)_{\omega\mathrm{-bas}}$ takes values in $\Omega^\bullet_\mathrm{bas}(P)$. By Lemma~\ref{lem:kappa_equivariant} such forms are not merely $L$-invariant but $\Sigma$-invariant; since $\pi$ is $\Sigma$-invariant and each $v\mapsto g\cdot v$ is a linear isomorphism $T_pP\to T_{g\cdot p}P$, a horizontal $\Sigma$-invariant form descends through any local section of $\pi$ and the descents through different sections agree, so $\Omega^\bullet_\mathrm{bas}(P)\cong\Omega^\bullet(X)$ with no connectedness hypothesis. Since $\kappa^\Theta$ is a map of dgas by Proposition \ref{prop:omega_bas_subdga} and Theorem \ref{thm:conn_curv}, this concludes the proof of Theorem \ref{thm:map_on_basics}.

Let us denote by $H^\bullet(W^\bullet(A,D)_{\omega\mathrm{-bas}})$ the cohomology of the basic sub-complex of the Weil algebra with $A$-dga structure induced by $\omega$.
\begin{definition}\label{def:main_CW_map}
    Let $(\Sigma,\omega)$ be a Lie groupoid with Cartan connection, and let $(P,\theta)$ be a principal $(\Sigma,\omega)$-bundle. The \textbf{Chern-Weil map} associated to $(P,\theta)$ is the cohomology map
    \[
    \kappa^{\mathrm{CW}}: H^\bullet(W^\bullet(A,D)_{\omega\mathrm{-bas}}) \to H^\bullet_\mathrm{dR}(X)
    \]
    induced by $\kappa^{\Theta}$ via Theorem \ref{thm:map_on_basics}. When necessary for clarity we write $\kappa^{\mathrm{CW}}_{(P,\theta)}$.
\end{definition}

Finally, we have the following, which explains the choice of notation. 
\begin{proposition}\label{prop:CW_indep_of_theta}
    Let $(\Sigma,\omega)$ be a Lie groupoid with Cartan connection, and let $P$ be a principal $\Sigma$-bundle over $X$. The Chern-Weil map $\kappa^{\mathrm{CW}}$ does not depend on the $(\Sigma,\omega)$-connection $\theta$.
\end{proposition}
The proof is completely analogous to the classical one; first, one shows that $\kappa^{\mathrm{CW}}$ is functorial.
\begin{lemma}\label{lem:CW_natural}
    Let $(\Sigma,\omega)$ be a Lie groupoid with Cartan connection, let $(P,\theta)$ be a principal $(\Sigma,\omega)$-bundle over $X$, and let $f:X'\to X$ be a smooth map. Then the pullback $f^*P:=P\times_X X'$ is a principal $\Sigma$-bundle over $X'$ and the form $\pr_1^*\theta$ is a principal $(\Sigma,\omega)$-connection. Moreover
    \[
    \kappa^{\mathrm{CW}}_{(f^*P,\,\pr_1^*\theta)} = f^*\circ \kappa^{\mathrm{CW}}_{(P,\theta)} .
    \]
\end{lemma}
\begin{proof}
    Notice that $f^*P$ is a smooth manifold because $\pi_P\to X$ is a surjective submersion; when endowed with the action
    \[
    g\cdot(p,x')=(g\cdot p,x')
    \]
    it becomes a principal $\Sigma$-bundle over $X'$ with momentum map $\mu\circ\pr_1$. Observe that the projection $\pr_1:f^*P\to P$ is $\Sigma$-equivariant; this implies that $\mathrm{pr}_1^*\theta$ is a principal connection on $P$. 

    Moreover,
    \[
    \Theta_{\mathrm{pr}_1^*P} = \bigl(d(\mu\circ \mathrm{pr}_1)-\rho\circ \mathrm{pr}_1^*\theta,\ \mathrm{pr}_1^*\theta\bigr) = \Theta_P\circ d\mathrm{pr}_1 
    \]
    and
    \[
    \mc_{\mathrm{pr}_1^*\theta} = \mathrm{pr}_1^*\mc_\theta
    \]
    directly from the definition of the Maurer-Cartan expression.
    Thus, 
    \[
    \kappa^{\Theta_{f^*P}} = \mathrm{pr}_1^*\circ \kappa^{\Theta_P}
    \]
    which concludes.
\end{proof}
\begin{proof}[Proof of Proposition \ref{prop:CW_indep_of_theta}]
    As in the classical case, we use that connections are closed under convex combinations, i.e. Lemma \ref{lem:pconn_are_convex} -- and use that the inclusions of $X$ into $X\times \mathbb{R}$ that fix the second coordinate are homotopic to each other.

    Explicitly, $\mathrm{pr}_X:X\times \mathbb{R}\to X$ induces the principal bundle 
    \[
    \hat{P}:=\mathrm{pr}_X^*P
    \]
    and, for any pair of principal connections $\theta_0,\theta_1$ on $P$, the lemma above implies that
    \[
    \hat{\theta} = (1-t)\hat{\theta}_0+t\hat{\theta}_1\in \Omega^1(\hat{P}, (\mu\circ \mathrm{pr}_1)^*A)
    \]
    is a principal connection on $\hat{P}$, where $\hat{\theta}_0:=\mathrm{pr}_P^*\theta_0$, $\hat{\theta_1}:=\mathrm{pr}_P^*\theta_1$, and $t$ is the coordinate on the second factor. The corresponding Chern-Weil map is denoted by $\hat{{\kappa}}^\mathrm{CW}$. 

    For $\epsilon\in [0,1]$, let $i_\epsilon: x\in X\mapsto (x,\epsilon) \in X\times \mathbb{R}$ be the inclusions. Since each $i_\epsilon$ is a section of $\mathrm{pr}_X$, we have $i_\epsilon^*\hat{P} = P$ for all $\epsilon$; the natural map $i_\epsilon^*\hat{P}\to \hat{P}$ is $\hat{i}_\epsilon: p\in P \mapsto (p,\epsilon)\in \hat{P}$, hence
    \[
    \hat{i}_\epsilon^*\hat{\theta} = (1-\epsilon)\hat{\theta}_0+\epsilon\hat{\theta}_1,
    \]
    which is a connection on $P$ by Lemma \ref{lem:pconn_are_convex}. We then have $\hat{i}_0^*\hat{\theta}=\theta_0$ and $\hat{i}_1^*\hat{\theta} =\theta_1$. By Lemma \ref{lem:CW_natural} above, $\kappa^\mathrm{CW}_{P,\theta_0} = i_0^*\hat{\kappa}^\mathrm{CW}$ and $\kappa^\mathrm{CW}_{P,\theta_1} = i_1^*\hat{\kappa}^\mathrm{CW}$. Since $i_0$ and $i_1$ are smoothly homotopic, they induce the same map in de Rham cohomology, and we are done. 
\end{proof}

The Chern-Weil map is also invariant under isomorphisms of principal bundles over the same
base. Together with
Proposition~\ref{prop:CW_indep_of_theta} this makes $\kappa^{\mathrm{CW}}$ an
invariant of the isomorphism class of $P$.

\begin{lemma}\label{lem:CW_iso_invariant}
    Let $(\Sigma,\omega)$ be a Lie groupoid with Cartan connection, let
    $F:P\to P'$ be an isomorphism of principal $\Sigma$-bundles over $X$, and let
    $\theta'$ be a principal $(\Sigma,\omega)$-connection on $P'$. Then $F^*\theta'$
    is a principal $(\Sigma,\omega)$-connection on $P$ and
    \[
        \kappa^{\mathrm{CW}}_{(P,F^*\theta')}=\kappa^{\mathrm{CW}}_{(P',\theta')} .
    \]
\end{lemma}
\begin{proof}
    As in Lemma~\ref{lem:CW_natural}, it is straightforward to see that $F^*\theta'$ is a principal $(\Sigma,\omega)$-connection.

    Write $\mu$ and $\mu'$ for the momentum maps of $P$ and $P'$. By equivariance, $\mu'\circ F = \mu$ and 
    \[
    \Theta'_P = \Theta'_{P'}\circ dF
    \]
    where the right hand side denotes $(\an_{F^*\theta'},F^*\theta')$ and $\Theta'_{P'}$ denotes $(\an_{\theta'},\theta')$; moreover, directly from the definition of Maurer-Cartan expression one has $\mc_{F^*\theta'} = F^*\mc_{\theta'}$. Thus $\kappa^{\Theta'_P} = F^*\kappa^{\Theta'_{P'}}$, which concludes.
\end{proof}

\section{Examples}\label{sec:examples}

The reader will have no difficulty checking that, when $(\Sigma,\omega)$ is a Lie group $G$ with its Maurer-Cartan form, our Chern-Weil theory recovers the classical theory; in fact, our construction is designed to generalize the construction of the Chern-Weil map from \cite{KamTond1975}. Thus, we move to other examples.

\subsection{Lie group bundles and families of principal group bundles}\label{subs:LieGroupBundles}
    We first consider Lie groupoids $\Sigma\tto M$ whose source and target map coincide. Such Lie groupoids are the same as \textbf{bundle of Lie groups}: a surjective submersion $\pi: \mathbb{G} \to M$ such that all fibers carry a Lie group structure that varies smoothly with respect to $M$. Indeed, any bundle of Lie groups carries a clear Lie groupoid structure and, if $\Sigma\tto M$ is a Lie groupoid satisfying $s=t$, one gets a bundle of Lie groups by setting $\mathbb{G}:=\Sigma$ and $\pi:=s=t$. 

    The Lie algebroid 
    of a bundle of Lie groups has zero anchor map. Consequently, the Lie bracket on its sections is $C^\infty(M)$-linear and comes from a fiberwise Lie bracket. Thus, $A\to M$ is a \textbf{bundle of Lie algebras}, and will be denoted by $A_{\mathbb{G}}\to M$. 

    Among bundle of Lie groups one has the distinguished class of \textbf{Lie group bundles}, defined by requiring that $\pi:\mathbb{G} \to M$ is a fiber bundle whose fiber is a fixed Lie group $G$; in particular, the local trivializations induce isomorphisms of Lie groups on the fibers. The corresponding infinitesimal notion is that of \textbf{Lie algebra bundle}: a bundle of Lie algebras that is a fiber bundle whose fiber is a fixed Lie algebra.

    More precisely, a Lie group bundle with typical fiber a Lie group $G$ is a bundle of Lie groups together with a trivializing atlas $(U_i,\phi_i)_{i\in I}$ such that for all $i\in I$
    \[
    \begin{tikzcd}
        \mathbb{G}|_{U_i} \arrow[dr,"\pi"]\arrow[rr,"\phi_i"] && U_i\times G\arrow[dl,"\pr_1",swap]\\
        & U_i
    \end{tikzcd}
    \]
    is a groupoid morphism. This local trivializability is known to be related to the existence of Cartan connections, as we briefly recall below.

    First recall that an Ehresmann connection $H\subset T\mathbb{G}$ for $\pi$ is \textbf{complete} when, for every curve $\gamma:I\to M$, every $\lambda_0\in I$ and every $g\in\mathbb{G}_{\gamma(\lambda_0)}$, the horizontal lift of $\gamma$ through $g$ at time $\lambda_0$ is defined on the whole of $I$. When $\mathbb{G}$ is Hausdorff, horizontal lifts are unique, being integral curves of a vector field on $\gamma^*\mathbb{G}\subseteq I\times\mathbb{G}$, and we write $\widetilde\gamma_g$ for the maximal horizontal lift of $\gamma$ through $g$ at time $\lambda_0$. See also~\cite[Theorem 5]{delHoyo2016}, which shows that a submersion is locally trivial if and only if it admits a complete connection. 

    \begin{lemma}\label{lem:group_bun_Cart_conn}
        Let $\pi:\mathbb{G}\to M$ be a bundle of Lie groups with Hausdorff arrow space, and let $H\subset T\mathbb{G}$ be an Ehresmann connection for $\pi$. Then $H$ is the kernel of a Cartan connection if and only if
        \[
        \widetilde\gamma_{gh}=\widetilde\gamma_g\,\widetilde\gamma_h
        \]
        for every curve $\gamma:I\to M$, every $\lambda_0\in I$ and all $g,h\in \mathbb{G}_{\gamma(\lambda_0)}$, on every interval containing $\lambda_0$ on which $\widetilde\gamma_g$ and $\widetilde\gamma_h$ are defined.
        In particular, if $H$ is the kernel of a complete Cartan connection, parallel transport along every curve in $M$ is by isomorphisms of Lie groups.
    \end{lemma}
    \begin{proof}
        Suppose that $H=\ker(\omega)$ with $\omega$ a multiplicative form, and that $\widetilde\gamma_g$ and $\widetilde\gamma_h$ are defined on an interval $J$. Since $s=t=\pi$, the curve $\lambda\mapsto\widetilde\gamma_g(\lambda)\widetilde\gamma_h(\lambda)$ is defined on $J$, lies over $\gamma$ and passes through $gh$ at time $\lambda_0$. Its velocity $dm(\widetilde\gamma_g',\widetilde\gamma_h')$ lies in $H\cdot H=H$ by Lemma~\ref{lem:multipl_equiv_notion} and, by uniqueness of horizontal lifts, it is $\widetilde\gamma_{gh}$ on $J$.

        Conversely, suppose that the identity from the statement holds. Choosing $\xi\in H_g$ and $\zeta\in H_h$ with $d\pi(\xi)=d\pi(\zeta)$ and  $\gamma$ with $\gamma(0)=\pi(g)$ and $\gamma'(0)=d\pi(\xi)$, we have $\xi=\widetilde\gamma_g'(0)$ and $\zeta=\widetilde\gamma_h'(0)$. Differentiating the identity from the statement at $0$, we have $dm(\xi,\zeta)=\widetilde\gamma_{gh}'(0)\in H$. On the other hand for $h=g^{-1}$ the identity yields $\widetilde\gamma_{g^{-1}}=\widetilde\gamma_g^{-1}$ and, again differentiating at $0$ $di(\xi)=\widetilde\gamma_{g^{-1}}'(0)\in H$. Hence, both $H\cdot H\subseteq H$ and $di(H)\subseteq H$ hold. To obtain the reverse inclusions, one observes $\xi=dm(\xi,\upsilon)$, with $\upsilon$ the vector tangent to the unit section over $d\pi(\xi)$ and that the manifold of units is tangent to $H$ -- as one sees by evaluating the identity from the statement at $g=h=1_{\gamma(\lambda_0)}$. By Lemma~\ref{lem:multipl_equiv_notion}, $H$ is the kernel of a Cartan connection, and we are done.
    \end{proof}

    In fact, one also has the following
    \begin{lemma}[{\cite[Corollary 2.19]{FerMar2023} for the first case; \cite[Proposition 4.4]{LaurStienXu2009} and \cite[Proposition 2.18]{FerMar2023} for the second}]\label{lem:mult_conn_complete}
        Let $\pi:\mathbb{G}\to M$ be a bundle of Lie groups with Hausdorff total space, and let $\omega$ be a Cartan connection on $\mathbb{G}$. If the fibres of $\pi$ are connected, or if $\mathbb{G}$ is a Lie group bundle, then $\omega$ is complete.
    \end{lemma}

    \begin{proposition}[{\cite[Proposition 2.18 and Corollary 2.19]{FerMar2023}}]\label{prop:CC_Lie_group_bundles}
        Any Lie group bundle admits a Cartan connection. Conversely, let $\pi:\mathbb{G}\to M$ be a bundle of Lie groups over a connected $M$, with Hausdorff total space, carrying a Cartan connection $\omega$. If $\omega$ is complete, or if the fibres of $\pi$ are connected, then $\mathbb{G}$ is a Lie group bundle.
    \end{proposition}

    \begin{proof}
        The connection forms of~\cite{FerMar2023} are $\mathfrak{g}$-valued, while ours take values in $t^*A$; the first half is therefore proved here, the two conventions differing by the identifications $\ell_i$ below.

        Let $\pi:\mathbb{G}\to M$ be a Lie group bundle with typical fiber $G$. Let $\omega^G_\mathrm{MC}$ be the Maurer-Cartan form on $G$, and let $(U_i,\phi_i)_{i\in I}$ be an atlas of trivialisations which are group isomorphisms on the fibres. Put $f_i:=\pr_2\circ\phi_i$, and let $\ell_i(x):A_x\to\mathfrak{g}$ be the derivative at the identity of $\phi_i|_{\mathbb{G}_x}$. A morphism of Lie groups $\alpha:K\to G$ satisfies $\alpha\circ L_k=L_{\alpha(k)}\circ\alpha$, hence
        \[
        \alpha^*\omega^G_\mathrm{MC}=d\alpha_e\circ\omega^K_\mc ,
        \]
        with $\omega^K_\mc$ the Maurer-Cartan form of $K$. Applied to $\phi_i|_{\mathbb{G}_x}$, this shows that
        \[
        \omega_i:=\ell_i^{-1}\circ f_i^*\omega_\mathrm{MC}\in\Omega^1\bigl(\mathbb{G}|_{U_i},t^*A\bigr)
        \]
        restricts to each fibre as its Maurer-Cartan form, so $\omega_i$ is a connection on $\mathbb{G}|_{U_i}$. Its kernel is tangent to the submanifolds $\phi_i^{-1}(U_i\times\{a\})$, which are closed under multiplication and inversion, as well as tangent to the unit manifold when $a$ is the unit. Thus, $\omega_i$ is multiplicative by Lemma~\ref{lem:multipl_equiv_notion}. In the multiplicativity equations
        \[
        m^*\omega_i=\pr_1^*\omega_i+g\cdot\pr_2^*\omega_i
        \]
        the representation $\cdot$ does not depend on $\omega_i$, because by $s=t$ it is just the adjoint action of the fibers. Hence, if $(\chi_i)_{i\in I}$ is a partition of unity subordinate to $(U_i)_{i\in I}$ and we set
        \[
        \omega := \sum\limits_{i}\pi^*(\chi_i)\, \omega_i,
        \]
        we see that $\omega$ restricts to $\omega^G_\mathrm{MC}$ on the fibers and, using that $\pi\circ m=\pi\circ\pr_1=\pi\circ\pr_2$, it is multiplicative.
        
        Conversely, by Lemma~\ref{lem:mult_conn_complete} one can assume that $\omega$ is complete, and follow ~\cite[Proposition 2.18]{FerMar2023}: on an open subset of the base supporting a chart, parallel transport yields a trivialization of $\mathbb{G}$, which is by isomorphisms of Lie groups due to Lemma~\ref{lem:group_bun_Cart_conn}.
    \end{proof}
    \begin{remark}\label{rmk:CC_Lie_group_bundles_hypothesis}
        Bundles of discrete groups provide counterexamples to the converse statement in Proposition~\ref{prop:CC_Lie_group_bundles}; see the discussion after~\cite[Corollary 2.19]{FerMar2023} and~\cite[Example 5.8]{LauMarcut2026}. For instance, over $M=\mathbb{R}$, one has the bundle of Lie groups
        \[
        \mathbb{G}\to \mathbb{R}\times
        \]
        with fibre 
        \[
        \mathbb{G}_t = \{e\},\quad t\leq 0
        \]
        and
        \[
        \mathbb{G}_t = \mathbb{Z}/2,\quad t>0
        \]
        It is a bundle of Lie groups with Hausdorff total space and zero dimensional fibres, so $A=0$ and $\omega=0$ is a Cartan connection, but it is obviously not a Lie group bundle.
    \end{remark}

    \begin{proposition}[{\cite[Proposition 2.10 and Corollary 2.12]{FerMar2023}}]\label{prop:LAB_from_mult_conn}
        Let $\pi:\mathbb{G}\to M$ be a bundle of Lie groups carrying a
        Cartan connection. Then its Lie algebroid
        $A_{\mathbb{G}}\to M$ is a Lie algebra bundle.
    \end{proposition}
    \begin{proof}
        Let $D$ be the infinitesimally Cartan connection on $A_{\mathbb{G}}$
        induced by the Cartan connection via
        Proposition~\ref{prop:diff_conn}. The anchor of $A_{\mathbb{G}}$ vanishes,
        so $\nabla^{TM}$ of Proposition~\ref{prop:adj_rep} vanishes too; moreover, $\cbas^D=0$ yields
        \[
        D_X[\alpha,\beta]=[D_X\alpha,\beta]+[\alpha,D_X\beta].
        \]
        Consequently, the parallel transports induced by $D$ are isomorphisms of Lie
        algebras, and can be used to locally trivialize $A_{\mathbb{G}}$. 
    \end{proof}
Let us also give a characterization of Lie group bundles admitting a \textit{flat} Cartan connections. We will show that the existence of a flat Cartan connections corresponds to a reduction of the structure group of the fiber bundle $\mathbb{G}\to M$ from $\mathrm{Diff}({G})$ to $\mathrm{Aut}(G)$ with the discrete topology -- i.e. to the existence of
a trivializing atlas $(U_i,\phi_i)_{i\in I}$
    \[
    \begin{tikzcd}
        \mathbb{G}|_{U_i} \arrow[dr,"\pi"]\arrow[rr,"\phi_i"] && U_i\times G\arrow[dl,"\pr_1",swap]\\
        & U_i
    \end{tikzcd}
    \]
    such that the transition functions
    \[
    \begin{tikzcd}
    U_{ij}\times G \arrow[dr,"\pr_1"]\arrow[rr,"\phi_j\circ\phi_i^{-1}"] && U_{ij}\times G\arrow[dl,"\pr_1",swap]\\
    &U_{ij}:= U_i\cap U_j
    \end{tikzcd}
    \]
    restrict to the fibers as automorphisms of $G$ which are \textbf{locally constant} in the base point. We call such an atlas a \textbf{flat structure} on $\mathbb{G}$.
    Of course when such an atlas exists, differentiating its transition functions at the identity gives an atlas of $A_{\mathbb{G}}\to M$ whose transition functions are Lie algebra automorphisms of $\mathfrak{g}$, locally constant in the base point.

    \begin{proposition}\label{prop:flat_CC_Lie_group_bundles}       
        Any Lie group bundle equipped with a flat structure carries a canonical flat Cartan connection, determined by that flat structure. Conversely, a Lie group bundle $\pi:\mathbb{G} \to M$ with a flat Cartan connection admits a canonical flat structure.
        
    \end{proposition}
    \begin{proof}
        Let $\pi:\mathbb{G} \to M$ be a Lie group bundle with typical fiber $G$, and let $\omega_\mathrm{MC}$ be the Maurer-Cartan form of $G$.
        
        If $\mathbb{G}$ carries a flat structure, let $\mathcal{U}:=(U_i,\phi_i)_{i\in I}$ be an atlas defining it, and write $\phi_j\circ\phi_i^{-1}(x,a)=(x,\psi_{ji}(a))$, with $\psi_{ji}\in\mathrm{Aut}(G)$ constant on each connected component of $U_i\cap U_j$. Put $f_i:=\pr_2\circ\phi_i$, and let $\ell_i(x):A_x\to\mathfrak{g}$ be the derivative at the identity of the group isomorphism $\phi_i|_{\mathbb{G}_x}:\mathbb{G}_x\to G$. Set
        \[
        \omega_i:=\ell_i^{-1}\circ f_i^*\omega_\mathrm{MC}\in\Omega^1\bigl(\mathbb{G}|_{U_i},t^*A\bigr).
        \]
        By the proof of Proposition~\ref{prop:CC_Lie_group_bundles}, each $\omega_i$ is a Cartan connection on $\mathbb{G}|_{U_i}$ whose kernel is tangent to the submanifolds $\phi_i^{-1}(U_i\times\{a\})$, and $\psi_{ji}^*\omega_\mathrm{MC}=d(\psi_{ji})_e\circ\omega_\mathrm{MC}$. On $U_i\cap U_j$ we have $f_j=\psi_{ji}\circ f_i$ and $\ell_j=d(\psi_{ji})_e\circ\ell_i$, with $\psi_{ji}$ locally constant, so
        \[
        \omega_j=\ell_j^{-1}\circ d(\psi_{ji})_e\circ f_i^*\omega_\mathrm{MC}=\omega_i ,
        \]
        and the $\omega_i$ glue to a Cartan connection $\omega$ on $\mathbb{G}$. Its kernel is involutive, so $\omega$ is flat. Observe that, for any other atlas compatible with the flat structure, the induced connection $\omega$ is the same.
        

        Conversely, let $\omega$ be a flat Cartan connection on $\mathbb{G}$ and $H:=\ker(\omega)$ the corresponding involutive distribution. A Lie group bundle has Hausdorff total space by local trivializability, so by Lemma~\ref{lem:mult_conn_complete} $\omega$ is complete. Using parallel transport on an open subset of the base, following~\cite[Proposition 2.18]{FerMar2023}, one obtains, for each $x\in M$ an open subset $U_x\subset M$ and a trivialization $\Phi_x:U_x\times\mathbb{G}_x\to\mathbb{G}|_{U_x}$ which is an isomorphism on each fiber. Each section $y\mapsto\Phi_x(y,g)$ lies in a leaf of $H$. Fixing group isomorphisms $\mathbb{G}_x\cong G$, the inverses of the maps $\Phi_U$ yields an atlas of $\mathbb{G}$ whose charts are group isomorphisms on the fibers; the transition functions send each $U_x\cap U_{x'}\times\{a\}$ to some $U_x\cap U_{x'}\times\{a'\}$ (after restricting to connected components). It follows that the transition functions are automorphisms of $G$ locally constant in the base point, as we wanted.
    \end{proof}
    \begin{remark}\label{rmk:flat_CC_Lie_group_bundles}
    Regardless of connectivity and completeness, not every Lie group bundle carries a flat Cartan connections. For instance, one can regard
    \[
    \mathbb{G}=T\mathbb{S}^2\to\mathbb{S}^2
    \]
    as a bundle of abelian groups. A flat Cartan connection would yield a parallelization of $\mathbb{S}^2$.
    \end{remark}
    \begin{remark}
        A Lie group bundle can carry several flat structures, hence several flat Cartan connections. On trivial bundles, this can be achieved by acting with a family of elements of $\mathrm{Aut}(G)$ depending on the base point.
        
        For instance, one has an obvious flat structure on $\mathbb{S}^1\times \mathbb{R}$ whose flat Cartan connection is the canonical one $\mathbb{S}^1$. On the other hand, setting $\psi(\vartheta):=e^{\sin\vartheta}$, $\vartheta\in \mathbb{S}^1$, one can change trivialization by 
        \[
        \phi'(\vartheta,g):=\bigl(\vartheta,\psi(\vartheta)g\bigr).
        \]
        The leaves of the flat Cartan connection corresponding to the associated flat structure are different from the previous one; they are the curves cut out by $\psi(\vartheta)g=c$, for $c\in \mathbb{R}$.
    \end{remark}

    \begin{remark}\label{rmk:LGB_flat_is_trivial_action}
        As we already saw in Proposition \ref{prop:flat_Cartan_action}, a flat infinitesimally Cartan connection on a Lie algebroid over a simply connected manifold makes it into an action algebroid. The groupoid counterpart, Proposition~\ref{prop:flat_Cartan_action_global}, asks in addition that the source fibres be simply connected and that the connection be complete, and yields an action groupoid. In the case of Lie group and Lie algebra bundles, the action is the trivial one.
    \end{remark}

    We now move to the Chern-Weil construction. First, let us fix a Lie algebra bundle $A_{\mathbb{G}}$ and an infinitesimally Cartan connection $D$. The structure described in the previous sections simplifies as follows.
    
    \vspace{.1cm}
    
    The adjoint representation on $\rho: A_{\mathbb{G}}\to TM$ reduces to the fiberwise adjoint representation of $A_{\mathbb{G}}$ on itself, since $\nabla^{TM}$ from Proposition \ref{prop:adj_rep} vanishes identically. It is also worth noticing that this representation is completely independent from $D$.

    \vspace{.1cm}
    
    The corresponding bowtie almost algebroid $TM\bowtie A_{\mathbb{G}}$ has the following structure operations
        \[
            \rho_\oplus (X, \alpha) = X
        \]
        \[
        [(X,\alpha), (Y,\beta)] = ([X,Y],[\alpha,\beta]+D_X\beta-D_Y\alpha)
        \]
        and its isotropy bundle is precisely $A_{\mathbb{G}}$ with its Lie bracket -- which, as above, does not depend on $D$. 
        
        Accordingly, the associated Chevalley-Eilenberg operator $d_\oplus$ reduces to the (graded) sum of the differential of $D$ and the Chevalley-Eilenberg differential of $A$ (which acts fiberwise). 

    \vspace{.1cm}
        
    Finally, the differential \eqref{eq:diff_Weil} on $W^\bullet (A_{\mathbb{G}},D)$ is the sum of the following derivations (described via their action on generators $\eta\in \Omega^1(M)$, $\alpha\in \Omega^1(A_{\mathbb{G}})$ and $\phi \in \Gamma(S^1A_{\mathbb{G}}^*)$):
        \begin{itemize}
        \item the differential $d_{\overline{W}}$: 
        \begin{equation}\label{eq:diff_Weil_of_alg}
        d_{\overline{W}}(\eta) = 0,\ d_{\overline{W}}(\alpha) = \iota(\alpha)+d_\mathrm{CE}\alpha,\ d_{\overline{W}}(a) = \iota(d_{\overline{\mathrm{cad}}}(\iota(a)))
        \end{equation}
        where $d_{\overline{\mathrm{cad}}}$ denotes the differential associated to the coadjoint representation of the fiber of $A_{\mathbb{G}}$;
        \item the de Rham differential of $M$;
        \item the derivations $d_{D^*}$ associated to the connection $D$ -- acting as the $0$ map on $\Omega^1(M)$;
        \item the curvature term
        \begin{equation}\label{eq:diff_curv_lgb}
        d_\mathrm{curv}(\eta) = 0,\ d_\mathrm{curv}(\alpha) = 0,\ d_\mathrm{curv}(\phi) = \langle \phi, \curv^D \rangle.
        \end{equation}
    \end{itemize}
    When $\omega$ is flat, the last derivation vanishes identically and $d_{D^*}$ becomes a differential. 

The differential $d_{\overline{W}}$ above is in fact a differential on the algebra
\[
\overline{W}^{k}(A_{\mathbb{G}}) = \Gamma\left(\bigoplus\limits_{k=p+2q}\Lambda^pA_{\mathbb{G}}^*\otimes S^qA_{\mathbb{G}}^*\right)
\]
which can be thought of as the "Weil algebra bundle" of $A_{\mathbb{G}}$ and comes with a natural map of dga's
\[
(W^\bullet(A_{\mathbb{G}}), d_W) \to (\overline{W}^\bullet(A_{\mathbb{G}}),d_{\overline{W}}).
\]

Now let $\mathbb{G}\to M$ be a Lie group bundle with fiber $G$, and $A_{\mathbb{G}}$ be its Lie algebra bundle; as usual $\mathfrak{g}$ denotes the Lie algebra of $G$, which is the typical fiber of $A_{\mathbb{G}}$.

Let $\pi:\mathbb{P}\to \mathbb{X}$ be a principal $\mathbb{G}$-bundle with moment map  $\mu: \mathbb{P}\to M$. Throughout this subsection $\mu$ is assumed to be a submersion, so that its fibres are manifolds. The notation is motivated by the fact that the fiber $G$ of $\mathbb{G}$ acts principally over the fibers $P_m := \mu^{-1}(m)$. The quotient of the action is the submanifold $X_m:=\pi(\mu^{-1}(m))\subset \mathbb{X}$. In other words, $\mathbb{P}$ can be thought of as a smooth family of principal $G$-bundles.

In this context, it makes sense to define a principal $\mathbb{G}$-connection as a smooth family of principal $G$-connections along the fibers of the momentum map $\mu$. Then, we define a \textbf{principal $\mathbb{G}$-connection} to be a foliated one form $\bar{\theta}\in \Omega^1_{\mu\mathrm{-fol}}(P,A_{\mathbb{G}})$ which is a principal $G$-connection when restricted to the $\mu$-fibers $P_m$, $m\in M$. Given a principal $G$-connection, the classical connection curvature construction yields, for all $m\in M$, the dga map
\[
\bar{\kappa}_{\bar{\theta}_m} : (W^\bullet (\mathfrak{g}), d_W) \to (\Omega^\bullet(P_m),d_\mathrm{dR})
\]
which, after passing to basic cohomologies, gives the classical Chern-Weil map
\[
\bar{\kappa}^\mathrm{CW}_m: S^\bullet(\mathfrak{g}^*)\to H^{2\bullet}_\mathrm{dR}(X_m).
\]
In fact, if $(\overline{W}^\bullet(A_{\mathbb{G}}), d_{\overline{W}})$ denotes the "Weil algebra bundle" introduced above, then $\bar{\theta}$ induces a dga map
\[
(\overline{W}^\bullet(A_{\mathbb{G}}), d_{\overline{W}}) \to (\Omega^\bullet_{\mu\mathrm{-fol}}(P),d_\mathrm{dR})
\]
by pairing elements of $\Gamma(\Lambda^\bullet A_{\mathbb{G}}^*)$ with $\theta$ and elements of $\Gamma(S^\bullet A_{\mathbb{G}}^*)$ with the Maurer-Cartan expression
\[
\mc_{\bar{\theta}}:= d_\mu\bar{\theta} + \frac{1}{2}[\bar{\theta},\bar{\theta}].
\]
Here, $[\ ,\ ]$ is the Lie bracket on $A_{\mathbb{G}}$ while $d_\mu$ denotes the foliated differential -- notice that $d_\mu\bar{\theta}$ is well defined because $A_{\mathbb{G}}$ is trivial along the $\mu$-fibers. If one assumes that the submanifolds $\{X_m\}_{m\in M}$ form a regular foliation $\mathcal{F}$ in $\mathbb{X}$, then we see that the map above induces a characteristic map with values in foliated cohomology
\[
\bar{\kappa}^\mathrm{CW}:  \Gamma_{\mathbb{G}\mathrm{-inv}}(S^\bullet(A_{\mathbb{G}}^*))\to H^{2\bullet}_\mathcal{F}(\mathbb{X}).
\]


We now explore the interaction between $\bar{\kappa}^\mathrm{CW}$ and the Chern-Weil map induced by a Cartan connection. Below, $\omega$ is a Cartan connection on $\mathbb{G}$ -- see Proposition \ref{prop:CC_Lie_group_bundles}.


\begin{lemma}\label{lem:LGB_connections}
    Let $\mathbb{P}$ be a principal $\mathbb{G}$-bundle over $\mathbb{X}$ such that the foliation $\mathcal{F}$ is regular. 
    The following facts hold true.
    \begin{itemize}
        \item The action of $\mathbb{G}$ on $T\mathbb{P}$ induced by $\omega$ descends to the normal bundle $T\mathbb{P}/T^\mu\mathbb{P}$.
        \item The short exact sequence
        \[
        \Omega^1(T\mathbb{P}/T^\mu\mathbb{P}, \mu^*A_{\mathbb{G}}) \to \Omega^1(\mathbb{P},\mu^*A_{\mathbb{G}}) \to \Omega^1_{\mu\mathrm{-fol}}(\mathbb{P},\mu^*A_{\mathbb{G}})
        \]
        is equivariant with respect to the action induced by $\omega$.
        \item The restriction map
        \[
        \Omega^1(\mathbb{P},\mu^*A_{\mathbb{G}}) \to \Omega^1_{\mu\mathrm{-fol}}(\mathbb{P},\mu^*A_{\mathbb{G}})
        \]
        sends principal $(\mathbb{G},\omega)$-connections to principal $\mathbb{G}$-connections.
        \item If $\bar{\theta}$ is a principal $\mathbb{G}$-connection, the set of extensions of $\bar{\theta}$ to principal $(\mathbb{G},\omega)$ connections is an affine space modeled on the space of $\omega$-invariant elements of $\Omega^1(T\mathbb{P}/T^\mu\mathbb{P}, \mu^*A_{\mathbb{G}})$.
        \item If $\theta$ is a principal $(\mathbb{G},\omega)$-connection extending the principal $\mathbb{G}$-connection $\bar{\theta}$, then $\mc_\theta$ extends $\mc_{\bar{\theta}}$.
    \end{itemize}
\end{lemma}
\begin{proof}
    For the first point, it is sufficient to observe that the action of $\mathbb{G}$ on $T\mathbb{P}$ induced by $\omega$ is given by the formula
    \[
    g\cdot v_p = dm_P(\hor^\omega_g(d\mu(v_p)), v_p),\quad v_p\in T_p\mathbb{P}
    \]
    and hence it restricts to $T^\mu\mathbb{P}$ as the differential of the action of $G$ on the $\mu$-fibers.
    
    Then, the second point follows because the sequence
    \[
    T^\mu\mathbb{P} \to T\mathbb{P} \to T\mathbb{P}/T^\mu\mathbb{P}
    \]
    is equivariant by construction.

    The third point follows from the second. 
    
    For the fourth point, first observe that the difference of two principal $(\mathbb{G},\omega)$-connections extending the same $\bar\theta$ is $\omega$-equivariant. To check that the set of such connections is non-empty, consider any principal $(\mathbb{G},\omega)$-connection $\theta_0$. Consider the difference
    \[
    \bar{\theta}-\bar{\theta}_0 \in \Omega^1_{\mu-\mathrm{fol}}(P,A_\mathbb{G}),
    \]
    where $\bar{\theta}_0$ denotes the restriction to the $\mu$-fibers. The form $\bar{\theta}-\bar{\theta}_0$ vanishes on $\pi$-vertical vectors in $\ker(d\mu)$, and is $\mathbb{G}$-equivariant; hence, it descends to a foliated form $\bar{\eta}\in \Omega^1(\mathcal{F},P[A_\mathbb{G}])$. Any foliated form admits an extension, hence there exists $\eta\in  \Omega^1(\mathbb{X},P[A_\mathbb{G}])$ which restricts to $\bar{\eta}$ along $\mathcal{F}$. The sum
    \[
    \theta=\theta_0+\pi^*\eta
    \]
    coincides with $\bar\theta$ when restricted to the $\mu$-fibers, and is a principal $(\mathbb{G},\omega)$-connection by construction.

    For the last point we just need to check that
    \[
    d_{\mu^*D}\theta|_{\Lambda^2T^\mu P} = d_\mu\bar{\theta}
    \]
    where $d_{\mu^*D}$ is the operator induced by the pullback of the connection $D$ on $A_{\mathbb{G}}$ associated to $\omega$. To that avail we recall that the pullback connection satisfies
    \[
    \mu^*D_X(\mu^*\alpha) = \mu^*(D_{d\mu(X)}\alpha)
    \]
    for any section $\alpha$ of $A_{\mathbb{G}}$ and any $\mu$-projectable vector field $X$, and that sections of the form $\mu^*\alpha$ generate the $C^\infty(\mathbb{P})$-module of sections of $\mu^*A_{\mathbb{G}}$. 
\end{proof}

If $\theta$ is a principal $(\mathbb{G},\omega)$-connection on $P$ that restricts to $\bar{\theta}$, the last point in the previous lemma together with the structure of $(W^\bullet(A_{\mathbb{G}}),d_W)$ described above imply that we have a commutative diagram of dga's
    \[
    \begin{tikzcd}
        (W^\bullet(A_{\mathbb{G}}),d_W) \arrow[r, "\kappa^{\mathrm{CW}}_\Theta"]\arrow[d] & (\Omega^\bullet(P),d_\mathrm{dR})\arrow[d]\\
        (\overline{W}^\bullet(A_{\mathbb{G}}),d_{\overline{W}} ) \arrow[r, "\kappa^{\mathrm{CW}}_{\bar{\theta}}"] &(\Omega^\bullet_{\mu\mathrm{-fol}}(P),d_\mu)
    \end{tikzcd}
    \]
Thus, we deduce
\begin{proposition}\label{prop:LGB_diagram}
    Let $\mathbb{P}$ be a principal $\mathbb{G}$-bundle, $\theta$ be a principal $(\mathbb{G},\omega)$-connection on $P$, and $\bar{\theta}$ be the principal $\mathbb{G}$-connection on $P$ obtained by restriction to the $\mu$-fibers. The following diagram is commutative
    \[
    \begin{tikzcd}
        H^\bullet(W^\bullet(A_{\mathbb{G}})_{\omega\mathrm{-bas}}) \arrow[r, "\kappa^{\mathrm{CW}}"]\arrow[d] & H^\bullet_\mathrm{dR}(X)\arrow[d]\\
        \Gamma_{\mathbb{G}\mathrm{-inv}}(S^\bullet(A_{\mathbb{G}})) \arrow[r, "\bar{\kappa}^{\mathrm{CW}}"] & H^\bullet_\mathcal{F}(X)
    \end{tikzcd}
    \]
\end{proposition}
The map 
$\bar{\kappa}^{\mathrm{CW}}$, packs invariants of the $\mu$-fibers seen as principal $G$-bundles; as a consequence, while independent from $\omega$, it carries less information than $\kappa^\mathrm{CW}$, as the following example shows.

\begin{example} \label{ex:LGB_obstruction}
    Let the Cartan connection $\omega$ on $\mathbb{G}$ be flat. Then, for all $p\in \mathbb{N}$, $S^p(A_{\mathbb{G}}^*)$ is a flat bundle when equipped with the flat connection $d_{D^*}$ induced by $\omega$, and so is the bundle $S^p(A_{\mathbb{G}}^*)_{\mathbb{G}\mathrm{-inv}}$ of fiberwise invariant polynomials. One has
    \[
    H^n(W^\bullet(A_{\mathbb{G}},D)_{\omega-\mathrm{bas}})\cong \bigoplus\limits_{q+2p=n} H^{q}(M,(S^p(A_{\mathbb{G}}^*)_{\mathbb{G}\mathrm{-inv}}, d_{D^*}))
    \]
    When the holonomy of $d_{D^*}$ is trivial, then we get the tensor decomposition
    \[
    H^\bullet(W^\bullet(A_{\mathbb{G}},D)_{\omega-\mathrm{bas}})\cong H^\bullet_\mathrm{dR}(M) \otimes (\ker d_{D^*})_{\mathbb{G}\mathrm{-inv}}.
    \]

    Let now $\mathbb{P}$ be a principal $\mathbb{G}$-bundle with $M=\mathbb{X}$ and $\mu=\pi$, and assume that the holonomy of $d_{D^*}$ is trivial. It follows that the map $\bar{\kappa}^{\mathrm{CW}}$ vanishes identically in positive polynomial degrees. On the other hand, $\kappa^\mathrm{CW}$ acts as the identity on $H^\bullet_\mathrm{dR}(M)$; its non trivial part is the one acting on $(\ker d_{D^*})_{\mathbb{G}\mathrm{-inv}}$, which might be non-zero and represent an obstruction to the trivializability of $\mathbb{P}$. Equivalently, $\kappa^\mathrm{CW}$ obstructs the existence a global section of $\mu$.
    
    One can build simple examples where $\kappa^\mathrm{CW}$ is non-trivial by turning principal group bundles into principal groupoid bundles where the groupoid is a trivial bundle of Lie groups. In this case, the Chern-Weil map obtained from the groupoid carries the same information as the classical Chern-Weil map.

    Indeed, if $P$ is a principal $G$-bundle over $M$, then $P$ is also a principal $\mathbb{G}$-bundle where $\mathbb{G}\to M$ is the second projection
    \[
    G\times M \to M.
    \]
    The cohomology of the basic subcomplex of the Weil algebra of $(G\times M,\pr_1^*\omega_\mathrm{MC})$ becomes
    \[
    H^\bullet_\mathrm{dR}(M) \otimes S^\bullet_G(\mathfrak{g}^*).
    \]
    Moreover, there is a one-to-one correspondence between principal $G$-connections and principal $(G\times M,\pr_1^*\omega^G_{\mathrm{MC}})$-connections on $P$ -- see subsection \ref{subs:ActionsAndCartanModel} below. The map $\kappa^{\mathrm{CW}}$ is just the tensor product of the identity with the classical Chern-Weil map. 


\end{example}

\begin{remark}
In general, the cohomology of $H^\bullet(W^\bullet(A_{\mathbb{G}},D)_{\omega-\mathrm{bas}})$ depends on the Cartan connection.
Consider the trivial Lie group bundle $\mathbb G=S^1\times\mathbb R\to S^1$. Let $e$ be the standard frame of
$A_{\mathbb G}$ and $\alpha=e^*$.
The flat Cartan connections
$\omega_\lambda=(dx+\lambda x\,d\theta)\otimes e$ induce
\[
d_{D^{\lambda *}}(f\alpha^p)
=(f'-p\lambda f)\,d\theta\otimes\alpha^p.
\]
Since the fibers are abelian, all fiberwise polynomials are invariant
under the adjoint action. However, the holonomy on
$S^p(A_{\mathbb G}^*)$ is multiplication by $\exp(p\lambda)$.
For $\lambda\ne0$ and $p>0$, both $H^0$ and $H^1$ of the resulting
local system vanish. Hence, following Example \ref{ex:LGB_obstruction} above,
\[
H^\bullet\bigl(
W^\bullet(A_{\mathbb G},D^\lambda)_{\omega_\lambda\text{-bas}}
\bigr)
\cong
\begin{cases}
H^\bullet_{\mathrm{dR}}(S^1)\otimes\mathbb R[a],&\lambda=0,\\
H^\bullet_{\mathrm{dR}}(S^1),&\lambda\ne0,
\end{cases}
\qquad a=\iota(\alpha),\quad |a|=2.
\]
\end{remark}

\subsection{Actions and the Cartan model}\label{subs:actionGroupoids}\label{subs:ActionsAndCartanModel}
Let $G$ be a Lie group acting on a manifold $M$, $G\ltimes M \tto M$ be the corresponding action groupoid, and $\mathfrak{g}\ltimes M$ be the associated Lie algebroid. 

We have already encountered the canonical Cartan connection on $G\ltimes M$: the pullback $\pr_1^*\omega_\mathrm{MC}$ of the Maurer-Cartan form on $G$ appeared in Proposition \ref{prop:flat_Cartan_action_global}. Its infinitesimal counterpart is the pullback of the standard flat connection on $\mathfrak{g}\ltimes M$, see Proposition \ref{prop:flat_Cartan_action}.

Let us describe the Weil algebra of $\mathfrak{g}\ltimes M$ equipped with such flat connection. As an algebra, it is given by the tensor product
    \[
    \Omega^\bullet (M) \otimes W^\bullet (\mathfrak{g})
    \]
    of the de Rham algebra of $M$ with the Weil algebra of $\mathfrak{g}$. The differential $d_W$ is Kalkman's differential in the BRST model for equivariant cohomology \cite{Kalkman93}:
    \[
    d_\mathrm{Kalk} = d_\mathrm{dR}\otimes \mathrm{id} + \mathrm{id}\otimes d_{\overline{W}}\  +\ d_\mathrm{act}
    \]
    where 
    \[
    d_\mathrm{act} = \sum\limits^k_{i=1} L_{\rho(\epsilon_i)}\otimes \alpha^i - \sum\limits^k_{j=1} \rho(\epsilon_j)\cont \otimes \iota(\alpha^j).
    \]
    Here, $k=\dim(\mathfrak{g})$, $\epsilon_1,\dots,\epsilon_n$ is a basis of $\mathfrak{g}$ with dual basis $\alpha_1,\dots,\alpha^n$, $L$ and $\cont$ denote respectively the Lie derivative and the contraction with a vector field; recall also that the anchor $\rho$ of $\mathfrak{g}\ltimes M$ is the infinitesimal
    action $a_M:\mathfrak{g}\to\mathfrak{X}(M)$.

    The Cartan model for equivariant cohomology appears here as the basic subcomplex of the canonical $\mathfrak{g}$-dga structure carried by $(W^\bullet(\mathfrak{g}\ltimes M),d_\mathrm{Kalk})$. The structure operations of the $\mathfrak{g}$-dga structure are:
    \begin{itemize}
        \item the map
            \[
         \cont: \mathfrak{g} \to \mathrm{Der}^{-1}(W^\bullet(\mathfrak{g}\ltimes M),d_W)
            \]
            sending each $\epsilon$ in $\mathfrak{g}$ to the derivation defined on generators by
            \[
             \alpha = \alpha(\epsilon),\quad \alpha \in \Lambda^1\mathfrak{g}^*
            \]
            and
            \[
            \epsilon \cont \eta = 0,\ \epsilon \cont \phi = 0, \quad \eta \in \Omega^1(M),\ \phi \in S^1\mathfrak{g}^*.
            \]
        \item the map
        \[
         L: \mathfrak{g} \to \mathrm{Der}^{0}(W^\bullet(\mathfrak{g}\ltimes M),d_W)
        \]
        defined by the Cartan-like formula
        \[
        L_\epsilon = [d,\epsilon \cont],\quad \epsilon \in \mathfrak{g}. 
        \]
    \end{itemize}
    The sub-algebra of basic elements with respect to the $\mathfrak{g}$-dga structure above is the algebra $S^\bullet(\mathfrak{g}^*)$-valued $\mathfrak{g}$-invariant forms on $M$ equipped with the Cartan differential. More precisely, basic elements form the algebra
    \[
    \mathcal{C}^\bullet := (\Omega^\bullet(M)\otimes S^\bullet(\mathfrak{g}^*))_{\mathfrak{g}\mathrm{-inv}}
    \]
    where:
    \begin{itemize}
        \item $\eta\in \Omega^\bullet(M)$ belongs to $\mathcal{C}$ if and only if its Lie derivative along fundamental vector fields vanishes;
        \item $\alpha\in S^\bullet(\mathfrak{g}^*)$ belongs to $\mathcal{C}$ if and only if it is invariant under the representation induced by the coadjoint action.
    \end{itemize}
    On $\mathcal{C}^\bullet$ one has the \textbf{Cartan differential} $d_\mathrm{Cart}$ obtained by restricting $d_\mathrm{Kalk}$; explicitly
    \[
    d_\mathrm{Cart}(\chi) :\  \epsilon \mapsto d_\mathrm{dR}(\chi(\epsilon))-\epsilon\cont \chi(\epsilon), \quad \epsilon \in \mathfrak{g}
    \]
    where one interprets $\chi \in \mathcal{C}^\bullet$ and its differential as a forms on $M$ with value in polynomial functions on $\mathfrak{g}$.

    When the infinitesimal action is induced by a group action of $G$ on $M$, one can "integrate" $L$ to the induced action of $G$ on $\Omega^\bullet(M)$ and $ S^\bullet(\mathfrak{g}^*)$, resulting in the $G$-basic subcomplex
    \[
    \big(\mathcal{C}^\bullet_{\mathrm{Cart}}:=\Omega^\bullet(M)\otimes S^\bullet(\mathfrak{g}))_{{G}\mathrm{-inv}}, d_\mathrm{Cart}\big)
    \]
    which is called the \textbf{Cartan model} for the equivariant cohomology of $M$. Indeed, when $G$ is compact the cohomology of the complex above is equal to equivariant cohomology of the $G$-manifold $M$~\cite[\S 4]{GoerZol2019}.

Let $\Sigma=G\ltimes M \tto M$ be the action groupoid associated to an action of the Lie group $G$ on $M$. There is a one to one correspondence between principal $G\ltimes M$-bundles over $X$ and principal $G$-bundle over $X$ together with an $M$-valued $G$-equivariant map. Indeed, if $P$ is a principal $G$-bundle over $X$ and $\mu: P\to M$ is a $G$-equivariant map, then $G\ltimes M$ acts on $P$ along $\mu$ by setting
\[
(g,\mu(p))\cdot p := g\cdot p,\quad g\in G,\ p\in P
\]
and the action is principal over $X$. Conversely, if $P$ is a principal $G\ltimes M$-bundle over $X$ with momentum map $\mu$, then one sets
\[
g\cdot p := (g,\mu(p))\cdot p, \quad g\in G,\ p\in P
\]
and observes that $\mu$ is $G$-equivariant.

Let $\omega=\pr_1^*\omega_\mathrm{MC}$ be the canonical Cartan connection on $G\ltimes M$. It is immediate to see that the induced action of $G\ltimes M$ on $TM$ is simply the linearization of the $G$-actions on $M$.  It follows that the basic subcomplex of the Weil algebra of $G\ltimes M$ induced by $\omega$ is
    \[
    \left(\left(\Omega^\bullet (M) \otimes S^\bullet (\mathfrak{g})\right)_{G\mathrm{-inv}}, d_\mathrm{Cart}\right),
    \]
    i.e. the Cartan model for equivariant cohomology -- see subsection \ref{subs:actionGroupoids} -- whose cohomology we denote by $H^\bullet_G(M)$.

    If $P$ is a principal $G\ltimes M$-bundle, the action induced by $\omega$ on $TP$ is the linearization of the $G$-action on $P$. Thus, a principal $(G\ltimes M,\omega)$-connection can be identified with an ordinary principal connection $\theta\in \Omega^1(P,\mathfrak{g})$ on $P$. The characteristic map 
    \[
    H^\bullet_G(M) \to H^\bullet (X)
    \]
    fits then into the following commutative diagram
    \[
    \begin{tikzcd}
        H^\bullet_G(M)\arrow[r,"\mu^*"]\arrow[rd,"\kappa_\mathrm{CW}",swap] & H^\bullet_G (P)\arrow[d] \\
        & H^\bullet_\mathrm{dR}(X)
    \end{tikzcd}
    \]
    where $\mu^*$ is the pullback of polynomial valued $G$-invariant forms via $\mu$, while the vertical map is induced by the projector to horizontal forms in $P$ with respect to $\theta$ and the classical Chern-Weil construction (in particular, it is an isomorphism for compact $G$).
    
\subsection{Cartan geometries}\label{subs:CartanConnections}
Here we restrict our attention to \textbf{transitive} Lie groupoids $\Sigma \tto M$, i.e. Lie groupoids such that the anchor map $(t,s):\Sigma \to M\times M$ is a surjective submersion. There is an equivalence of categories between transitive Lie groupoids and principal group bundles. At the level of objects, the equivalence can be described as follows.
\begin{itemize}
    \item One associates to a transitive Lie groupoid $\Sigma\tto M$ a principal $G$-bundle $Q\to M$ by choosing $x\in M$, setting the group $G$ to be the \textbf{isotropy group} of $\Sigma$ at $x$, i.e. $G=s^{-1}(x)\cap t^{-1}(x)$ and taking $Q=t^{-1}(x)$. While this assignment depends on the choice of a basepoint, the transitivity of $\Sigma$ ensures that different basepoint result in isomorphic principal bundles.
    \item One associates to a principal $G$-bundle $Q\to M$ its \textbf{gauge groupoid} $\G_Q$, defined as the quotient of the pair groupoid $Q\times Q\tto Q$ with respect to the diagonal action.
\end{itemize}

Building on this equivalence, in \cite[Theorem 3.3]{Cattafi2021} (see also~\cite{Catt2020}) the author characterizes Cartan connections on transitive groupoids in terms of data on principal group bundles. 
Below, we recall the definition of Cartan connection and outline how the correspondence works.

The following definitions appear in \cite{AccDeMelStr25}.
\begin{definition}\label{def:admin_data}
    An \textbf{admissible data} for a Cartan geometry is a pair $(G,V)$ consisting of a Lie group $G$ and a representation $V\in \mathrm{Rep}(G)$ that extends the adjoint representation on the Lie algebra $\mathfrak{g}$ of $G$ -- i.e. there is an inclusion of $G$-representations $\mathfrak{g}\hookrightarrow V$. 
    
    A \textbf{morphism of admissible data} $(G_1,V_1) \to (G_2,V_2)$ is a Lie group morphism $\Phi:G_1\to G_2$ together with a map of representations
    \[
    \Psi:V_1 \to V_2
    \]
    that extends the linearization $\phi:\mathfrak{g}_1\to \mathfrak{g}_2$ of $\Phi$. 
\end{definition}
\begin{definition}\label{def:cartan_connection}
    Let $Q\to M$ be a principal $G$-bundle and $(G,V)$ an admissible data for a Cartan geometry. A \textbf{Cartan connection} on $Q$ is a one form $\Omega\in \Omega^1(Q,V)$ such that
    \begin{itemize}
    \item $\Omega$ is a pointwise isomorphism;
    \item $\Omega$ is compatible with the infinitesimal action, i.e.
    \[
    \Omega\circ a_Q = \mathrm{id};
    \]
    \item $\Omega$ is $G$-equivariant, i.e.
    \[
    \Omega(g\cdot v) = g\cdot \Omega(v), \quad v\in TQ,\ g\in G.
    \]
    \end{itemize}
    The pair $(Q,\Omega)$ is also called a \textbf{Cartan geometry}. 
    
    Given two Cartan geometries $(Q_1,\Omega_1)$ and $(Q_2,\Omega_2)$ with admissible data $(G_1,V_1)$ and $(G_2,V_2)$ a \textbf{morphism of Cartan geometries} $(Q_1,\Omega_1) \to (Q_2,\Omega_2)$ consist of 
    \begin{itemize}
        \item a morphism $(\Phi,\Psi):(G_1,V_1)\to (G_2,V_2)$ called the \textbf{underlying morphism} of admissible data;
        \item  a morphism of principal bundles \[
        \begin{tikzcd}
        Q_1 \arrow[d]\arrow[r,"F"] & Q_2\arrow[d]\\
        M_1 \arrow[r,"f"] & M_2
        \end{tikzcd}
        \]
        whose underlying morphism of Lie groups is $\Phi:G_1\to G_2$ ;
    \end{itemize}
    such that
    \[
    F^*\Omega_2 = \Psi\circ \Omega^1.
    \]
\end{definition}
\begin{remark}\label{rmk:modelGeometry}
    It is customary to require the admissible data $(G,V)$ to be a \textbf{model geometry} -- in the sense of e.g. \cite{Sharpe1997}. That is, one assumes $V$ to carry a Lie bracket $\{\ ,\ \}$ such that $\mathfrak{g}$ is a subalgebra of $\mathfrak{h}:= (V, \{\ ,\ \})$. With such a set up one is able to define the \textbf{curvature} of a Cartan connection $\Omega$ with underlying data $(G,\mathfrak{h})$ by means of the Maurer-Cartan expression
    \[
    d\Omega + \frac{1}{2}\{\Omega,\Omega\}.
    \]
\end{remark}

Recall that the Lie algebroid of the gauge groupoid $\G_Q$ of a principal bundle $Q$ can be identified with the vector bundle $TQ/G\to M$ whose sections are $G$-invariant vector fields, equipped with the Lie bracket of vector fields and the differential of the projection to $M$ as anchor. If $(Q,\Omega)$ is a Cartan geometry with underlying data $(G,V)$, we see that, as a vector bundle over $M$, $TQ/G$ is isomorphic to the bundle 
\[
Q[V] = Q\times V / G.
\]
Moreover 
\begin{lemma}\label{lem:QV_representation}
    Let $(Q,\Omega)$ be a Cartan geometry with underlying data $(G,V)$. Then the associated vector bundle $Q[V]$ is naturally a representation of $\mathcal{G}_Q$, where for $[p:q]\in \G_Q$ and $[(q,v)]\in Q[V]$,
        \[
        [p:q] [(q,v)] = [(p,v)] \in Q[V].
        \]
\end{lemma}

The following Proposition is proven in~\cite[Theorem 3.3]{Cattafi2021}; see also~\cite{Catt2020}.

\begin{proposition}\label{prop:gauge_construction}
    Let $(Q,\Omega)$ be a Cartan geometry with underlying data $(G,V)$. The one form
    \[
    (v_p,v_q) \mapsto \Omega(v_p)-\Omega(v_q)
    \]
    descends to the gauge groupoid $\mathcal{G}_Q$ as a Cartan connection $\omega\in \Omega^1(\mathcal{G}_Q,t^*Q[V])$.
    Moreover, local automorphisms of $(Q,{\Omega})$ are in one to one correspondence with local bisections of $\G_Q$ tangent to $\ker({\omega})$.
\end{proposition}



\begin{example}\label{exm:Klein_geometries}
A \textbf{Klein geometry} is a pair of Lie groups $(G,H)$ such that $G\hookrightarrow H$ is an embedded Lie subgroup. The pair induces a principal $G$-bundle by taking $H$ and action of $G$ by left multiplication; the base is the right coset space $G\backslash H$. Then, the {right invariant} Maurer-Cartan form $\omega^H_\mathrm{MC}$ {$(v_h)=dR_{h^{-1}}(v_h)$} of $H$ {satisfies $\omega^H_\mathrm{MC}\circ a_Q=\id$ for the left $G$-action, and} defines a Cartan connection. In this case, the underlying admissible data is $(\mathfrak{g},\mathfrak{h})$; hence it is canonically model geometry, see Remark\ref{rmk:modelGeometry}. The curvature of the Cartan connection $\omega^H_\mathrm{MC}$ is its Maurer-Cartan expression, and thus it vanishes. One can show that flat Cartan geometries are locally of this type -- see \cite{Sharpe1997}.

As discussed in \cite{AccDeMelStr25}, the gauge groupoid of a Klein geometry is canonically isomorphic to the action groupoid  $(H\ltimes G\backslash H)$ where the left action of $G$ on $G\backslash H$ is 
\[
h\cdot [k] = [k\cdot h^{-1}], \quad k,h\in H. 
\]
Under such isomorphism the Cartan connection induced by $\omega^H_\mathrm{MC}$ via Proposition \ref{prop:gauge_construction} corresponds to the canonical Cartan connection $\pr_1^*\omega^H_\mathrm{MC}$ that appeared in Proposition \ref{prop:flat_Cartan_action_global} and subsection \ref{subs:ActionsAndCartanModel}.
\end{example}

Below we describe a sub-algebra of the algebra $S^\bullet(TQ/G)^*$ of polynomials in on $TQ/G$ that consists of closed elements in the Weil algebra of $(\G_Q,\omega)$; such algebra first appeared in \cite{AccDeMelStr25} in relation to cohomological invariants of Cartan geometries.

We define $S^\bullet_\mathrm{char}(TQ/G)^*$ as the sub-algebra of $S^\bullet(TQ/G)^*$ consisting of polynomials on $TQ/G$ that are
\begin{itemize}
    \item invariant under the representation of $\G_Q$ on $S^\bullet(TQ/G)^*$ -- i.e. for all $g\in \G_Q$ it holds
    \[
    g\cdot \gamma_{s(g)} = \gamma_{t(g)};
    \]
    \item annihilated by $\mc_\omega$ -- i.e. if $\gamma \in \Gamma(S^\bullet(TQ/G)^*)_\mathrm{bas}$ then
    \[
    \gamma\circ (\mc_\omega,-,\ldots,
    -) = 0 \in \Omega^2(\G_Q)\otimes \Gamma(S^{\bullet -1}((TQ/G)^*))
    \] 
    \item closed under the differential of the connection on $S^\bullet(TQ/G)^*$ induced by the infinitesimal connection associated to $\omega$.
\end{itemize}

Then, if we denote by $D$ the infinitesimal connection associated to $\omega$, we have
\begin{lemma}\label{lem:char_polys_are_closed}
    The sub-algebra $S^\bullet_\mathrm{char}(TQ/G)^*\subset W^\bullet(TQ/G,D)$ consists of elements that are closed under the differential $d_W$ \eqref{eq:diff_Weil}.
\end{lemma}
\begin{proof}
    Recall that $d_W$ is given by
    \[
    d_W = d_\iota+d_\mathrm{cad}+d_\mathrm{curv}
    \]
    where the three operators on the right hand side act on generators as shown in equations \eqref{eq:diff_shift}, \eqref{eq:diff_ad}, and \eqref{eq:diff_curv} respectively.

    First, any element of $S^\bullet(TQ/G)^*\subset W^\bullet(TQ/G,D)$ is closed under $d_\iota$. 
    
    Moving to $d_\mathrm{cad}$, we see that it acts on elements of $S^\bullet(TQ/G)^*$ as the sum of the differential of the connection induced by $D$ and the differential induced by the adjoint representation of $TQ/G$ on itself induced by $\omega$. Elements in $S^\bullet_\mathrm{char}(TQ/G)^*$ are closed under the differential induced by $D$ by definition; furthermore, they are invariant under the action of $\G_Q$ on $S^\bullet(TQ/G)^*$, which differentiates to the $TQ/G$-connection on $S^\bullet(TQ/G)^*$ induced by the adjoint. Hence, elements in $S^\bullet_\mathrm{char}(TQ/G)^*$ are closed under $d_\mathrm{cad}$.
    
    Finally, $d_\mathrm{curv}$ pairs an element of $TQ/G^*$ with the curvature $\curv^D$ of $D$. By Proposition \ref{prop:curv_alg_is_lin_curv_gpd}, $\curv^D$ is the linearization of $\Curv^\omega$; then, the fact that elements of $S^\bullet(TQ/G)^*$ are annihilated by $\mc_\omega$ combined with Proposition \ref{prop:mc_expr_vanishing} implies that they are closed under $d_\mathrm{curv}$.
\end{proof}

\begin{remark}\label{rmk:flat_char_polys_are_usual}
    The condition involving $\mc_\omega$ is empty when dealing with a flat groupoid. Consequently, when $(Q,\Omega)\cong (H,\omega_\mathrm{MC}^H)$ is a Klein geometry -- Example \ref{exm:Klein_geometries} above -- one sees that $S^\bullet_\mathrm{char}(TQ/G)^*\cong S_H^\bullet(\mathfrak{h})$, the algebra of $H$-invariant polynomials on $H$; see also \cite{AccDeMelStr25}.
\end{remark}

As recalled in Subsection \ref{subs:CartanConnections}, transitive groupoids with Cartan connections can be described in terms of Cartan geometries -- see Proposition \ref{prop:gauge_construction}. The following proposition from \cite{AccDeMelStr25} explains how principal bundles with connection behave under such correspondence; we state it with the two conditions on $\mu$ and $\theta$ that the correspondence requires.

Below, $(Q,\Omega)$ denotes the Cartan geometry associated to $(\Sigma,\omega)$, and $(G,V)$ its underlying admissible data -- see Definition \ref{def:admin_data}.
\begin{proposition}[{\cite[Proposition 3.5]{AccDeMelStr25}}]\label{prop:ADS_correspondence}
    Let $(Q,\Omega)\to M$ be a Cartan geometry with underlying data $(G,V)$, and let $(\G_Q,\omega)$ denote its gauge groupoid equipped with the Cartan connection $\omega$ induced by $\Omega$ -- see Proposition \ref{prop:gauge_construction}. There is a one to one correspondence between isomorphism classes of Cartan geometries with underlying data $(G,V)$ and isomorphism classes of principal $(\mathcal{G}_Q,{\omega})$-bundles $(P,\theta)$ whose momentum map $\mu$ is a submersion and which satisfy $\ker(\theta)\oplus\ker(d\mu)=TP$.
\end{proposition}
Concretely, one associates to a principal $(\G_Q,\omega)$-bundle $(P,\theta)$ with moment map $\mu$ the Cartan geometry $(\mu^{-1}(x),-\theta|_{\mu^{-1}(x)})$, where $A_x$ is identified with $V$ through $\Omega$ at a point of $Q$ over $x$; a different choice of basepoint $x\in M$ yields an isomorphic geometry. Conversely, a Cartan geometry $(R,\chi)$ with underlying data $(G,V)$ induces the principal $(\G_Q,\omega)$-bundle $(P^\chi,\theta^\chi)$ where 
\[
P^\chi := (Q\times R)/G
\]
and $\theta^\chi$ is the form induced on the quotient by
\[
\pr_Q^*\Omega - \pr_R^*\chi \in \Omega^1(Q\times R,V).
\]

In \cite{AccDeMelStr25} the authors of the present paper expanded upon the classical Chern-Weil construction for Cartan geometries from \cite{AleMichCartan} by exploiting the above Proposition, and introducing a characteristic map for principal $(\G_Q,\omega)$-bundles. Such map is the restriction of the Chern-Weil map from Definition \ref{def:main_CW_map}
after restricting to the sub-algebra of the domain $S^\bullet_\mathrm{char}(TQ/G)^*$.

Indeed, combining Theorem \ref{thm:map_on_basics}, Definition \ref{def:main_CW_map}, and Proposition \ref{prop:ADS_correspondence}, we have the following corollary, which is Theorem 1.1 in \cite{AccDeMelStr25}.
\begin{corollary}\label{cor:ADS_map}
    Let $(Q,\Omega)$ and $(P,\Omega')$ be Cartan geometries over $M$ and $B$ respectively, with the same underlying admissible data. Then, there is a characteristic map
    \[
    \kappa_P: S^\bullet_\mathrm{char}(TQ/G)^* \to H^\bullet_\mathrm{dR}(B)
    \]
    which coincides with the classical characteristic map of $(P,\Omega')$ \cite[2.10, 2.11 and 3.4]{AleMichCartan} when $(Q, \Omega)$ is the Cartangeometry associated to a Klein geometry.
\end{corollary}

\subsection{Jet groupoids of Lie pseudogroups}\label{subs:infinite_jets}

We already encountered the first prolongation $J^1\Sigma \tto M$ of a Lie groupoid $\Sigma \tto M$: the space of arrows of $J^1\Sigma$ consists of first jets of bisections of $\Sigma$. By taking higher jets, one discovers a tower of Lie groupoids 
\[
\dots \to J^{k+1}\Sigma \overset{\pr}{\rightarrow} J^{k}\Sigma \to \dots \to J^1\Sigma \overset{\pr}{\rightarrow} \Sigma \tto M
\]
where every arrow is a surjective and submersive Lie groupoid map. One also has a groupoid isomorphisms
\[
J^\infty\Sigma \cong \underset{\leftarrow}{\mathrm{lim}} (J^k\Gamma)
\]
between the groupoid of infinite jets of bisections of $\Sigma$ and the projective limit of the tower above. Consequently, $J^\infty\Sigma \tto M$ carries a canonical structure of \textbf{profinite Lie groupoid}. Essentially, such a structure is a kind of tame infinite dimensional smooth structure that allows one to extend much of the basics of differential geometry, as well as the theory Lie theory of groupoids, to $J^\infty\Sigma\tto M$. Here, we limit ourselves to recall a few key concepts; we refer to \cite[Appendix A]{AccCr2023} for a complete discussion.

\begin{itemize}
    \item A \textbf{tower of manifolds} is a sequence of finite dimensional manifolds $M_k$ together with surjective submersions $p^{k,k-1}:M_k\to M_{k-1}$.
    \item A \textbf{(normal) pf-atlas} on a set $M$ together with a bijection $M\to \underset{\leftarrow}{\mathrm{lim}} (M_k)$ into the inverse limit of a tower of manifolds, and a pf-manifold structure on $M$ can be appropriately defined as an equivalence class of such pf-atlases (Definitions A.1 and A.3 in \cite{AccCr2023}); below, we always work with a canonical pf-atlas. A pf-manifold carries a tangent bundle is $TM=\varprojlim TM_k$ and a module of vector fields $\mathfrak{X}(M)\cong\mathrm{Der}(C^\infty(M))$ (\S A.3 in \cite{AccCr2023}).
    \item A \textbf{direct pf-vector bundle} $E\to M$ can be defined as a pf-manifold admitting an atlas consisting of a tower of vector bundles over a tower of bases (\cite[Definition A.14]{AccCr2023}). One example is given by the infinite jet $J^\infty A=\varprojlim J^kA$ of the Lie algebroid $A$ of a Lie groupoid $\Sigma$. A \textbf{$q$-form on a pf-manifold $M$ with coefficients} in a direct pf-vector bundle $E$ is a $q$-alternating $C^\infty(M)$-multilinear map $\mathfrak{X}(M)^{\times q}\to\Gamma(E)$ that is continuous for the $\mathfrak{f}$-adic topology (\cite[Definition A.20]{AccCr2023}). For forms with real coefficients, this implies $\Omega^\bullet (M) = \bigcup\limits_{k} \Omega^\bullet(M_k)$, which can also be taken as an alternative definition.
    \item A \textbf{strict Lie pf-groupoid} $\Sigma\tto M$ is a pf-manifold together with an atlas consisting of a tower of Lie groupoids $\Sigma_k\tto M_k$ related by surjective and submersive Lie groupoid maps $p^{k,k-1}:\Sigma^k\to \Sigma^{k-1}$ (Definition A.34); it is \emph{proper}, respectively has \emph{cohomologically $l$-connected fibres}, when the $\Sigma_k$ may be chosen to have such properties. Strict actions, \textbf{strict principal bundles}, and strict representations for a strict pf groupoid are defined via the same principle (\S A.6 in \cite{AccCr2023}): they are a tower of actions (respectively principal bundles, representations) such that the tower projections are morphisms of actions (respectively principal bundles, representations).
    \item For a strict pf-groupoid, $A=\mathrm{Lie}(\Sigma)$ is defined as in the finite dimensional case; it is a \textbf{strict pf-algebroid}, and $A=\varprojlim A_k$, with $A_k=\mathrm{Lie}(\Sigma_k)$ (\S A.6 in \cite{AccCr2023}). 
    \item The Lie functor on representations can be generalized \cite[\S A.6]{AccCr2023}.
\end{itemize}
In particular, the infinite jet $J^\infty\Sigma$ of a Lie groupoid $\Sigma\tto M$ is a strict pf Lie groupoid with associated tower
\[
\ldots \to J^k\Sigma \to J^{k-1}\Sigma \to \ldots \to J^1\Sigma \to \Sigma \tto M.
\]
Since
\[
\Omega^\bullet (J^\infty\Sigma) = \bigcup\limits_{k} \Omega^\bullet(J^k\Sigma),
\]
there is a natural way to endow $J^\infty\Sigma$ wiith a de Rham/Cartan calculus. The infinite jet $J^\infty A$ of the Lie algebroid $A$ of $\Sigma$ can also be given a profinite smooth structure, with tower is 
\[
\dots \to J^{k+1}A \overset{\pr}{\rightarrow} J^{k}A \to \dots \to J^1A \overset{\pr}{\rightarrow} A \to M,
\]
which is just the tower of Lie algebroids of the jet groupoids $J^k\Sigma$; thus, $J^\infty A$ is the profinite Lie algebroid associated to the profinite Lie groupoid $J^\infty\Sigma$.

With the above concepts at hand, one can talk about connection forms $\omega\in \Omega^1(J^\infty\Sigma,t^*J^\infty A)$, and investigate their multiplicativity.
\begin{proposition}\label{prop:CartanInfJets}
    There is a canonical multiplicative flat connection $\omega^\infty$ on $J^\infty\Sigma$, uniquely defined by the property that, for a bisection $\bis^\infty:U\subset M \to J^\infty\Sigma$, there exists a bisection ${\bis}:U\to \Sigma$ such that
    \[
    \bis^\infty: x\mapsto j^\infty_x {\bis}
    \]
    if and only if the image of $\bis$ is tangent to $\omega^\infty$.

    Infinitesimally, there is a canonical infinitesimally multiplicative flat connection
    \[
    D^\infty:\mathfrak{X}(M)\times \Gamma(J^\infty A) \to \Gamma(J^\infty A)
    \]
    on $J^\infty A$ with the property that a section $\alpha^\infty\in \Gamma(J^\infty A)$ is flat if and only if there exists $\alpha\in \Gamma(A)$ such that
    \[
    \alpha^\infty: x\mapsto j^\infty_x\alpha.
    \]
    \end{proposition}
For the sake of clarifying the objects at hand, we provide a sketch of the proof.

Each finite jet groupoid $J^k\Sigma$ carries the canonical \textbf{Cartan form} $\omega^k\in \Omega^1(J^k\Sigma,J^{k-1}A)$ -- note the shift on the coefficients, see Remark~\ref{rmk:omega_k_is_not_a_connection} below -- defined by
\[
\omega^k_{j^k\bis}=d\pr -dj^{k-1}\bis \circ ds;
\]
equivalently, $\omega^k$ can be defined as the unique one-form such that, for a local bisection $\bis^k$ of $J^k\Sigma$,
\[
(\bis^k)^*\omega^k = 0 \text{ if and only if } \bis^k = j^k\bis
\]
for some local bisection $\bis$ of $\Sigma$ with the same domain.

\begin{remark}\label{rmk:omega_k_is_not_a_connection}
    The Cartan form $\omega^k$ takes values in $J^{k-1}A$, and is \emph{not} a connection on $J^k\Sigma$ in the sense of Lemma~\ref{lem:various_conn_pov}. All the statements below hold on the infinite jet and do not have a naive finite-order counterpart obtained by replacing $\infty$ with a poisitive integer $k\in \mathbb{N}$.
\end{remark}

Since
\[
\pr^*\omega^k = d\pr|_{A^k}\circ\omega^{k+1}
\]
the sequence of forms $(\omega^k)_{k\geq 1}$ canonically corresponds to a form $\omega^\infty$ on the profinite Lie groupoid $J^\infty\Sigma$ with values in $J^\infty A$, which inherits the property
\[
(\bis^\infty)^*\omega^\infty = 0 \text{ if and only if } \bis^\infty = j^\infty\bis.
\]

The profinite distribution $\ker(\omega^\infty)$ -- which canonically corresponds to the sequence $(\ker(\omega^k))_{k\geq 1}$ -- is complementary to the fibers of the source projection $s:J^\infty\Sigma \to M$ and involutive as a consequence of the following two properties, called the \textbf{prolongation conditions}:
\begin{itemize}
    \item $d\pr\left( \ker (\omega^{k+1}) \cap \ker (ds)\right) = 0$;
    \item $[d\pr\ker(\omega^{k+1}),d\pr\ker(\omega^{k+1})]\subset \ker(\omega^k)$.
\end{itemize}

Finally, the forms $\omega^k$'s are all multiplicative -- see~\cite[Proposition 6.1.14]{Sal13}; it follows that $\omega^\infty$ is multiplicative as well.

    The profinite connection $D^\infty$ is constructed out in a completely analogous way out of the canonical \textbf{Spencer operators} 
    \[
    D^k: \mathfrak{X}(M)\times \Gamma(J^k A) \to \Gamma(J^{k-1} A).
    \]
    on finite order jets, which are the infinitesimal counterpart of the forms $\omega^k$~\cite[\S 5.1]{Sal13},~\cite{CrSalStr2012}.

Notice that, in general, $(J^\infty A, D^\infty)$ and $(J^\infty\Sigma, \omega^\infty)$  are out of the scope of Propositions \ref{prop:flat_Cartan_action_global} and \ref{prop:flat_Cartan_action}, because the usual existence and uniqueness results for ODEs fail in the profinite setting.
Below, $\mathrm{Bis}_\mathrm{loc}(\Sigma)$ denotes the set of local bisections of the Lie groupoid $\Sigma$. Notice that $\mathrm{Bis}_\mathrm{loc}(\Sigma)$ is the arrow space of a groupoid
\[
\mathrm{Bis}_\mathrm{loc}(\Sigma) \tto \mathrm{Top}(X)
\]
over the topology of $X$, where the source and target of a bisection $\bis:U \to \Sigma$ are respectively $U$ and $\mathrm{Im}(t\circ \bis)$, and the multiplication is induced by the product of bisections.
\begin{definition}\label{def:Pseudogroup}
    A \textbf{pseudogroup} supported by $\Sigma$ is a subset
    \[
    \Gamma \subset \mathrm{Bis}_\mathrm{loc}(\Sigma)
    \]
    that is a subgroupoid and satisfies the following sheaf-like condition: if $\bis \in\mathrm{Bis}_\mathrm{loc}(\Sigma)$, then
    \[
    \bis \in \Gamma \text{ if and only if } \bis|_{U} \in \Gamma \text{ for all $U\in \mathrm{dom}({\bis})$.}
    \]
\end{definition}
Notice that, if $\Gamma$ is a pseudogroup supported by $\Sigma\tto M$, then for all $k$ the subset
\[
J^k\Gamma =\{j^k\bis_x:\ \bis \in \Gamma,\ x\in \mathrm{dom}(\bis)\}\in J^k\Sigma
\]
is in fact a subgroupoid.
\begin{definition}\label{def:LiePseudogroup}
    A \textbf{Lie pseudogroup} supported by $\Sigma$ is a pseudogroup $\Gamma$ supported by $\Sigma$ such that
    \begin{itemize}
        \item for all $k\geq 1$,
        \[
        J^k\Gamma \subset J^k\Sigma
        \] is a Lie subgroupoid;
        \item for all $k\geq 1$,
        \[
        \pr:J^{k+1}\Gamma \to J^k\Gamma
        \] is a surjective submersion;
        \item it holds
        \[
        J^{\infty}\Gamma \cong \lim\limits_{\leftarrow}J^k\Gamma.
        \]
    \end{itemize}
\end{definition}
Equivalently, the definition above can be formulated by saying that $J^\infty\Gamma$ is a profinite Lie subgroupoid of $J^\infty\Sigma$, see \cite{AccCr2023}. In particular, for a Lie pseudogroup $\Gamma$ supported by $\Sigma$, the infinite jet groupoid $J^\infty\Gamma$ carries a canonical profinite Lie groupoid structure, and the following proposition is almost tautological.

\begin{proposition}\label{prop:pseudogroup_restriction}
    Let $\Gamma$ be a Lie pseudogroup supported by $\Sigma$. The Cartan form $\omega^\infty$ on $J^\infty\Sigma$ restricts to a flat Cartan connection on $J^\infty\Gamma$.
\end{proposition}

We will keep the notation $\omega^\infty$ for the restriction of the Cartan form to $J^\infty\Gamma$. We will denote by $A^k_\Gamma$ the Lie algebroid of the $k$-jet groupoid $J^k\Gamma$ and by $A^\infty_\Gamma$ the profinite Lie algebroid of $J^\infty\Gamma$. Of course, $A^k_\Gamma\subset J^kA$ and $A^\infty_\Gamma\subset J^\infty\Gamma$ are (profinite) sub-algebroids. Notice that, while Proposition \ref{prop:flat_Cartan_action} does not hold in the profinite setting, formula \eqref{eq:torsion_bracket} applied to
$D^\infty$ still defines a fiberwise Lie bracket on $A^\infty_\Gamma$ \cite{AccCr2023}.

When $\Sigma = M\times M$ is the pair groupoid, local bisections correspond to locally defined diffeomorphisms of $M$, and a Lie pseudogroup $\Gamma$ supported by $\Sigma$ is a classical Lie pseudogroup of transformations of $M$. Then, the fiber at $x\in M$ of the algebroid $A^\infty_\Gamma$ can be realized as 
the space of \textbf{formal $\Gamma$-vector fields} at $x$ -- i.e. infinite jets of vector fields defined around $x$ whose flow is by elements of $\Gamma$. The fiberwise Lie bracket \ref{eq:torsion_bracket} on the fiber of $A^\infty_\Gamma$at $x\in M$ is then
\[
\{j^\infty_x X,j^\infty_x Y\} =j^\infty_x[X,Y]
\]
for every two $\Gamma$-vector fields $X$ and $Y$ defined around $x$.

The infinite jet groupoid $J^\infty\Gamma$ of a Lie pseudogroup $\Gamma$ is
thus a profinite Lie groupoid and, as such, comes equipped with the canonical
Cartan connection $\omega^\infty$ of
Proposition~\ref{prop:CartanInfJets}, which is flat.
This example has been studied extensively in \cite{AccCr2023}; there, the authors also investigated principal $(J^\infty\Gamma,\omega^\infty)$-bundles $(P^\infty,\theta^\infty)$, under the additional assumption that $\theta^\infty$ is flat (in the terminology of \cite{AccCr2023}, such bundles are called flat $(J^\infty\Gamma,\omega^\infty)$-structures). They are relevant as a model for (formal) pseudogroup structures on manifolds. We refer the reader to \cite{Catt2020}, \cite{AccCr2023} for more details; here, we limit ourselves to comment upon the characteristic classes constructed in \cite{AccCr2023}, and how they relate to the construction in the present paper. 

Since we are dealing with flat bundles, $\mc_{\theta^\infty}$ vanishes and, as in Proposition \ref{prop:anchored_form_connection}, the map
\[
\Theta^\infty:  TP \to TM\bowtie A_\Gamma^\infty
\]
is a map of Lie algebroids, thus inducing the map of dga's
\[
(\Theta^\infty)^*: (\Omega^\bullet(TM\bowtie A_\Gamma^\infty),d_\oplus) \to (\Omega^\bullet(P^\infty),d_\mathrm{dR}).
\]
\begin{remark}\label{rmk:secondary_classes}
It is worth stressing the analogy with the classical theory of secondary characteristic classes for flat principal group bundles: when a principal $G$-bundle $P$ over $X$ admits a flat connection $\theta$, its Chern-Weil map vanishes identically in positive polynomial degrees and one has instead the characteristic map
\[
\kappa^P_\mathrm{flat}: H^\bullet(\mathfrak{g},K) \to H^\bullet_\mathrm{dR}(X)
\]
defined on the Lie algebra cohomology of $\mathfrak{g}$ relative to a maximal compact subgroup $K$. The map simply evaluates a representative on the connection $\theta$, and, conceptually, it is obtained from the map of dga's
\[
(C^\bullet(\mathfrak{g}),d_\mathrm{CE}) \to (\Omega^\bullet(P),d_\mathrm{dR})
\]
after passing to $K$-basic cochains.


\end{remark}
When dealing with (non-transitive) Lie groupoids, there is no satisfactory analogue of maximal compact subgroups. Thus, it is not immediately clear how to restrict the map
\[
(\Theta^\infty)^*: (\Omega^\bullet(TM\bowtie A_\Gamma^\infty),d_\oplus) \to (\Omega^\bullet(P^\infty),d_\mathrm{dR}).
\]
to obtain a map with values in the deRham cohomology of $X$. In \cite{AccCr2023}, such a map is obtained by working at the level of the Lie groupoid $J^\infty\Gamma$, rather than its algebroid $A_\Gamma^\infty$. 

Recall that the \textbf{Bott-Shulman complex} of a Lie groupoid $\Sigma\tto M$ is the complex
\[
(\Omega^\bullet(N^\bullet \Sigma), d_\mathrm{BS})
\]
consisting of differential forms on the nerve $N^\bullet\Sigma$ of $\Sigma$, equipped with the differential $d_\mathrm{BS}$ obtained as the total differential in the double complex
\[
\begin{tikzcd}
    \Omega^k(N^{p+1}\Sigma) \arrow[r,"d_\mathrm{dR}"] & \Omega^{k+1}(N^{p+1}\Sigma)\\
    \Omega^k(N^p\Sigma)\arrow[r,"d_\mathrm{dR}"]\arrow[u,"\delta"] \arrow[r,"d_\mathrm{dR}"] & \Omega^{k+1}(N^p\Sigma)\arrow[u,"\delta"] 
\end{tikzcd}
\]
where $\delta=\sum\limits_{i}(-1)^id_i^*$ is the \textbf{simplicial differential} of $\Sigma$ -- that is, the maps $d_i$ are the face maps of the nerve. As shown in \cite{AccCr2023}, a flat Cartan connection $\omega$ on $\Sigma$ defines another complex
called the \textbf{Haefliger bicomplex},
\[
(\Omega^\bullet(N^\bullet \Sigma,\omega),d_\mathrm{Hae}) 
\]
consisting of forms that are horizontal with respect to (the connection on the nerve induced by) $\omega$; $d_\mathrm{Hae}$ is obtained by replacing the de Rham differential with the differential on horizontal form. For a Lie group with its Maurer-Cartan form, the Haefliger complex is precisely the complex of differentiable group cochains.

The Haefliger bicomplex has an infinitesimal counterpart; Proposition 5.15 in \cite{AccCr2023} identifies it with the Chevalley--Eilenberg complex of $TM\bowtie A$. Then, Proposition 6.5 in \cite{AccCr2023} constructs a Van Est type map
\[
H^\bullet_\mathrm{Hae}(\Sigma,\omega) \to H^\bullet(TM\bowtie A)
\]
from the cohomology of the Haefliger bicomplex into $H^\bullet(TM\bowtie A)$. In fact, when $\Sigma$ is transitive and flat, and $K\subset \Sigma_x$ is a compact Lie subgroup of the isotropy Lie group at some $x\in M$ such that $\Sigma_x/K$ is contractible, the above map factors through an isomorphism 
\[
H^\bullet_\mathrm{Hae}(\Sigma,\omega) \to H^\bullet(\mathfrak{a}_x, K)\quad \text{(Theorem $6.7$ in \cite{AccCr2023})}
\]
where $\mathfrak{a}_x$ is the isotropy Lie algebra of $TM\bowtie A$ at $x$.

The discussion above can be applied to $(J^\infty\Gamma,\omega^\infty)$ by using the machinery of profinite differential geometry \cite{AccCr2023}. Then, the van Est map of~\cite[\S 6.1]{AccCr2023} shows that, for any flat $(J^\infty\Gamma,\omega^\infty)$-bundle $(P^\infty,\theta^\infty)$, the connection $\theta^\infty$ induces a map
\[
\kappa^{P^\infty}_\mathrm{flat}: H^\bullet_\mathrm{Hae}(J^\infty\Gamma,\omega^\infty)\to H^\bullet_\mathrm{dR}(X)
\]
When the pseudogroup $\Gamma$ is transitive and there exists a compact subgroup $K\subset (J^\infty\Gamma)_x$ of the isotropy $(J^\infty\Gamma)_x$ at $x\in M$ with the property that $(J^\infty\Gamma)_x/K$ is contractible, then one ends up with the characteristic map 
\[
\kappa^{P^\infty}_\mathrm{flat}: H^\bullet(\mathfrak{a}_x, K)\to H^\bullet_\mathrm{dR}(X),
\]
recovering characteristic classes of $\Gamma$-foliations \cite{BottHaefliger}. 

\printbibliography

\end{document}